\documentclass[letterpaper,11pt]{article}
\usepackage{dsfont}
\usepackage[utf8]{inputenc}
\usepackage[margin=1in,footskip=0.25in]{geometry}
\usepackage{mystyle}

\begin{document}

	\title{On the Convergence of Interior-Point Methods for Bound-Constrained Nonlinear Optimization Problems with Noise}

\author{Shima Dezfulian\thanks{(\url{dezfulian@u.northwestern.edu}) Department of Industrial Engineering and Management Sciences, Northwestern University
		University} \and Andreas W\"achter \thanks{(\url{andreas.waechter@northwestern.edu})  Department of Industrial Engineering and Management Sciences, Northwestern University}}
	\date{}
	\maketitle
 
\begin{abstract}
We analyze the convergence properties of a modified barrier method for solving bound-constrained optimization problems where evaluations of the objective function and its derivatives are affected by  bounded and non-diminishing noise.  The only modification compared to a standard barrier method is a relaxation of the Armijo line-search condition.
We prove that the algorithm generates iterates at which the size of the barrier function gradient eventually falls below a threshold that converges to zero if the noise level converges to zero.
Based on this result, we propose a practical stopping test that does not require estimates of unknown problem parameters and identifies iterations in which the theoretical threshold is reached.
We also analyze the local convergence properties of the method when noisy second derivatives are used.  Under a strict-complementarity assumption, we show that iterates stay in a neighborhood around the optimal solution once it is entered.  The neighborhood is defined in a scaled norm that becomes narrower for variables with active bound constraints as the barrier parameter is decreased.  As a consequence, we show that active bound constraints can be identified despite noise.
Numerical results demonstrate the effectiveness of the stopping test and illustrate the active-set identification properties of the method.
\end{abstract}



\section{Introduction}\label{sec.intro}
Interior-point methods are some of the most powerful techniques for solving nonlinear optimization problems with general equality and inequality constraints. They enjoy convergence guarantees in the non-noisy setting under conditions often satisfied in practice. However, when applied to problems where the function or constraints can only be evaluated with noise, their performance can be erratic. It is therefore natural to ask if the interior-point framework can be adapted to the noisy setting or if it is intrinsically unsuitable in this case. 

In this paper, as a first step towards understanding the impact of noise in the interior-point framework, we focus on bound-constrained optimization problems
\begin{equation}\label{eq.orig_bnd_const_prob}
	\min_{x \in \RR^n}\ f(x) \ \ \st \ \ x \geq 0,
\end{equation}
where $f: \RR^n \to \RR$ is continuously differentiable. 
We consider the standard barrier approach, where the only modification is a natural relaxation of the Armijo line-search condition.

In this paper, we assume that noise is persistent and cannot be diminished, in other words, it is outside the control of the user. Therefore, the convergence to the exact solution cannot be guaranteed and we present instead convergence results to neighborhoods with small barrier function gradients.  
We also assume that an estimate of the noise level in the function is available. We argue that this is natural since knowledge of the noise level is necessary to determine if a problem is satisfactorily solved and is required to ensure that certain computations within the algorithm are not contaminated by noise.

Applications of nonlinear programming in which the function and/or constraints contain
noise abound. We refer the reader to the examples given in \cite{ng2014multifidelity} involving robust design in the presence of stochastic noise, and to \cite{more2011estimating} for examples involving (deterministic) computational error. Interest in solving noisy nonlinear problems has been receiving increasing attention \cite{berahas2019derivative, berahas2021global, berahas2021sequential, 
	curtis2021inexact,
	curtis2023sequential,  curtis2019stochastic, na2023inequality, oztoprak2023constrained, paquette2020stochastic, qiu2023sequential, sun2023trust, xie2020analysis}.

\subsection{Contributions}
Our global convergence results are based on \cite{berahas2019derivative} which analyzes the behavior of line-search methods for unconstrained optimization with relaxed Armijo conditions in the presence of bounded noise.  The analysis in \cite{berahas2019derivative} does not immediately apply in our setting because the required assumption that the objective function is bounded is violated by the $\log$-barrier term. In addition, the analysis in \cite{berahas2019derivative} focuses on the gradient descent method, while we permit more general choices of the Hessian matrix, including second-derivatives.
Our global convergence result is stated similarly to that in \cite{oztoprak2023constrained} and shows that eventually the stationary measure (the barrier objective gradients in our setting) must become small.  The thresholds in \cite{oztoprak2023constrained} are defined in terms of unknown problem-dependent quantities and cannot be used as a practical stopping test.
Overcoming this deficiency, we present a novel termination test that does not depend on unknown parameters and permits the user to recognize when the algorithm reaches the level of accuracy predicted by the global analysis.  The effectiveness of this test is confirmed with numerical experiments. 

Our local convergence analysis shows that locally the iterates converge at a quadratic-linear-type rate when noisy second-derivatives are used for a non-degenerate instance, similar to \cite{kelley2022newton}.  However, the analysis in \cite{kelley2022newton} is not immediately applicable because the Hessian of the barrier function becomes unbounded as variables approach their bounds.  We overcome this by utilizing a scaled norm, and we also prove the existence of a neighborhood around an optimal solution inside which the iterates are confined once it is entered.  
The particular scaling also allows us to argue that the method is able to identify which of the bound constraints are active at the solution as the barrier parameter goes to zero, despite the presence of noise.  We present numerical results that demonstrate the identification of the active set.

\subsection{Literature Review}
Nonlinear optimization problems in the presence of noise have received increasing interest in recent years. In particular, recent research has focused on designing noise-tolerant algorithms based on certain assumptions about noise and its properties. Many of these studies aim to modify existing algorithms to obtain convergence guarantees. 

Unconstrained problems with noisy function and gradient evaluations are studied in \cite{berahas2019derivative,  berahas2021global, blanchet2019convergence, cao2023first, curtis2019stochastic, paquette2020stochastic, scheinberg2022stochastic, shi2022noise, sun2023trust, xie2020analysis}. They can be grouped into two categories: those that consider stochastic noise \cite{berahas2021global, blanchet2019convergence, cao2023first, curtis2019stochastic, paquette2020stochastic, scheinberg2022stochastic} and those that consider bounded and non-diminishing noise \cite{berahas2019derivative, shi2022noise, sun2023trust, xie2020analysis}. Since our assumptions align with the second category we provide a brief overview of them here. In \cite{berahas2019derivative}, linear convergence of the gradient-descent algorithm with a relaxed Armijo line-search to a neighborhood of the solution for non-noisy problems is proved when the objective function is strongly convex. 
In \cite{sun2023trust}, a modification of the trust-region algorithm that relaxes the numerator and denominator of the actual-to-predicted-reduction ratio is introduced. This modification guarantees  convergence to a neighborhood of the non-noisy solution. In \cite{xie2020analysis}, a modification of  the BFGS algorithm with a specialized Armijo-Wolfe line-search is proposed and convergence guarantees to a neighbourhood of the solution are provided.

Noisy equality-constrained optimization problems are studied in \cite{berahas2021sequential, curtis2021inexact, oztoprak2023constrained}. In \cite{berahas2021sequential}, it is assumed that the evaluations of the objective function and its gradient are contaminated with a stochastic noise; however, equality constraint evaluations are deterministic. In this context, a modification of the line-search SQP algorithm is proposed that replaces the traditional line search with a step-size selection scheme based on Lipschitz constants or their approximations. The algorithm achieves global convergence, in expectation, to a local solution.  The work in \cite{curtis2021inexact} considers a similar problem setting as \cite{berahas2021sequential}; however, it allows inexact solutions for the subproblems. More relevant to our context, \cite{oztoprak2023constrained} considers the setting where the objective function, equality constraints, and all gradients are affected by a non-diminishing bounded noise. The authors propose a modification of the line-search SQP algorithm that relaxes the Armijo condition \cite{nocedal1999numerical}, similar to \cite{berahas2019derivative}. Convergence to a neighborhood of the solution is shown under the assumption that first-order steps are taken.

Noisy optimization methods for solving problems with both inequality  and equality constraints are studied in \cite{curtis2023sequential, na2023inequality, qiu2023sequential}, where noise is assumed to be stochastic. All of these works assume noisy objective function and deterministic constraint evaluations and propose a modification of the SQP algorithm.  Convergence in expectation \cite{curtis2023sequential} and almost sure convergence \cite{na2023inequality, qiu2023sequential} are established in this setting.

In this paper, we aim to design a noise-tolerant interior-point algorithm. In the deterministic setting, interior-point methods have been extensively studied and proven successful for solving very large-scale nonlinear optimization problems with equality and inequality constraints.
As one of the first references, \cite{yamashita1998globally} considers a primal-dual line-search interior-point algorithm and provides global convergence to a local solution. Trust-region interior-point methods are studied, for example, in  \cite{byrd2000trust, byrd1999interior, yamashita2005globally} and a convergence proof is provided in \cite{byrd2000trust, yamashita2005globally}. A line-search filter interior-point algorithms are studied in \cite{wachter2005line, wachter2006implementation} and convergence proof of such algorithms is provided in \cite{wachter2005line}. 

In a setting with stochastic noise, an interior-point algorithm for solving bound-constrained problems is proposed and analyzed \cite{curtis2023stochastic}. The algorithm proposed in \cite{curtis2023stochastic} deviates  more significantly from the standard interior-point framework than our work, as it does not employ the fraction-to-the-boundary rule; instead, it suggests a strategy that maintains the iterates within an inner neighborhood of the feasible region.

\subsection{Notation}\label{sec:notation}
The sets of real numbers, $n$-dimensional real vectors, and $n$-by-$m$-dimensional real matrices are denoted by $\RR$, $\RR^n$, and $\RR^{n \times m}$, respectively.  The sets of nonnegative and positive real numbers are denoted by $\RR_{\geq0}$ and $\RR_{>0}$, respectively. The set of nonnegative integers is denoted by $\NN := \{0,1,\dots\}$ and we define $[n] := \{1,\dots,n\}$ for any $n \in \NN \setminus \{0\}$.  The identity matrix is denoted as $I$, and the vector of all ones is denoted as $\mathbf{e}$, where in each case the size of
the object is determined by the context. We write $v_{i} \in \RR$ to denote the $i$th element of a vector $v \in \RR^n$ for all $i \in [n]$ and $v_k \in \RR^n$  to denote the value of vector $v \in \RR^n$ in iteration $k \in \NN$ of the algorithm. In addition, $v_{k, i} \in \RR$ denotes the $i$th element of vector $v_k \in \RR^n$ in iteration $k \in \NN$ of the algorithm, for all $i \in [n]$. For any vector $v \in \RR^n$ we write $V \in \RR^{n \times n}$ to denote $\text{diag}(v)$. The $\ell_2$-norm of any vector $v \in \RR^n$ is denoted by $\| v\|$ and the induced $2$-norm of any matrix $M \in \RR^{m \times n}$ is denoted by $\| M \|$. The minimum and maximum  eigenvalues of a matrix $M \in \RR^{n \times n}$ are respectively denoted by $\sigma_{\min}(M)$ and $\sigma_{\max}(M)$. 
Given two symmetric matrices $M_1$ and $M_2$ in $\RR^{n \times n}$, we write $M_1 \succ M_2$ and $M_1  \succeq M_2$ to denote $M_1 - M_2$ is positive definite and positive semi-definite, respectively. For a positive definite matrix $M \succ 0$ and a vector $v$, we use $\| v \|_M $ to denote $\sqrt{v^T M v}$.
If $v_1\in\RR^{n_1}$ and $v_2\in\RR^{n_2}$, we write $(v_1,v_2)$ to denote $(v_1^T,v_2^T)^T$. 
A fraction with a positive, finite numerator and a zero denominator evaluates to $\infty$.

\subsection{Organization}
The formulation of the barrier subproblem and an algorithm for solving instances with noisy function evaluations are presented in Section~\ref{sec.barrier_problem}. The global convergence analysis for solving the barrier subproblem is presented in Section~\ref{sec.global_conv}, and local convergence is analyzed in Section~\ref{sec.local_conv}.  Numerical experiments are discussed in Section~\ref{sec.numerical_exp} and concluding remarks are given in Section~\ref{sec.conclusion}.

\section{An interior point method with noisy function evaluations} 
\label{sec.barrier_problem}
The analysis of interior-point methods for problems with general equality and inequality constraints is rather complex due (in part) to the complicated mechanisms of globalization techniques based on penalty functions or filters.  Thus, in this paper, we consider
the simpler setting of bound-constrained optimization \eqref{eq.orig_bnd_const_prob}, where the main design questions can be studied more directly and the numerical results are easier to analyze. In this case, the globalization mechanism reduces to a line search for the barrier objective function.

\subsection{Barrier methods}

In order to solve problem~\eqref{eq.orig_bnd_const_prob} using an interior-point or barrier algorithm, we define the \emph{barrier subproblem}  
\begin{align}
	\label{eq.barrier_bnd_const}  
	\min_{x \in \RR^n} \ \  \varphi^{\mu}(x) := \,  f(x) - \mu \sum_{i = 1}^n \log(x_i),
\end{align}
for a given barrier parameter $\mu \in \RR_{>0}$.  The main idea behind interior-point methods is to solve, often only to a certain tolerance,  a sequence of barrier problems \eqref{eq.barrier_bnd_const} for a decreasing sequence of barrier parameters that converges to zero.

Before introducing our algorithm that is able to handle the situation in which only noisy estimates of $f$ and its derivatives are available, let us first recall some basic facts about barrier methods.
We start with the primal-dual formulation of the necessary first-order optimality conditions for the barrier problem \eqref{eq.barrier_bnd_const}.
\begin{theorem}[\cite{nocedal1999numerical}]
	\label{theorem:ness_optcond}
	Suppose $f$ is continuously differentiable and let $\mu \in \RR_{>0}$. 
	If $x_*^\mu\in \RR_{>0}^n$ is a local minimizer of \eqref{eq.barrier_bnd_const}, then there exists a vector $z^\mu_* \in \RR_{>0}^n$ so that
	\begin{align}
		\label{eq.kktmu}
		\begin{split}
			\nabla f(x) - z =&\; 0, \\
			Xz =&\; \mu \mathbf{e},
		\end{split}
	\end{align}
	holds with $x=x_*^\mu$ and $z=z_*^\mu$.
\end{theorem}
From this, one concludes $z_*^\mu = \mu( X_*^{\mu})^{-1}\mathbf{e}$.
The following assumption and theorem states sufficient second-order optimality conditions for the \emph{original} problem \eqref{eq.orig_bnd_const_prob}, as well as the local existence of a central path. 
\begin{assumption}\label{as:suff_optcond}
	$f$ is twice continuously differentiable.
	Let $x_*^0\in \RR_{\geq0}^n$ and $z_*^0\in \RR_{\geq0}^n$ be such that (i) \eqref{eq.kktmu} holds for $\mu=0$, $x=x_*^0$, and $z=z_*^0$; (ii) strict complementarity holds, i.e., $x_*^0+z_*^0>0$ for each $i\in[n]$; and (iii) the reduced Hessian $H_{*, \III\III}^0$---with $H_*^0=\na^2 f\left(x_*^0\right)$ and $\III = \{i\in[n]: x_{*, i}^0>0\}$---is positive definite, i.e., $\sigma_l^H=\sigma_{\min}(H_{*, \III\III}^0) >0$.
\end{assumption}

For later reference, we also define $\AAA=[n]\setminus\III$.

\begin{theorem}[\cite{nocedal1999numerical}]
	\label{theorem:suff_optcond}
	Suppose Assumption~\ref{as:suff_optcond} holds. Then $x_*^0$ is a strict local minimizer of \eqref{eq.orig_bnd_const_prob} .
\end{theorem}

\begin{theorem}[\cite{wright2002properties}]
	\label{thm:central_path}
	Suppose Assumption~\ref{as:suff_optcond} holds. Then there exists $\hat\mu \in \RR_{>0}$ and a continuously differentiable primal-dual path $\phi(\mu):[0,\hat\mu]\to\RR^{n}\times\RR^{n}$ so that $\phi(\mu)=(x_*^\mu,z_*^\mu)$ where $x_*^\mu$ and $z_*^\mu$ satisfy \eqref{eq.kktmu}  with $x=x_*^\mu$ and $z=z_*^\mu$ for all $\mu\in[0,\hat \mu]$.
	Furthermore, there exists $\theta_\mu\in(0,1)$ such that for all $i\in\III$ and $\mu\in[0,\hat \mu]$
	\begin{equation*}
		\theta_\mu\, \mu \leq |x_{*, i}^{\mu} - x_{*, i}^0 | \leq \frac{1}{\theta_\mu} \mu,
	\end{equation*}
	where $x_{*, i}^{\mu}$ denotes the $i$th element of $x_*^{\mu}$.
\end{theorem}


\subsection{A line-search method for solving the barrier problem with noisy function evaluations}

In this paper we assume that we have only access to noisy evaluations of $f$, its gradient $g:= \nabla f$, and potentially its Hessian $H := \nabla^2 f$. The noisy evaluations of $f$ and $g$ are denoted by $\tilde f$ and $\tilde g$. We further assume that the evaluation noise is bounded and non-diminishing. 

\begin{assumption}
	\label{assumption.bnd_err} 
	Let $\tilde f$ and  $\tilde g$ denote noisy evaluations of $f$ and $g$, respectively. There exist constants $\eps_f , \eps_g \in \RR_{\geq 0} $ such that for all $x \in \RR^n$
	\begin{align}
		&\tilde f(x) := \,   f(x) + \vareps_f(x), \ \text{with} \quad |\tilde f(x) - f(x) | = \, |\vareps_f(x)| \leq \eps_f \label{eq.f_err}, \\
		&\tilde g(x) := \,  g(x) + \vareps_g(x), \ \text{with} \quad  \| \tilde g(x) - g(x) \| = \, \| \vareps_g(x) \| \leq \eps_g.       \label{eq.g_err}
	\end{align}
	Here $\vareps_f(x) \in \RR$ denotes the  error in the evaluation of $f(x)$ and $\vareps_g(x) \in \RR^n$  the component-wise errors in the evaluation of $g(x)$. 
\end{assumption}

Even though we did not write it this way, the error terms $\vareps_f(x)$ and $\vareps_g(x)$ may not be deterministic, i.e., the oracle providing $\tilde f(x)$ and $\tilde g(x)$ may result in different values even if $x$ is unchanged.  We chose this notation to simplify the presentation.

{In many practical applications, noise is bounded. Consider 
	an engineering system whose performance we want to optimize, but that contains uncertainty in some of its physical parameters. By simulating it, we obtain an optimization problem with noisy evaluations, and if the decision variables are 
	bounded, we can assume that the system response is always bounded \cite{ng2014multifidelity}.
%
A similar situation may arise if $\tilde f(x)$ is obtained by a complicated numerical procedure that results in a bounded round-off error.

In this section, we assume that the barrier parameter $\mu \in \RR_{> 0}$ is fixed. As we discuss the algorithm for solving problem~\eqref{eq.barrier_bnd_const} and provide its convergence analysis, we use the superscript $\mu$ to emphasize the dependence on the barrier parameter. 
In the standard barrier method, subproblem~\eqref{eq.barrier_bnd_const} is solved by a {Newton-like} method, starting from a feasible point $x_0^\mu$.   We let $\hat H_k^\mu$ denote a real symmetric matrix (typically attempting to approximate $\na^2 f(\xk)$)
and define
\begin{align}
	\label{eq.hat_G_k}
	\hat G_k^{\mu} := \,  \hat H_k^{\mu}  + \mu (\Xk)^{-2},
\end{align}
as an approximation of the Hessian of the barrier function, $\nabla^2\varphi^{\mu}(\xk)$.  For our approach, it is crucial that the values and derivatives of the $\log$-barrier terms are not approximated since they can easily be computed exactly.

We can now outline the barrier algorithm.
In iteration $k$ with an iterate $x_k$, the search direction $d_k$  is calculated as 
\begin{equation}
	\label{eq.direction_bnd}
	d_k 
	= \, - (\hat G_k^{\mu})^{-1} \nabla \tilde \varphi^{\mu}(\xk),
\end{equation}
where, by slight abuse of notation, we define $\nabla \tilde \varphi^{\mu}(x) = \tilde g(x) - \mu X^{-1}\mathbf{e}$. The next iterate is computed as, $\xkk=\xk+\alpha_kd_k$ for some step size $\alpha_k\in(0,1]$.

In order to ensure the strict feasibility of the iterates, i.e., $\xk > 0$,  the step size $\alpha_k$ is chosen to satisfy the \emph{fraction-to-the-boundary rule}, i.e., $\alpha_k \in (0, \alpha_k^{\max}]$, where 
\begin{equation}
	\label{eq.frac_to_the_bnd}
	\alpha_k^{\max} := \, \max \left\{ \alpha \in (0, 1]:  \xk + \alpha d_k \geq (1 - \tau) \xk \right\},
\end{equation}
for some fixed parameter $\tau \in (0, 1)$. 
Furthermore, to force a decrease in the barrier function, the step size $\alpha_k$ is required to satisfy the \emph{relaxed} Armijo condition \cite{berahas2019derivative} with parameter $\nu \in (0, \tfrac{1}{2})$ and relaxation $\eps_R$,  
\begin{equation}
	\label{eq.relaxed_armijo_bnd}
	\tilde \varphi^{\mu}(\xk + \alpha_k d_k) \leq \tilde \varphi^{\mu}(\xk) + \nu \alpha_k \nabla \tilde \varphi^{\mu}(\xk)^T d_k + \eps_R ,
\end{equation}
where $\eps_R > 2 \eps_f$ and $\tilde \varphi^{\mu}(x) = \tilde f(x) - \mu \sum_{i = 1}^n \log(x_i)$. Note that the line search cannot fail due to the definition of $\eps_R$ since \eqref{eq.relaxed_armijo_bnd} is satisfied for all sufficiently small $\alpha_k$ by the continuity of $\varphi^{\mu}$. The method is summarized in Algorithm~\ref{alg.bnd_const_fixed_mu}. Its convergence proof  is provided in the next section.

\begin{algorithm}[htp!]
	\caption{: Relaxed Armijo line-search for solving barrier subproblem \eqref{eq.barrier_bnd_const}  }
	\label{alg.bnd_const_fixed_mu}
	\begin{algorithmic}[1]
		\Require $\nu \in (0, \tfrac{1}{2})$; $\eps_R > 2 \eps_f$; $\tau \in (0, 1)$; $\mu \in \RR_{> 0}$; initial iterate $x_0^\mu \in \RR^n$.
		\For{$k = 0, 1, \dots$}
		\State Choose $\hat G^\mu_k$ by \eqref{eq.hat_G_k} and compute $d_k$ by \eqref{eq.direction_bnd}. \label{st.compute_dir}
		\State Set $\alpha_k \gets \alpha_k^{\max}$ with $\alpha_k^{\max}$ from \eqref{eq.frac_to_the_bnd}.
		\While {$\alpha_k$ does not satisfy   \eqref{eq.relaxed_armijo_bnd} 
		}\label{s:armijo_begin}
		\State $\alpha_k \gets \tfrac{1}{2} \alpha_k$.
		\EndWhile\label{s:armijo_end}
		\State Set $\xkk \gets \xk + \alpha_k d_k$.
		\EndFor
	\end{algorithmic}
\end{algorithm}

To ensure that the step $d_k$ in \eqref{eq.direction_bnd} can be computed, we make the following assumption about $\hat G_k^{\mu}$.
\begin{assumption}
	\label{assumption.quasi_Newton_Hess_singular_val}
	$\hat G_k^{\mu}$ is uniformly positive definite, i.e.,  there exists $\hat\sigma_l^{G} \in \RR_{> 0}$ such that for all $k\in\NN$
	\begin{equation}\label{eq:sigma_l_G}
		\sigma_{\min}(\hat G_k^{\mu})\geq \hat\sigma_l^{G}.
	\end{equation}
\end{assumption} 

This assumption is not very restrictive.  In Lemma~\ref{lemma.G_unif_posdef_iter} we show that \eqref{eq:sigma_l_G} holds if the sufficient second-order optimality conditions (Assumption~\ref{as:suff_optcond}) are satisfied, $\mu$ is sufficiently small, and $\hat H^\mu_k$ is a sufficiently accurate approximation of $\na^2 f(\xk)$.

\section{Global convergence}
\label{sec.global_conv}
In this section, we examine the convergence of Algorithm~\ref{alg.bnd_const_fixed_mu} for a fixed value of $\mu$.  Our work extends the results in \cite{berahas2019derivative} in several ways by
(i) permitting an unbounded barrier objective function; (ii) considering second-order (Newton-like) steps instead of gradient steps; and (iii) allowing nonconvex instead of only strongly convex functions. 
Furthermore, Section~\ref{sec.stopping_test} introduces a practical termination test that could also be applied to the algorithm in \cite{berahas2019derivative}.

It is well known that primal-dual interior-point methods (e.g., Algorithm~\ref{alg.bnd_const_update_barrier} in \appE) perform much better in practice than the primal method described above.  The primal-dual method maintains iterates $z^\mu_k$ approximating the bound multipliers and calculates the search directions from \eqref{eq.direction_bnd} with 
\begin{equation}\label{eq:G_primaldual}
	\hat G_k^\mu=\hat H_k^\mu + Z_k^\mu (X_k^\mu)^{-1},
\end{equation}
instead of \eqref{eq.hat_G_k}.
For simplicity, we analyze in this section the primal method using \eqref{eq.hat_G_k}, but the results hold also for the primal-dual method if a safeguard is implemented like in \cite{wachter2006implementation} that ensures that $z^\mu_k$ has positive entries and does not deviate arbitrarily far from $\mu (X_k^\mu)^{-1}\mathbf e$ (see \eqref{eq.dual_var_proj} in \appE.)
The only difference is that the threshold $\delta_x$ in Lemma~\ref{lemma.lb_on_x_bnd_const} might be smaller than in the primal case.

To simplify the notation, we let $\varphi_k^{\mu}=\varphi^{\mu}(\xk)$ and $\na\varphi_k^\mu=\na\varphi^\mu(\xk)$ and use similar simplifications for the noisy quantities.

\subsection{Global convergence analysis}
In addition to Assumptions~\ref{assumption.bnd_err} and \ref{assumption.quasi_Newton_Hess_singular_val}, we 
require the following assumptions. 

\begin{assumption}
	\label{assumption.bnd_const_grad_lipschitz}
	There exists an open set $\mathcal X \subset \RR^n$ such that $\{\xk+\alpha d_k : \alpha\in[0, 1]\}\subset\mathcal X$ for all $k$, 
	$f$ is differentiable and bounded below on $\cal X$ and the gradient $\nabla f$ is bounded and Lipschitz continuous with constant $L_g \in \RR_{>0}$ over $\XXX$, i.e., for all $x, y \in \cal X$ one has
	\[
	\| g(x) - g(y) \| \leq L_g \|x - y\|. \]
\end{assumption} 
\begin{assumption}
	\label{as:Hbounded}
	The Hessian approximation $\hat H_k$ is uniformly bounded.
\end{assumption}

In the following, we prove some preliminary results before stating the main theorem. 
In addition, for the remainder of this section, we assume that Assumptions~\ref{assumption.bnd_err}, \ref{assumption.quasi_Newton_Hess_singular_val}, \ref{assumption.bnd_const_grad_lipschitz}, and \ref{as:Hbounded} hold for all the lemmas and theorems stated and proved herein.

\begin{lemma}
	\label{lemma:barrier_bounded}
	Let $\xk\in\RR^n_{>0}$, $d_k\in\RR^n$, $\tau$ be the fraction-to-the-boundary parameter  in \eqref{eq.frac_to_the_bnd}, and suppose $\alpha \leq \alpha_k^{\max}$. Then 
	\begin{align*}
		-&  \left( \sum_{i=1}^n \log(x_{k, i}^\mu + \alpha d_{k, i}) - \sum_{i=1}^n \log(x_{k, i}^\mu) - \alpha \sum_{i=1}^n \tfrac{d_{k, i}}{x_{k, i}^\mu} \right) \leq \alpha^2 \tfrac{1}{1-\tau} \| (\Xk)^{-1} d_k \|^2
	\end{align*}
\end{lemma}
\begin{proof}
	The proof is similar to that of Lemma~4 in \cite{byrd2000trust}. 
\end{proof}

\begin{lemma}
	\label{lemma.lb_on_x_bnd_const}
	There exists a constant $\delta_x \in \RR_{>0} $ such that $\xk \geq \delta_x \mathbf{e}$ for all $k$.
\end{lemma}
\begin{proof}
	This claim was proven as Theorem~3 in \cite{wachter2005line}.
\end{proof}

\begin{lemma}
	\label{lemma.bnd_d}
	There exists $\Delta_d \in \RR_{>0}$  such that $\|d_k\|\leq \Delta_d$ for all $k$.
\end{lemma}
\begin{proof} By Assumptions~\ref{assumption.quasi_Newton_Hess_singular_val}, together with Lemma~\ref{lemma.lb_on_x_bnd_const} one finds
	$
	\| d_k \|
	\leq \| (\hat G_k^{\mu})^{-1} \| \| \nabla \tilde \varphi^{\mu}(\xk) \|
	\leq \tfrac{1}{\hat \sigma_l^{G} } \| \tilde g(\xk) - \mu (\Xk)^{-1}\mathbf{e} \|
	\leq  \tfrac{1}{\hat \sigma_l^{G}}\left(\Delta_g + \vareps_g + \tfrac{\mu}{\delta_x}\right)
	$, where the upper bound $\Delta_g$ on $\|g(x_k^\mu)\|$ exists by Assumption~\ref{assumption.bnd_const_grad_lipschitz}.
\end{proof}

\begin{lemma}
	\label{lemma.non-noisy-bnd-iterations}
	Let us define
	\begin{align}
		\label{eq.bar_bar_alpha}
		\bar \alpha := \, \min \left\{1,\tfrac{\tau\delta_x}{\Delta_d}, \tfrac{\hat \sigma_l^{G}}{L_g + \tfrac{2 \mu}{(1 - \tau) \delta_x^2}} \right\},
	\end{align}
	then for any $k$ that $\bar\alpha\leq\alpha_k^{\max}$ and for any $\alpha \in (0, \bar \alpha)$
	\begin{align*}
		\varphi^{\mu}(\xk + \alpha d_k) - \varphi^{\mu}(\xk) \leq&
		-\tfrac{\alpha}{2} \| \nabla \varphi^{\mu}(\xk) \|_{(\hat G_k^{\mu})^{-1}}^2 + \tfrac{\alpha}{2} \| \vareps_g(\xk) \|_{(\hat G_k^{\mu})^{-1}}^2.
	\end{align*}
\end{lemma}

\begin{proof}
	Consider a fixed iteration $k$.
	Then, from \eqref{eq.frac_to_the_bnd}, $\alpha_k^{\max}=1$ or $\alpha_k^{\max} = \tfrac{-\tau x_{k,i}^\mu}{d_{k,i}}$ with $d_{k,i}<0$ for some $i$.  Using Lemmas~\ref{lemma.lb_on_x_bnd_const} and \ref{lemma.bnd_d}, we conclude that $\bar\alpha\leq\alpha_k^{\max}$.
	
	By Taylor's theorem, $\alpha \leq \bar\alpha\leq\alpha_k^{\max}$, and Lemma~\ref{lemma:barrier_bounded}, one finds
	\begin{align}
		\label{eq.dummy_taylor}
		\begin{split}
			\varphi^{\mu}(\xk + \alpha d_k) - \varphi^{\mu}(\xk) =&\; f(\xk + \alpha d_k ) - \mu \sum_{i=1}^n \log(x_{k, i}^\mu + \alpha d_{k, i})\\
			&- f(\xk) + \mu \sum_{i=1}^n \log(x_{k, i}^\mu) \\
			\leq&\; \alpha \nabla \varphi^{\mu}(\xk)^T d_k + \tfrac{\alpha^2}{2}  d_k^T \breve H_k  d_k + \tfrac{ \mu}{1 - \tau}  \alpha^2 \| (\Xk)^{-1} d_k \|^2 \\
			=&\; \alpha \nabla \varphi^{\mu}(\xk)^T d_k + \tfrac{\alpha^2}{2}  d_k^T \breve G_k^{\mu} d_k
		\end{split}
	\end{align}
	where $\breve H_k^\mu := \nabla^2 f(\xk +  \breve \alpha d_k)$ for some $ \breve \alpha \in (0, 1)$, 
	and $\breve G_k^{\mu} := \breve H_k^\mu  + \tfrac{2 \mu}{1 - \tau} (\Xk)^{-2}$.
	By substituting $d_k$ from \eqref{eq.direction_bnd} one has 
	\begin{equation}
		\label{eq:bound_constrained_temp1}
		\begin{split}
			&\; \varphi^{\mu}(\xk + \alpha d_k) - \varphi^{\mu}(\xk) \\
			\leq& \;  - \alpha \nabla \varphi^{\mu}(\xk)^T (\hat G_k^{\mu})^{-1}(\nabla \varphi^{\mu}(\xk) + \vareps_g(\xk)) \\
			& \; + \tfrac{\alpha^2 }{2}  (\nabla\varphi^{\mu}(\xk) + \vareps_g(\xk))^T (\hat G_k^{\mu})^{-1} \breve G_k^{\mu} (\hat G_k^{\mu})^{-1} (\nabla\varphi^{\mu}(\xk) + \vareps_g(\xk)) \\
			=& \;  - \tfrac{\alpha}{2} \nabla \varphi^{\mu}(\xk)^T (\hat G_k^{\mu})^{-1} \left(2\hat G_k^{\mu} -  \alpha \breve G_k^{\mu} \right) (\hat G_k^{\mu})^{-1} \nabla \varphi^{\mu}(\xk) \\
			&\; +\tfrac{\alpha^2}{2}  \vareps_g(\xk)^T (\hat G_k^{\mu})^{-1} \breve G_k^{\mu} (\hat G_k^{\mu})^{-1} \vareps_g(\xk) \\
			&\; -\alpha \nabla \varphi^{\mu}(\xk)^T (\hat G_k^{\mu})^{-1} \left(\hat G_k^{\mu} - \alpha \breve G_k^{\mu} \right) (\hat G_k^{\mu})^{-1} \vareps_g(\xk).
		\end{split}
	\end{equation}
	By Assumptions~\ref{assumption.bnd_const_grad_lipschitz} and \ref{assumption.quasi_Newton_Hess_singular_val} and Lemma~\ref{lemma.lb_on_x_bnd_const} we have 
	\begin{align*}
		\hat G_k^{\mu}- \alpha \breve G_k^{\mu} \succeq \, \left( \hat \sigma_l^{G} - \alpha \left(L_g + \tfrac{2\mu}{(1 - \tau) \delta_x^2} \right) \right) I,
	\end{align*}
	which implies that for any $\alpha \in (0, \bar \alpha)$, $\hat G_k^{\mu}- \alpha \breve G_k^{\mu} \succ 0$.
	Let $R_k^T R_k$ be the Cholesky factorization of $\hat G_k^{\mu} - \alpha \breve G_k^{\mu}$ for such an $\alpha$.
	Using the fact that $- v^T w \leq \tfrac{\|v\|^2}{2} + \tfrac{\|w\|^2}{2} $ for arbitrary vectors $v, w \in \RR^n$, we have
	\begin{equation*} 
		\begin{split}
			& -\alpha \nabla \varphi^{\mu}(\xk)^T (\hat G_k^{\mu})^{-1} \left(\hat G_k^{\mu} - \alpha \breve G_k^{\mu} \right) (\hat G_k^{\mu})^{-1} \vareps_g(\xk)\\
			\leq &\; \tfrac{\alpha}{2} \left(\|R_k (\hat G_k^{\mu})^{-1} \nabla \varphi^{\mu}(\xk)\|^2 + \|R_k (\hat G_k^{\mu})^{-1} \vareps_g(\xk)\|^2   \right) \\
			\leq &\; \tfrac{\alpha}{2} \nabla \varphi^{\mu}(\xk)^T (\hat G_k^{\mu})^{-1} R_k^T R_k (\hat G_k^{\mu})^{-1} \nabla \varphi^{\mu}(\xk) + \tfrac{\alpha}{2} \vareps_g(\xk)^T (\hat G_k^{\mu})^{-1} R_k^T R_k (\hat G_k^{\mu})^{-1} \vareps_g(\xk) .
		\end{split}
	\end{equation*}
	Combining this with \eqref{eq:bound_constrained_temp1} we obtain
	\begin{equation*}
		\begin{split}
			& \varphi^{\mu}(\xk + \alpha d_k) - \varphi^{\mu}(\xk) \\
			\leq&\; -\tfrac{\alpha}{2} \nabla \varphi^{\mu}(\xk)^T (\hat G_k^{\mu})^{-1} \left( 2\hat G_k^{\mu}- {\alpha}  \breve G_k^{\mu} - R_k^T R_k\right) (\hat G_k^{\mu})^{-1} \nabla \varphi^{\mu}(\xk) \\
			& + \tfrac{\alpha}{2}  \vareps_g(\xk)^T (\hat G_k^{\mu})^{-1} \left(\alpha \breve G_k^{\mu} + R_k^T R_k\right) (\hat G_k^{\mu})^{-1} \vareps_g(\xk)  \\
			=&\;
			- \tfrac{\alpha}{2} \nabla \varphi^{\mu}(\xk)^T (\hat G_k^{\mu})^{-1}  \nabla \varphi^{\mu}(\xk) 
			+ \tfrac{\alpha}{2} \vareps_g(\xk)^T (\hat G_k^{\mu})^{-1}  \vareps_g(\xk)  \\
			=&\; - \tfrac{\alpha}{2} \|\nabla \varphi^{\mu}(\xk) \|_{(\hat G_k^{\mu})^{-1}}^2
			+ \tfrac{\alpha}{2}  \|\vareps_g(\xk) \|_{(\hat G_k^{\mu})^{-1}}^2,
		\end{split}
	\end{equation*} 
	as desired. 
\end{proof}

For the next results, we define the noise-to-signal ratio $\beta_k$  \cite{berahas2019derivative} as
\begin{align}
	\label{eq.kappa_mu_beta}
	\beta_k := \, \tfrac{\| \vareps_g(\xk) \|_{(\hat G_k^{\mu})^{-1}}}{\|\nabla \varphi^{\mu}(\xk) \|_{(\hat G_k^{\mu})^{-1}}},
\end{align}
and observe that by Assumption~\ref{assumption.bnd_err} and triangle inequality, one finds
\begin{align}
	\label{eq.bnd_const_noisy_grad_bounds}
	(1 - \beta_k) \| \nabla \varphi^{\mu}(\xk) \|_{(\hat G_k^{\mu})^{-1}} \leq \| \nabla \tilde \varphi^{\mu}(\xk) \|_{(\hat G_k^{\mu})^{-1}} \leq (1 + \beta_k) \| \nabla \varphi^{\mu}(\xk) \|_{(\hat G_k^{\mu})^{-1}}.
\end{align}

		\begin{lemma}
			\label{lemma.bar_alpha}
			If $\beta_k \leq \tfrac{1 - 2 \nu}{1 + 2 \nu}$ then the relaxed Armijo condition \eqref{eq.relaxed_armijo_bnd} holds for any $\alpha \in (0, {\bar \alpha})$ where $\bar \alpha$ is defined in \eqref{eq.bar_bar_alpha}.
		\end{lemma}
		\begin{proof}
			By Lemma~\ref{lemma.non-noisy-bnd-iterations} for any $\alpha \in (0, {\bar \alpha})$
			\begin{align*}
				\varphi^{\mu}(\xk + \alpha d_k) - \varphi^{\mu}(\xk) \leq&
				-\tfrac{\alpha}{2} \| \nabla \varphi^{\mu}(\xk) \|_{(\hat G_k^{\mu})^{-1}}^2 + \tfrac{\alpha}{2} \| \vareps_g(\xk) \|_{(\hat G_k^{\mu})^{-1}}^2.
			\end{align*}
			Hence, by Assumption~\ref{assumption.bnd_err}, \eqref{eq.kappa_mu_beta}, and \eqref{eq.bnd_const_noisy_grad_bounds},  one finds that for such an $\alpha$
			\begin{align}
				\label{eq.bnd_const_noisy_phi_ineq}
				\begin{split}
					\tilde \varphi^{\mu}(\xk + \alpha d_k) - \tilde \varphi^{\mu}(\xk) \leq &\,  -\tfrac{\alpha}{2}  \| \nabla \varphi^{\mu}(\xk) \|_{(\hat G_k^{\mu})^{-1}}^2 +  \tfrac{\alpha}{2} \|\vareps_g(\xk)\|_{(\hat G_k^{\mu})^{-1}}^2 + 2 \eps_f \\
					\leq &\, -\tfrac{\alpha}{2}  (1 -\beta_k^2) \| \nabla \varphi^{\mu}(\xk) \|_{(\hat G_k^{\mu})^{-1}}^2 + 2 \eps_f\\
					<& \,  -\tfrac{\alpha}{2}\tfrac{1 - \beta_k }{1 + \beta_k} \| \nabla \tilde \varphi^{\mu}(\xk) \|_{(\hat G_k^{\mu})^{-1}}^2   + \eps_R,
				\end{split}
			\end{align}
			where we used {$\eps_R > 2 \eps_f$} and \eqref{eq.bnd_const_noisy_grad_bounds}  to obtain the last inequality. 
			If $\beta_k \leq \tfrac{1 - 2 \nu}{1 + 2 \nu}$, then $ -\tfrac{1 }{2 } \tfrac{(1 -  \beta_k) }{(1 + \beta_k)} \leq -\nu$ and one finds 
			\begin{align*}
				\tilde \varphi^{\mu}(\xk + \alpha d_k) - \tilde \varphi^{\mu}(\xk)  
				<&    -\nu  \alpha  \| \nabla \tilde \varphi^{\mu}(\xk) \|_{(\hat G_k^{\mu})^{-1}}^2   + \eps_R,
			\end{align*}  
			which together with \eqref{eq.hat_G_k} implies \eqref{eq.relaxed_armijo_bnd}.
		\end{proof}
		
		We can now state the global convergence result for Algorithm~\ref{alg.bnd_const_fixed_mu}. 
		\begin{theorem}
			\label{thm.neigh}
			Let $D$ be a symmetric positive-definite matrix  and $ \gamma\in(0,1)$.
			Consider the set 
			$\CCC(D, \nu,\gamma)$ defined as 
			\begin{equation}\label{eq:defCCC}
				\CCC(D,\nu,\gamma) := \left \{x\in\XXX: \|\nabla \tilde \varphi^{\mu}(x) \|_{D} \leq \max \left\{ \sqrt{\tfrac{\hat\sigma_u^D}{\hat\sigma_l^G }} \left( \tfrac{1 + 2 \nu}{1 - 2 \nu} + 1\right) \eps_g, \sqrt{ \hat \sigma_u^D}\sqrt{\tfrac{4 \eps_f + 2 \eps_R}{{\gamma\bar \alpha} \nu}}\right\} \right\}
			\end{equation}
			where  $\hat \sigma_u^D = \sup_{k} \left\|D^{1/2} \hat G_k^{\mu} D^{1/2}  \right\|$
			and $\bar \alpha$ is defined in \eqref{eq.bar_bar_alpha}.
			Then the iterates $\{\xk\}$ generated by Algorithm~\ref{alg.bnd_const_fixed_mu}  enter $\CCC(D, \nu,\gamma)$ infinitely many times.
		\end{theorem}
		\begin{proof}
			Let $\hat k \in \NN$ and $k\geq \hat k$ so that $\xk \notin \CCC(D,\nu,\gamma)$. By the definition of $\CCC(D,\nu,\gamma)$ 
			\begin{align}
				\label{eq.notin_tilde_c}
				\left\| \nabla \tilde \varphi^{\mu}(\xk)\right\|_{D} > \sqrt{\tfrac{\hat\sigma_u^D}{\hat\sigma_l^G }}\left(\tfrac{1 + 2 \nu}{1 - 2 \nu} + 1 \right)\eps_g, \ \text{and} \ 
				\|\nabla \tilde \varphi^{\mu}(\xk)\|_D > \sqrt{\hat \sigma_u^D} \sqrt{\tfrac{4 \eps_f + 2\eps_R}{{\gamma\bar \alpha} \nu }}.
			\end{align}
			The definition of $ \hat \sigma_u^D$ implies $ \left\|D^{1/2} (\hat G_k^{\mu})^{1/2} \right\| \leq \sqrt{\hat \sigma_u^D}$. Since $\hat G_k^{\mu}$ is invertible, we have $ \left\|\nabla \tilde \varphi^{\mu}(\xk)\right\|_{D} \leq  \left\|D^{1/2} (\hat G_k^{\mu})^{1/2} \right\| \left\| (\hat G_k^{\mu})^{-1/2} \nabla \tilde \varphi^{\mu}(\xk)\right\|
			$. This, along with the first inequality in \eqref{eq.notin_tilde_c} implies
			\begin{align}
				\label{eq.global_conv_tmp1}
				\begin{split}
					\sqrt{\hat \sigma_u^D}\left\| \nabla \tilde \varphi^{\mu}(\xk)\right\|_{(\hat G_k^{\mu})^{-1}}
					& \geq \left\|\nabla \tilde \varphi^{\mu}(\xk)\right\|_{D} 
					>   \sqrt{\tfrac{\hat \sigma_u^D }{\hat \sigma_l^G }}\left(\tfrac{1 + 2 \nu}{1 - 2 \nu} + 1 \right)\eps_g \\
					&\geq  \sqrt{ \hat \sigma_u^D }\left(\tfrac{1 + 2 \nu}{1 - 2 \nu} + 1 \right)\|\vareps_g(\xk) \|_{(\hat G_k^{\mu})^{-1}},
				\end{split}
			\end{align}
			because $\|\vareps_g(\xk) \|_{(\hat G_k^{\mu})^{-1}}\leq 1/\sqrt{\hat \sigma_l^G}\cdot\epsilon_g$ and \eqref{eq:sigma_l_G}.
			By Assumption~\ref{assumption.bnd_err}, we have $\| \nabla \tilde \varphi^{\mu}(\xk)\|_{(\hat G_k^{\mu})^{-1}} \leq  \|\vareps_{g}(\xk)\|_{(\hat G_k^{\mu})^{-1}} + \|\nabla  \varphi^{\mu}(\xk) \|_{(\hat G_k^{\mu})^{-1}}$. Together this yields
			\begin{align}
				\label{eq.tmp1}
				\|\nabla  \varphi^{\mu}(\xk)\|_{(\hat G_k^{\mu})^{-1}} >&\; \left(\tfrac{1 + 2 \nu}{1 - 2 \nu} \right)\|\vareps_g(\xk)\|_{(\hat G_k^{\mu})^{-1}}.
			\end{align}
			Hence, by Lemma~\ref{lemma.bar_alpha} and Assumption~\ref{assumption.bnd_err} one concludes that the backtracking line-search finds $\alpha_k >  \bar \alpha/2$  satisfying the relaxed Armijo condition, that is
			\begin{align} 
				\label{eq.phi_c1_condition}
				\begin{split}
					\varphi^{\mu}(\xk + \alpha_k d_k) - \varphi^{\mu}(\xk) \leq&\;  \tilde \varphi^{\mu}(\xk + \alpha_k d_k) - \tilde \varphi^{\mu}(\xk) + 2 \eps_f \\
					<&\; - \tfrac{\bar \alpha}{2} \nu \left\| \nabla \tilde \varphi^{\mu}(\xk) \right\|_{(\hat G_k^{\mu})^{-1}}^2 + 2 \eps_f + \eps_R. 
				\end{split}
			\end{align}
			The second inequality in \eqref{eq.notin_tilde_c}, along with the first inequality in \eqref{eq.global_conv_tmp1} yields
			\[
			\|\nabla \tilde \varphi^{\mu}(\xk)\|_{(\hat G_k^{\mu})^{-1}} > \sqrt{\tfrac{4\eps_f +  2\eps_R}{\gamma{\bar \alpha} \nu }},
			\]
			which in combination with \eqref{eq.phi_c1_condition} implies
			\[
			\varphi^{\mu}(\xk + \alpha_k d_k) -  \varphi^{\mu}(\xk) <\, (1-\tfrac{1}{\gamma}) (2\eps_f +  \eps_R) <\, 0.
			\]
			Since $\varphi$ is bounded below by Assumption~\ref{assumption.bnd_const_grad_lipschitz},  $\xk \notin \CCC(D,\nu,\gamma)$ cannot hold for all $k\geq \hat k$, where $\hat k$ can be arbitrarily large.
		\end{proof}
		
		This theorem presents a global convergence result in the sense that it states that the gradient is small for infinitely many iterations.   The set $\CCC(D,\nu,\gamma)$ is defined in terms of the noisy gradient, but it can easily be transferred into a condition using the norm of the non-noisy gradient $\|\nabla \varphi^\mu(\xk)\|_D$, by adding $\epsilon_g$ to the right-hand side, multiplied by a constant that accounts for the difference between $\|\cdot\|_D$ and $\|\cdot\|$.
		Since $\gamma\in(0,1)$ can be chosen arbitrarily close to one, an alternative way to state this result is that there exists a limit point of $\xk$ that is in $\CCC(D,\nu,1)$.
		
		Different choices for the norm $\|\cdot\|_D$ are possible, resulting in different conditions.
		For instance, for $D=I$, we have on the right-hand side $\hat \sigma_u^D=\max_k\|\hat H_k+\mu (\Xk)^{-2}\|$, a quantity that converges to infinity as $\mu\to0$ if $\xk\approx x_*^{\mu}$. (Recall that $\mu (X_*^{\mu})^{-2}= Z^\mu_*(X^\mu_*)^{-1}$ due to \eqref{eq.kktmu}.) At the same time, though, $\left\|\nabla \tilde \varphi^{\mu}(\xk)\right\|= \left\|\tilde g(\xk) - \mu (\Xk)^{-1}\mathbf{e}\right\|$ is a bounded entity. This makes the result essentially meaningless for $D=I$ as $\mu\to0$.
		
		On the other hand, choosing $D=(X^\mu_*)^{-2}$ yields $\hat \sigma_u^D=\sup_k\|X^\mu_*\hat H_kX^\mu_*+\mu I\|$ and the spurious high curvature induced by the barrier term is neutralized, leaving behind a matrix that becomes a scaled version of the reduced Hessian as $\mu\to0$.  Finally, $D=G^\mu(x_*^\mu)=\nabla^2 f(x_*^\mu)+\mu (X_*^\mu)^{-2}$ provides a scale-invariant measure and, if $\hat H_k\approx \nabla^2 f(x_*^\mu)$, yields $\hat \sigma_u^D\approx1$.  This is what is used in the stopping test below.

		\subsection{A practical stopping test}
		\label{sec.stopping_test}
		In this section, having established the global convergence result in Theorem~\ref{thm.neigh}, we take our findings a step further by providing a practical stopping test for Algorithm~\ref{alg.bnd_const_fixed_mu}.
		The practicality of this stopping test lies in the fact that it does not depend on unknown parameters such as $\hat \sigma_u^D$, $\hat\sigma_l^G$, and $\bar{\alpha}$. Instead, all these parameters have been replaced by their iteration-dependent counterparts which can be computed as the algorithm runs.  

		This result is formally presented in the following theorem.
		
		\begin{theorem}
			\label{thm.practical_termination_v2}
			There exists $k \in \NN$ such that
			\begin{enumerate} [label=(\roman*)]
				\item \label{it.stop1_v2} $\left\|\nabla \tilde \varphi^{\mu}(x_k^\mu) \right\|_{(\hat G_k^{\mu})^{-1}} \leq T_{1,k}$, or 
				\item \label{it.stop2_v2} $\left\|\nabla \tilde \varphi^{\mu}(x_k^\mu) \right\|_{(\hat G_k^{\mu})^{-1}} > T_{1,k}$ and $\left\|\nabla \tilde \varphi^{\mu}(x_k^\mu) \right\|_{(\hat G_k^{\mu})^{-1}} \leq T_{2,k}$,
			\end{enumerate}
			where
			\begin{equation}\label{eq:Tk}
				T_{1,k} = {\tfrac{1}{\sqrt{\hat \sigma_{l,k}^G}}} \left( \tfrac{1 + 2 \nu_k}{1 - 2 \nu_k} + 1 \right) \eps_g \quad\text{ and }\quad T_{2,k} = \sqrt{\tfrac{2 \eps_f + \eps_R}{\gamma \alpha_k \nu_k}},
			\end{equation}
			$\gamma \in (0, 1)$ is a fixed constant that can be chosen arbitrarily close to one, $ \hat \sigma_{l,k}^G =  {\sigma_{\min}(\hat G_k^{\mu})}$, and $\nu_k$ is an arbitrarily chosen scalar in $[\nu,\hat\nu_{1,k}]$, where
			\begin{equation}\label{eq:hatnu1}
				\hat \nu_{1,k} := \tfrac{\tilde \varphi(x_k) - \tilde \varphi(x_k + \alpha_k d_k) + \eps_R}{\alpha_k \|\nabla \tilde \varphi^\mu(x_k^\mu) \|_{(\hat G_k^\mu)^{-1}}^2} \geq \nu.
			\end{equation}
			Furthermore, such iteration satisfies $\xk\in \CCC\left((\hat G_k^{\mu})^{-1},\nu_k,\gamma\right)$.
		\end{theorem}
		
		\begin{proof}
			Suppose that $\| \nabla \tilde \varphi^{\mu}(x_k^\mu) \|_{(\hat G_k^{\mu})^{-1}}>T_{2,k}$ for all $k\in\NN$.
			Observe that $\hat \nu_{1,k} $ is the largest value of $\nu$ such that the Armijo condition \eqref{eq.relaxed_armijo_bnd} holds for the step size $\alpha_k$  in the $k$th iteration.  Therefore, $\hat \nu_{1,k} \geq\nu$ and
			\begin{align*}
				\varphi^{\mu}(\xk + \alpha_k d_k) - \varphi^{\mu}(\xk) &\leq \tilde \varphi^{\mu}(\xk + \alpha_k d_k) - \tilde \varphi^{\mu}(\xk) + 2 \eps_f \\
				& \leq  - \alpha_k \nu_k \|\nabla \tilde \varphi^\mu(x_k^\mu)\|_{(\hat G_k^{\mu})^{-1}}^2 + 2 \eps_f + \eps_R \\
				&<  \left(1 - \tfrac{1}{\gamma}\right) \left(2 \eps_f + \eps_R\right) < 0,
			\end{align*}
			where we used the definition of $T_{2,k}$. But this cannot hold for all $k\in\NN$ since $\varphi^\mu$ is bounded below. Hence, there exists $k\in\NN$ such that $\| \nabla \tilde \varphi^{\mu}(x_k^\mu) \|_{(\hat G_k^{\mu})^{-1}}\leq T_{2,k}$, and therefore one of the conditions~\ref{it.stop1_v2} and \ref{it.stop2_v2} must hold.
			
			Next, let $k \in \NN$ such that condition~\ref{it.stop1_v2} or condition~\ref{it.stop2_v2} holds.  If \ref{it.stop1_v2}  holds, recall that $\hat\sigma_u^D=1$ for $D=(\hat G_k^{\mu})^{-1}$ and $\hat \sigma_{l,k}^G\geq \hat\sigma_l^G$.  This implies that $T_{k,1}$ is not larger than the first term in the $\max$-expression in \eqref{eq:defCCC}, and therefore $\xk\in \CCC\left((\hat G_k^{\mu})^{-1},\nu_k,\gamma\right)$.
			
			Now suppose that \ref{it.stop2_v2} holds. Since $\hat \sigma_{l,k}^G = \sigma_{\min}\left(\hat G_k^{\mu}\right)$, one has $\tfrac{\eps_g}{\sqrt{\hat \sigma_{l,k}^G}} \geq \|\vareps_g(\xk)\|_{(\hat G_k^{\mu})^{-1}}$.
			This, together with $\| \nabla \varphi^\mu(\xk) \|_{(\hat G_k^{\mu})^{-1}} + \|\vareps_g(\xk) \|_{(\hat G_k^{\mu})^{-1}}\geq \| \nabla \tilde \varphi^\mu(\xk) \|_{(\hat G_k^{\mu})^{-1}}$ and the first inequality in \ref{it.stop2_v2} implies 
			\begin{align*}
				\| \nabla \varphi^\mu(\xk) \|_{(\hat G_k^{\mu})^{-1}} + \|\vareps_g(\xk) \|_{(\hat G_k^{\mu})^{-1}}
				>
				\left( \tfrac{1 + 2 \nu_k}{1 - 2 \nu_k} + 1 \right) \|\vareps_g(\xk) \|_{(\hat G_k^{\mu})^{-1}},
			\end{align*}
			and by \eqref{eq.kappa_mu_beta}, one has
			$\beta_k < \tfrac{1 - 2 \nu_k}{1 + 2 \nu_k}$. Since $\nu_k \geq \nu$, one also has that
			$\beta_k < \tfrac{1 - 2 \nu}{1 + 2 \nu}$. Hence, by Lemma~\ref{lemma.bar_alpha} one concludes that \eqref{eq.relaxed_armijo_bnd} is satisfied with $\alpha_k > \tfrac{\bar \alpha}{2}$. 
			Consequently, $T_{k,2}$ is not larger than the second term in the $\max$-expression in \eqref{eq:defCCC}, and the second inequality in \ref{it.stop2_v2} yields $\xk\in \CCC((\hat G_k^{\mu})^{-1},\nu_k,\gamma)$.   
		\end{proof}
		
		\begin{remark}\label{rem:balanceT}
			This theorem can be used to define a practical termination test.  During the run of the algorithm, we can check in each iteration whether condition \ref{it.stop1_v2} or \ref{it.stop2_v2} in Theorem \ref{thm.practical_termination_v2} holds.  If that is the case, we can conclude that the noisy gradient of the barrier function at the current iterate $\xk$ is as small as we were able to prove in the theoretical convergence Theorem~\ref{thm.neigh} (if we choose $\nu_k=\nu$), without the need to know the values of the parameters $\hat\sigma_l^G$ and $\bar\alpha$.  Since these values are pessimistic uniform estimates, the use of the actual values $\hat\sigma_{l,k}^G$ and $\alpha_k$ in $T_{1,k}$ and $T_{k,2}$ typically results in a smaller final value of $\left\|\nabla \tilde \varphi^{\mu}(\xk) \right\|_{(\hat G_k^{\mu})^{-1}}$ than what Theorem~\ref{thm.neigh} predicts.
			
			By keeping the choice of the Armijo condition parameter $\nu_k$ flexible, we can tighten the final tolerance even further.  The tolerances $T_{1,k}$ and $T_{2,k}$ are monotonically increasing and decreasing, respectively, as $\nu_k$ becomes larger.  To obtain the tightest bound we choose the value that balances the two tolerances
			, namely
			\[
			\hat \nu_{2,k} := \tfrac{(2 \eps_f + \eps_R)\hat \sigma_{l,k}^G + \gamma \alpha_k \eps_g^2 - \sqrt{\gamma^2 \alpha_k^2 \eps_g^4 + 2 (2 \eps_f + \eps_R )\hat \sigma_{l,k}^G \gamma \alpha_k \eps_g^2} }{2 (2 \eps_f + \eps_R)\hat \sigma_{l,k}^G }.
			\]
			However, since the assumptions of Theorem \ref{thm.practical_termination_v2} require that $\nu_k\in[\nu,\hat\nu_{1,k}]$, the tightest value we can choose (and use in our numerical experiments) is
			\begin{equation}
				\label{eq.nu_k}
				\nu_k := \max\left\{\nu, \min\left\{\hat \nu_{1,k}, \hat \nu_{2,k} \right\} \right\}.
			\end{equation}
		\end{remark}

		\section{Local convergence} 
		\label{sec.local_conv}
		
		\newcommand{\Bdelta}{\bar B_\delta(x^\mu_*)}
		\newcommand{\Bbardelta}{\bar B_{\bar\delta}(x^\mu_*)}

		In this section, we examine the behavior of the iterates when they are close to a local optimal solution satisfying the second-order optimality conditions and when Newton-type steps are taken with a noisy approximation of the Hessian matrix.
		First, we derive in Section~\ref{sec:local_quadratic} a noisy version of a local quadratic convergence rate, presented in Theorem~\ref{thm_main_neigh_1b}.  This permits us to show in Section~\ref{sec:local_neighborhood} the existence of ``neighborhoods of confusion'' around the optimal solution in which the iterates remain once entered.  
		%
		%
		A significant effort is made to keep all constants independent of the barrier parameter $\mu$ because this allows us to eventually argue that interior-point methods can identify the active set (for non-degenerate problems) despite noise.
		We emphasize that we do not require that the problem is convex.
		
		Throughout this section, we suppose that Assumptions~\ref{as:suff_optcond}, \ref{assumption.bnd_err}, \ref{assumption.bnd_const_grad_lipschitz}, and \ref{as:Hbounded} hold for a local primal-dual optimal solution $(x_*^0,z_*^0)$ of the original bound-constrained problem~\eqref{eq.orig_bnd_const_prob}. In what follows, $(x_*^\mu,z_*^{\mu})$ for $\mu\in[0,\hat\mu]$ denotes the primal-dual optimal solutions 
		along the differentiable central path, as defined in Theorem~\ref{thm:central_path}.
		In the analysis, neighborhoods are defined in a scaled norm that takes the geometry of the barrier terms into account.  For a given radius $\delta \in \RR_{>0}$, we define the open ball
		\begin{equation}\label{eq:def_neighborhood}
			B_\delta(x^\mu_*) = \left\{x\in\RR^n : \|x-x^\mu_*\right\|_{(X^\mu_*)^{-2}}<\delta \},
		\end{equation}
		and we denote its closure with $\Bdelta$. 
		Without loss of generality, we assume that $x_*^\mu=(x_{*,\AAA}^\mu,x_{*,\III}^\mu)$ ($\AAA$ and $\III$ defined in Assumption~\ref{as:suff_optcond}) and 
		we use $\xk$ to denote the iterates
		generated by Algorithm~\ref{alg.bnd_const_fixed_mu} when it is executed for a fixed 
		$\mu$ and define
		\begin{equation}
			\label{eq.G}
			G^\mu(x) = \,  H(x)  + \mu X^{-2}.  
		\end{equation}
		We let $H^\mu_k=H(\xk)$, $G^\mu_k=G(\xk)$, $H_*^\mu=H(x^\mu_*)$, and $G_*^\mu=G(x^\mu_*)$, and we use similar simplifications for the noisy quantities.
		We also define $z_k^\mu = \mu (X_k^\mu)^{-1}\mathbf e$.

		This part of our analysis relies on the contraction property of Newton's method, extended to the noisy case as analyzed in \cite{kelley2022newton}.  Our work goes beyond that in \cite{kelley2022newton} by allowing an unbounded objective function $\varphi^\mu$ and deriving constants that are independent of $\mu$. 
		For this analysis to be applicable, the matrix $\hat H_k$ in \eqref{eq.hat_G_k} cannot be chosen as freely as in Section~\ref{sec.global_conv}.  Instead, we need to assume that the algorithm uses a sufficiently accurate noisy approximation of the true second derivatives $\tilde H_k^\mu\approx\na_{xx}^2 f(\xk)$ when an iterate is close to the local minimum $x_*^\mu$.
		More precisely, whenever the noisy approximation of $\na^2\varphi^\mu(\xk)$, given by $\tilde G_k^{\mu} := \tilde H_k^\mu + \mu (\Xk)^{-2}$, satisfies
		\begin{equation}\label{eq:tildesigma_l_G}
			\sigma_{\min}(\tilde G_k^{\mu})\geq \tilde\sigma_l^G,
		\end{equation}
		for a fixed value $\tilde\sigma_l^G>0$, the algorithm chooses $\hat G_k^\mu=\tilde G_k^\mu$ in Step~\ref{st.compute_dir} and computes the search direction from
		\begin{equation}\label{eq.noisy_d_N}
			d_k = -(\tilde G_k^\mu)^{-1}\na \tilde \varphi^\mu_k.
		\end{equation}
		Lemma~\ref{lemma.G_unif_posdef_iter} guarantees the existence of a value for $\tilde\sigma_l^G>0$ so that \eqref{eq:tildesigma_l_G} holds when $\mu$ is sufficiently small and $x_k$ is sufficiently close to $x_*^\mu$, assuming that the noise $\|\tilde H_k^\mu-\na^2 f(\xk)\|$ is not too large.
		In this way, Assumption~\ref{assumption.quasi_Newton_Hess_singular_val} remains satisfied and the convergence analysis in Section~\ref{sec.global_conv} still applies. At the same time, the algorithm takes steps based on a noisy Newton-type direction close to the local minimum.

		We make the following standard assumption throughout this section.
		Recall that $\XXX$ is defined in Assumption~\ref{assumption.bnd_const_grad_lipschitz}.
		\begin{assumption}
			\label{assumption.bnd_const_hess_lipschitz_1}
			There exists an open set $\XXX_l$ satisfying $\cup_{\mu\in(0,\hat\mu]}\bar B_1(x_*^\mu)\subseteq\XXX_l\subseteq\XXX$ so that the Hessian $H(x)$ exists and is  Lipschitz continuous over $\mathcal{X}$ with constant $L_H \in \RR_{>0}$, i.e., for all $x, y \in \mathcal{X}_l $, we have
			$
			\| H(x) - H(y) \| \leq L_H \|x - y\|.
			$
		\end{assumption}
		In order to state the permissible noise level for the Hessian matrix, we define
		\[
		\Gamma_*^0 := H_*^0X_*^0 + Z_*^0 =
		\begin{pmatrix}
			Z^0_{*,\AAA} & H^0_{*,\AAA\III}X_{*,\III}^0 \\ 
			0 & H^0_{*,\III\III}X_{*,\III}^0 
		\end{pmatrix}.
		\]
		The second equality holds because $x_*^0=(0,x_{*,\III}^0)$ and $z_*^0=(z_{*,\AAA}^0,0)$ due to complementarity.
		Note that this matrix is nonsingular, due to Assumption~\ref{as:suff_optcond}, and that
		\begin{equation}\label{eq:Gamma0limit}
			\lim_{\mu\to0} G^\mu_*X_*^{\mu}= \lim_{\mu\to0} (H^\mu_*X_*^{\mu} + Z^\mu_*) = \Gamma_*^0.
		\end{equation}
		\begin{assumption}
			\label{as.eps_H}
			The noisy evaluation of the Hessian, denoted by $\tilde H$, satisfies
			\begin{align}
				\| \tilde H(x) - H(x) \| \leq \eps_H,        \label{eq.hessian_err}
			\end{align}
			for all $x \in \mathcal{X}_l$ and some $\epsilon_H \in \RR_{>0}$ that satisfies
			\begin{equation}\label{eq:epsHstrict}
				\eps_H < \min \left \{ \tfrac{1}{4}\sigma_l^H, \tfrac{1}{4 \|x_*^{0}\|_{\infty} \|(\Gamma_*^0)^{-1}\|} \right\},
			\end{equation}
			where $\sigma_l^H$ is defined in Assumption~\ref{as:suff_optcond}.
		\end{assumption}

		\subsection{Local noisy quadratic-linear convergence rate}
		\label{sec:local_quadratic}
		The main result of this section, stated in Theorem~\ref{thm_main_neigh_1b}, is a generalization of Theorem~2.4 in \cite{kelley2022newton} for the barrier problem.  It is essentially
		a noisy version of a local quadratic convergence rate:
		For sufficiently small $\mu$, there exist 
		$M_1, M_2, M_3 \in \RR_{>0}$ (independent of $\mu$) so that
		\begin{equation}\label{eq:contract} e_k^{\mu,+}
			\leq M_{2} (e^{\mu}_k)^2
			+  M_{1} \eps_H e^\mu_k
			+ M_{0} \epsilon_g,
		\end{equation}
		where $e^\mu_k =  \| x_{k} - x_*^{\mu} \|_{(X_*^{\mu})^{-2}}$ and $e^{\mu,+}_k =  \| x_{k} + d_k - x_*^{\mu} \|_{(X_*^{\mu})^{-2}}$. 
		That is, when the noise level is small compared to the error in the iterates, the steps are almost locally quadratically convergent. However, when the iterates get closer to the true optimal solution, the steps are merely linearly convergent, and eventually the noise limits how accurately the iterates can approach the optimal solution.

		As in \cite{kelley2022newton}, we need to assume that the iterates are sufficiently close to $x^\mu_*$ for the contraction property \eqref{eq:contract} to hold.  Most of the results in this section are valid for iterates $\xk\in\Bbardelta$ where the radius $\bar\delta$ satisfies
		\begin{equation}\label{eq:bardelta0}
			\bar\delta \in\left(0,\; \xi_\delta\cdot\min \left\{\tfrac{\sigma_l^H}{2 L_H}, \,\tfrac{1}{2 \tilde L^0  \|(\Gamma_*^0)^{-1}\|},1 \right\}\right],
		\end{equation}
		with
		$\tilde L^0 = L_H \|x_*^{0} \|_{\infty}^2 +  \tfrac{2 - \bar \delta}{(1 - \bar \delta)^2} \|z^0_*\|_{\infty}$
		and a fixed parameter $\xi_\delta\in(0,1)$ arbitrarily close to one.
		Note that the constant $\tilde L^0$ on the right-hand side of \eqref{eq:bardelta0} also depends on $\bar\delta$.  Nevertheless, since $\tilde L^0$ converges to a positive number as $\bar\delta\to0$ due to Assumption~\ref{as:suff_optcond}, there always exists a value of $\bar\delta\in(0,1)$ that satisfies \eqref{eq:bardelta0}.

		We begin our analysis with some technical results.

		\begin{lemma}
			\label{lemma.DX_bnds}
			Let $\mu\in(0,\hat\mu]$ (with $\hat\mu$ from Theorem~\ref{thm:central_path}) and $\delta\in(0,1)$, and suppose $x\in\Bdelta$. Then
			\begin{equation}
				\label{eq.x_*_over_x_k}
				\tfrac{1}{1+\delta} \leq \tfrac{x_{*,i}^{\mu}}{ x_i} \leq \tfrac{1}{1-\delta}\ \text{ for all }i\in[n].
			\end{equation}

		\end{lemma}
		\begin{proof}
			By the definition of the $ \|\,\cdot\,\|_{(X_*^{\mu})^{-2}}$-norm
			\[
			\delta^2 \geq \|x - x_*^{\mu} \|_{(X_*^{\mu})^{-2}}^2 = \sum_{i=1}^n \tfrac{|x_{i} - x_{*,i}^{\mu} |^2}{\left(x_{*,i}^{\mu}\right)^2},
			\]
			which implies
			$\tfrac{|x_{i} - x_{*,i}^{\mu}|}{x_{*,i}^{\mu}} \leq\delta$ for all $i \in [n]$. 
			In addition, observe that for all $i \in [n]$
			$
			\tfrac{x_{i}}{x_{*,i}^{\mu}} = \tfrac{x_{i} - x_{*,i}^{\mu}}{x_{*,i}^{\mu}} + 1.
			$
			Combining the last two relationships yields \eqref{eq.x_*_over_x_k}.
		\end{proof}

		\begin{lemma}\label{lemma:CT_Kelley} \cite[Lemma 2.1 ]{kelley2022newton}
			If matrix $A$ is non-singular and 
			$\|A - B \| \leq \tfrac{1}{2 \|A^{-1}\|}$,
			then  $B$ is non-singular, $\|B^{-1}\| < 2 \|A^{-1}\| $, and 
			$ \| A^{-1} - B^{-1} \| \leq 2 \|A^{-1}\|^2 \|A - B\|$.
		\end{lemma}

		\begin{lemma} \label{lemma.H_II_bnd}
			There exists $\mu_1 \in (0, \hat\mu)$ such that for all $\mu \in [0, \mu_1]$ and for all $\bar\delta$ satisfying \eqref{eq:bardelta0}, one has that
			\begin{align}
				\eps_H &< \min \left \{\tfrac{1}{4}\sigma_l^H, \tfrac{1}{4 \|x_*^{\mu}\|_{\infty} \|(G_*^\mu X_*^{\mu})^{-1}\|} \right\}, \label{eq:epsH} \ \ \text{and} \\
				\bar\delta &\leq \min \left\{\tfrac{\sigma_l^H}{2 L_H }, \tfrac{1}{2 \tilde L^\mu  \|(G_*^\mu X_*^\mu)^{-1}\|},1 \right\}, \label{eq:bardelta}
			\end{align}
			where
			$\tilde L^\mu = L_H \|x_*^{\mu} \|_{\infty}^2 +  \tfrac{2 - \bar \delta}{(1 - \bar \delta)^2} \|z^\mu_*\|_{\infty},$
			and that, whenever $\xk\in \Bbardelta$,
			\begin{align}
				\sigma_{\min} \left( H^\mu_{k, \III\III} \right) &\geq \tfrac{1}{4}\sigma_l^H > 0, \label{eq:lamHmin}  \ \text{and} \\
				\sigma_{\min} \left( \tilde H^\mu_{k, \III\III}\right) &\geq \tfrac{1}{4}\sigma_l^H -\epsilon_H > 0,  \label{eq:tildelamHmin}  
			\end{align}
		\end{lemma}
		\begin{proof}
			By Theorem~\ref{thm:central_path}, $x^\mu_*$ and $z_*^\mu$ are continuous functions of $\mu$.  As a consequence, all quantities on the right-hand sides of
			\eqref{eq:epsH} and \eqref{eq:bardelta} depend continuously on $\mu$ and converge to the corresponding values in \eqref{eq:epsHstrict} and \eqref{eq:bardelta0} as $\mu\to0$. Therefore there exists $\mu_1>0$ so that for all $\mu\in[0,\mu_1]$ we have that \eqref{eq:epsH} and \eqref{eq:bardelta} hold.  After possibly reducing $\mu_1$, we further have that
			\begin{equation}
				\label{eq.t2}
				\|x_*^{\mu} - x_*^0 \| \leq \tfrac{\sigma_l^H}{2 L_H }, \ \ \text{for all}  \ \ \mu\in[0,\mu_1].
			\end{equation}
			By Assumption~\ref{assumption.bnd_const_hess_lipschitz_1}, one has
			$
			\|H_{*, \III\III}^{\mu} -  H_{*, \III\III}^0 \| \leq L_H \|x_*^{\mu} - x_*^0 \|, 
			$
			which combined with \eqref{eq.t2} yields
			$\|H_{*, \III\III}^{\mu} -  H_{*, \III\III}^0 \| \leq \tfrac{\sigma_l^H}{2}$.
			Thus, by Lemma~\ref{lemma:CT_Kelley}
			\begin{align}
				\label{eq.t3}
				\|(H_{*, \III\III}^{\mu})^{-1}\| \leq 2 \|(H_{*, \III\III}^{0})^{-1}\| \leq \tfrac{2}{\sigma_l^H}.
			\end{align}
			For $\xk\in\Bbardelta$, by Assumption~\ref{assumption.bnd_const_hess_lipschitz_1} and \eqref{eq:bardelta} we have 
			\begin{align*}
				\|H_{*, \III\III}^{\mu} -  H_{k, \III\III}^\mu \| & \leq  L_H \|x_*^{\mu} - \xk\| 
				= L_H e^\mu_k\leq  L_H\bar\delta \leq \tfrac{\sigma_l^H}{2}.
			\end{align*}
			Hence, applying 
			Lemma~\ref{lemma:CT_Kelley} again, along with \eqref{eq.t3}, yields
			$$ \|(H_{k, \III\III}^\mu)^{-1}\| \leq 2 \|(H_{*, \III\III}^{\mu})^{-1}\| \leq \tfrac{4}{\sigma_l^H},$$
			or equivalently \eqref{eq:lamHmin} holds.
			Finally
			\begin{align*}
				\sigma_{\min}\left(\tilde H_{k, \III\III}^\mu\right) &= \sigma_{\min}\left(H_{k, \III\III}^\mu - (H_{k, \III\III}^\mu-\tilde H_{k, \III\III}^\mu)\right)\\&\geq \sigma_{\min}\left(H_{k, \III\III}^\mu\right) - \sigma_{\max}\left(H_{k, \III\III}^\mu-\tilde H_{k, \III\III}^\mu\right) \\
				&= \sigma_{\min}\left(H_{k, \III\III}^\mu\right) - \|H_{k, \III\III}^\mu-\tilde H_{k, \III\III}^\mu\|,
			\end{align*}
			which implies \eqref{eq:tildelamHmin} due to \eqref{eq.hessian_err} and \eqref{eq:epsH}.
		\end{proof}

		The following lemma shows that the search direction can be computed in \eqref{eq.noisy_d_N}.  
		\begin{lemma}\label{lemma.G_unif_posdef_iter}
			Suppose $\bar\delta\in(0,1)$ satisfies \eqref{eq:bardelta0}.   Then, there exist $\mu_2\in(0,\mu_1]$, and $\sigma^G_l, \tilde\sigma^G_l \in \RR_{> 0}$ so that for all $\mu \in [0, \mu_2]$ one has that
			\[
			\sigma_{\min}\left(G_k^\mu\right)\geq\sigma_l^G, \,\text{ and } \,  \sigma_{\min}(\tilde G_k^\mu)\geq\tilde\sigma_l^G,
			\]
			whenever $\xk\in\Bbardelta$.  In particular, $G_k^\mu$ and $\tilde G_k^\mu$ are nonsingular if $\xk\in\Bbardelta$.
		\end{lemma}
		\begin{proof}
			We first prove the existence of $\sigma^G_l$.
			Let $\mu_1$ be the threshold from Lemma~\ref{lemma.H_II_bnd} and let $\lambda_H=\tfrac{\sigma_l^H}{8}$, and  $M_H=\max_{\mu\in[0,\mu_1], x \in  \Bbardelta  }\|H(x)\|$.
			Define the  quadratic $\phi(t)=-\left(\lam_H+M_H\right)t^2 - 2M_H\,t + \lam_H$ and let $M_\AAA\in(0,1)$  be its positive root. 
			%
			
			Now suppose $\xk\in\Bbardelta$ and choose any $v\in\RR^
			n$ with $\|v\|=1$.  We need to show that there exists $\sigma^G_l \in \RR_{>0}$ such that $v^TG_k^\mu v\geq \sigma^G_l$ for all sufficiently small $\mu$ independent of the particular choices of $v$ and $\xk$.  
			We consider the following two cases where we set $v=(v_\AAA,v_\III)$ with the $v_\AAA$ and $v_\III$
			correspond to the index sets $\AAA$ and $\III$.
			
			Case (i):  Suppose $\|v_\AAA\|\leq M_\AAA$.  Then 
			$-\left(\lam_H+M_H\right)\|v_\AAA\|^2 - 2M_H\|v_\AAA\| + \lam_H \geq 0$,
			because $M_\AAA$ is the positive root of $\phi(t)$, $\|v_\AAA\|$ is between the roots, and $\phi$ is concave. 
			By $1=\|v\|^2 = \|v_\III\|^2 + \|v_\AAA\|^2$, one concludes that $\|v_\III\|\leq1$. Substituting these relationships into the previous inequality yields, for any $\mu\in[0,\mu_1]$ that $0\leq\lambda_H\|v_\III\|^2 - M_H\|v_\AAA\|^2 - 2M_H\|v_\III\|\|v_\AAA\|$.
			Using \eqref{eq:lamHmin} and the definition of $\lam_H$ we have
			$
			\lam_H \|v_\III\|^2 = \tfrac14 \sigma_l^H\|v_\III\|^2 - \tfrac18 \sigma_l^H\|v_\III\|^2 \leq v_\III^T H_{k,\III\III}^{\mu}v_\III - \tfrac18 \sigma_l^H\|v_\III\|^2. 
			$
			Combining the last two inequalities yields
			\begin{equation}
				0 \leq -\tfrac18\sigma_l^H\|v_\III\|^2 + v_\III^T H_{k,\III\III}^{\mu}v_\III + v_\AAA^T H_{k,\AAA\AAA}^{\mu}v_\AAA  + 2 v_\III^T H_{k,\III\AAA}^{\mu}v_\AAA, 
				\label{eq:lamHvI_2}
			\end{equation}
			where we used the definition of $M_H$ and the fact that $\|H_{k,\AAA\AAA}^{\mu}\|\leq \|H_k^{\mu}\|$ and $\|H_{k,\III\AAA}^{\mu}\|\leq \|H_k^{\mu}\|$.
			Combining \eqref{eq.G} with $z^\mu_k=\mu (X_k^\mu)^{-1}\mathbf{e}$ we obtain
			\begin{align}
				\label{eq:vGmuv_2}
				v^TG_k^\mu v  &  = v_\III^T H_{k,\III\III}^{\mu}v_\III + v_\AAA^T H_{k,\AAA\AAA}^{\mu}v_\AAA  + 2 v_\III^T H_{k,\III\AAA}^{\mu}v_\AAA \\ & \qquad+ v_\III^T Z_{k,\III}^\mu(X_{k,\III}^\mu)^{-1} v_\III  + v_\AAA^T Z_{k,\AAA}^\mu(X_{k,\AAA}^\mu)^{-1} v_\AAA\nonumber.
			\end{align}
			Because the last two terms are non-negative, this together with \eqref{eq:lamHvI_2} leaves us with
			$
			v^TG_k^\mu v \geq \tfrac18\sigma_l^H\|v_\III\|^2.
			$
			Since $\|v_\III\|^2 = 1 - \|v_\AAA\|^2$, $\|v_\AAA\|\leq M_\AAA$, and $M_\AAA \in (0, 1)$, it follows that $v^TG_k^\mu v \geq \sigma^G_l$ with $\sigma^G_l=\tfrac18\sigma_l^H(1-M_\AAA^2)>0$ for all $\mu\in(0,\mu_1]$.

			Case (ii):  Suppose that $\|v_\AAA\|>M_\AAA$.  
			Since $\|v\|=1$ and $H^\mu_k$ is uniformly bounded, the first three terms on the right-hand side of \eqref{eq:vGmuv_2} are uniformly bounded below by a constant $M\in\RR$ for all $\mu\in(0,\mu_1]$.  Since the fourth term is non-negative, we obtain
			\begin{equation}\label{eq:quark}
				v^TG_k^\mu v \geq M + v_\AAA^T Z_{k,\AAA}^\mu(X_{k,\AAA}^\mu)^{-1} v_\AAA
				\geq M +
				\min_{i\in\AAA}\tfrac{z^\mu_{k, i}}{x^\mu_{k, i}}\cdot M_\AAA^2 
			\end{equation}
			for all $\mu\in(0,\mu_1]$.
			By $z^\mu_{k,i}=\mu/x^\mu_{k,i}$, Lemma~\ref{lemma.DX_bnds}, and $\bar\delta \in (0, 1)$ we have for $i\in[n]$
			\[
			\tfrac{z^\mu_{k, i}}{x^\mu_{k, i}} = \tfrac{z^\mu_{*,i}}{x^\mu_{*,i}} \left(\tfrac{x_{*, i}^{\mu}}{x^\mu_{k, i}}\right)^2 \geq \tfrac{z^\mu_{*,i}}{x^\mu_{*,i}} \left(\tfrac{1}{1 + \bar \delta}\right)^2  > \tfrac{1}{4} \tfrac{z^\mu_{*,i}}{x^\mu_{*,i}}. \]
			By strict complementarity (Assumption~\ref{as:suff_optcond}) one concludes $\tfrac{z^\mu_{*,i}}{x^\mu_{*,i}}\to\infty$ and therefore $\tfrac{z^\mu_{k,i}}{x^\mu_{k,i}}\to\infty$ for all $i\in\AAA$ as $\mu\to0$.  Then \eqref{eq:quark} implies that there exists $\mu_2\in(0,\mu_1]$ such that $v^TG_k^\mu v \geq\sigma^G_l$ for all $\mu\in(0,\mu_2]$.
			
			The existence of $\tilde\sigma^G_l$ follows from the same arguments, after replacing $H^\mu_k$ by $\tilde H^\mu_k$ and $\lambda_H=\tfrac{\sigma_l^H}{8}$ by $\lambda_H=\tfrac12(\tfrac14\sigma_l^H-\eps_H)$. Observe that  $\lam_H$ is still positive, due to \eqref{eq:tildelamHmin}.
		\end{proof}

		We now present an iteration-dependent version of the first main result \eqref{eq:contract}.  We will make repeated use of the relationship $\|(X_*^\mu)^{-1}V\| = \sigma_{\max}((X_*^\mu)^{-1}V)=\|(X_*^\mu)^{-1}v\|_\infty \leq \|(X_*^\mu)^{-1}v\|=\|v\|_{(X_*^\mu)^{-2}}$ for $v\in\RR_{\geq 0 }^n$.
		
		\begin{lemma}
			\label{lemma.bound_const_prelim_x_neigh}
			Let $\mu\in(0,\mu_2]$ where $\mu_2$ is the threshold from Lemma~\ref{lemma.G_unif_posdef_iter}, and suppose $\bar\delta\in(0,1)$ satisfies \eqref{eq:bardelta0}.
			Then, for $\xk\in\Bbardelta$
			\begin{align}
				\label{eq.neigh_1st_bnd}
				\begin{split}
					e_k^{\mu,+}
					& \leq M_k^{H, \mu} \left\|(G_k^{\mu} X_*^{\mu})^{-1}\right\|  (e^\mu_k)^2\\
					&\qquad +   M_k^{g, \mu} \| (G_k^{\mu}X_*^{\mu})^{-1} - (\tilde G^\mu_k X_*^{\mu})^{-1}\| e^\mu_k
					+ \|(\tilde G_k^{\mu} X_*^{\mu})^{-1} \| \epsilon_g 
				\end{split}
			\end{align}
			with
			\begin{align*}
				M_k^{H, \mu} = \tfrac{L_H \|x_*^{\mu}\|_\infty^2}{2}  + \mu  \|X_*^{\mu} (X^\mu_k)^{-2}\| \ \ \text{ and } \ \ 
				M_{k}^{g,\mu} = L_g \|x_*^{\mu}\|_\infty +  \mu\| (X^\mu_k)^{-1}\|.
			\end{align*}
			
		\end{lemma}
		\begin{proof}
			Consider the {noisy Newton step} \eqref{eq.noisy_d_N}.
			Note that $\tilde G_k^\mu$ and $G_k^\mu$ are nonsingular due to Lemma~\ref{lemma.G_unif_posdef_iter}. With  $\nabla \varphi^{\mu}_* = 0$ this yields
			\begin{equation}
				\label{eq.tmp00}
				\begin{split}
					\xk + d_k  =  & \; \xk  + (\tilde G_k^{\mu})^{-1}\nabla \tilde \varphi_k^{\mu} \\
					= &\; \xk  - (G_k^{\mu})^{-1}\nabla \varphi_k^{\mu} 
					\\ &  + \left( (G_k^{\mu})^{-1} - (\tilde G_k^{\mu})^{-1} \right) (\nabla \varphi_k^{\mu} - \nabla \varphi_*^{\mu})
					+ (\tilde G_k^{\mu})^{-1} \left(\nabla \varphi_k^{\mu} - \nabla \tilde \varphi_k^{\mu}  \right) \\
					=&\; \xk  - (G_k^{\mu})^{-1}(\nabla \varphi_k^{\mu} - \nabla \varphi_*^{\mu}) 
					+ \left( (G_k^{\mu})^{-1} - (\tilde G_k^{\mu})^{-1} \right) \left(\nabla \varphi_k^{\mu} - \nabla \varphi_*^{\mu}  \right)
					\\ &  + (\tilde G_k^{\mu})^{-1} \left(g^\mu_k - \tilde g^\mu_k \right) \\
					=&\; \xk  - (G_k^{\mu})^{-1}\left(g^\mu_k-g_*^\mu -\mu \left((X^\mu_k)^{-1} - (X_*^{\mu})^{-1}\right)\mathbf{e}\right) \\
					& + \left( (G_k^{\mu})^{-1} - (\tilde G_k^{\mu})^{-1} \right) \left(\nabla \varphi_k^{\mu} -\nabla \varphi_*^{\mu}\right) 
					+ (\tilde G_k^{\mu})^{-1} \left(g^\mu_k - \tilde g^\mu_k \right),     
				\end{split}
			\end{equation}
			where the third equality is obtained by the fact that $\nabla \varphi_k^{\mu} - \nabla \tilde \varphi_k^{\mu} = g^\mu_k - \tilde g^\mu_k$. In addition, by Taylor's theorem  we have 
			\begin{align}
				\label{eq.tmp11}
				g^\mu_k - g_*^\mu = \bar H^\mu_k (\xk - x_*^{\mu}),
			\end{align}
			where $
			\bar H^\mu_k := \int_0^1 \nabla^2 f(\xk + t (x_*^{\mu} - \xk)) dt$.
			Combining \eqref{eq.tmp00} and \eqref{eq.tmp11}, and subtracting $x_*^{\mu}$ from both sides of \eqref{eq.tmp00} leads to
			\begin{align*}
				\xk + d_k - x_*^{\mu} 
				=&\;  (G_k^{\mu})^{-1} \left(G_k^{\mu} - \bar H^\mu_k - \mu (X^\mu_k)^{-1} (X_*^{\mu})^{-1}   \right) (\xk - x_*^{\mu}) \\ 
				&+ \left( (G_k^{\mu})^{-1} - (\tilde G_k^{\mu})^{-1} \right) (\nabla \varphi_k^{\mu} - \nabla \varphi_*^{\mu})
				+ (\tilde G_k^{\mu})^{-1} \left(g^\mu_k - \tilde g^\mu_k \right) \\
				=&\; (G_k^{\mu})^{-1} \left(H^\mu_k - \bar H^\mu_k + \mu ((X^\mu_k)^{-2} - (X_k^\mu)^{-1} (X_*^{\mu})^{-1})   \right) (\xk - x_*^{\mu}) \\
				&+ \left( (G_k^{\mu})^{-1} - (\tilde G_k^{\mu})^{-1} \right) (\nabla \varphi_k^{\mu} - \nabla \varphi_*^{\mu})
				+  (\tilde G_k^{\mu})^{-1} \left(g^\mu_k - \tilde g^\mu_k \right).
			\end{align*}
			Multiplying this from the left by $(X_*^{\mu})^{-1}$ and rearranging terms yields
			\begin{align*}
				& (X_*^{\mu})^{-1} (x_{k} + d_k - x_*^{\mu}) \\
				=\ &   (G_k^{\mu} X_*^{\mu})^{-1} \left(  H^\mu_k - \bar H^\mu_k - \mu (X_*^{\mu})^{-1} (X^\mu_k)^{-2} (X^\mu_k - X_*^{\mu}) \right) (\xk - x_*^{\mu})\\
				&+ \left( (G_k^{\mu} X_*^{\mu})^{-1} - (\tilde G_k^{\mu} X_*^{\mu})^{-1} \right) \left(g^\mu_k - g_*^{\mu} + \mu (X_*^{\mu})^{-1} (X_k^\mu)^{-1}(\xk - x_*^{\mu})\right) \\
				&+ (\tilde G_k^{\mu} X_*^{\mu})^{-1} \left(g^\mu_k - \tilde g^\mu_k \right).   
			\end{align*}
			Taking norms from both sides along with Assumptions~\ref{assumption.bnd_err}, \ref{assumption.bnd_const_grad_lipschitz}, and \ref{assumption.bnd_const_hess_lipschitz_1} implies 
			\begin{align*}
				&\;\|x^\mu_{k} + d_k - x_*^{\mu}\|_{(X_*^{\mu})^{-2}}\\
				\leq&\;   \|(G_k^{\mu} X_*^{\mu})^{-1}\| \left\| \left( H^\mu_k - \bar H^\mu_k  - \mu X_*^{\mu} (X_k^\mu)^{-2} (X_k^\mu- X_*^{\mu})  (X_*^{\mu})^{-2}\right)  (\xk - x_*^{\mu}) \right\|\\
				&\ \ + \| (G_k^{\mu} X_*^{\mu})^{-1} - (\tilde G_k^{\mu} X_*^{\mu})^{-1} \|\|g^\mu_k - g_*^{\mu} + \mu (X_*^{\mu})^{-1} (X_k^\mu)^{-1}(\xk - x_*^{\mu})\| \\
				& \ \ + \|(\tilde G_k^{\mu} X_*^{\mu})^{-1}\| \left\|g^\mu_k - \tilde g^\mu_k \right\| \\
				\leq&\; \|(G_k^{\mu} X_*^{\mu})^{-1}\| \left( \tfrac{L_H \|X_*^{\mu}\|^2}{2}  +  \mu \|X_*^{\mu} (X_k^\mu)^{-2}\| \right) \|\xk - x_*^{\mu}\|_{(X_*^{\mu})^{-2}}^2\\
				& \ \ + \| (G_k^{\mu} X_*^{\mu})^{-1} - (\tilde G_k^{\mu} X_*^{\mu})^{-1} \| \left(L_g \|X_*^{\mu}\| + \mu \|(X_k^\mu)^{-1}\| \right) \|\xk - x_*^{\mu}\|_{(X_*^{\mu})^{-2}} \\
				& \ \ + \|(\tilde G_k^{\mu} X_*^{\mu})^{-1}\| \eps_g. 
			\end{align*}
			This completes the proof.
		\end{proof}

			Our next goal is to establish upper bounds for all the terms on the right-hand side of \eqref{eq.neigh_1st_bnd}. We accomplish this with the following lemmas.
			
			\begin{lemma}\label{lem:technical1}
				Let $\mu\in(0,\mu_2]$ with $\mu_2$
				from Lemma~\ref{lemma.G_unif_posdef_iter}, and suppose $\bar\delta\in(0,1)$ satisfies \eqref{eq:bardelta0}.
				Then for $\xk\in\Bbardelta$, 
				$\| G_k^{\mu} X_*^{\mu} -  G_*^{\mu} X_*^{\mu} \| \leq \frac{1}{2 \|(G_*^{\mu} X_*^{\mu})^{-1}\|}$.
			\end{lemma}
			\begin{proof}
				By the definition of $G_k^{\mu}$ and $G_*^{\mu}$, and Assumption~\ref{assumption.bnd_const_hess_lipschitz_1} we have
				\begin{align*}
					&\; \|G_k^{\mu} X_*^{\mu} -  G_*^{\mu} X_*^{\mu} \| \\\leq &\;\| (H^\mu_k  - H_*^\mu) X_*^{\mu} + \mu ((X_k^{\mu})^{-2} - \mu (X_*^\mu)^{-2}) X_*^{\mu}   \| \\
					\leq &\; L_H \|X_*^{\mu}\|^2 \|\xk - x_*^{\mu} \|_{(X_*^{\mu})^{-2}} + \mu  \|(X_*^{\mu})^{-2} (X_k^\mu)^{-2} ((X_*^{\mu})^2 - (X_k^\mu)^2) X_*^{\mu}\| \\
					\leq &\; L_H \|x_*^{\mu}\|_{\infty}^2 \|\xk - x_*^{\mu} \|_{(X_*^{\mu})^{-2}}\\
					&+ \mu \|(X_*^{\mu})^{-2} (X_k^\mu)^{-2} (X_*^{\mu} + X_k^\mu) (X_*^{\mu})^2\|\|\xk - x_*^{\mu}\|_{(X_*^{\mu})^{-2}} \\
					= & \, \left(L_H \|x_*^{\mu}\|_{\infty}^2 + \mu \|
					%
					%
					(X_*^{\mu})^{-1} (X_*^{\mu} (X_k^\mu)^{-1}) (I+X_*^{\mu} (X_k^\mu)^{-1}) \|\right)
					\|\xk - x_*^{\mu}\|_{(X_*^{\mu})^{-2}} \\
					\leq & \, \left(L_H \|x_*^{\mu}\|_{\infty}^2 + \mu \|(X_*^{\mu})^{-1}\|  \| X_*^{\mu} (X_k^\mu)^{-1}\| (1 + \| X_*^{\mu} (X_k^\mu)^{-1} \|)  \right) \|\xk - x_*^{\mu}\|_{(X_*^{\mu})^{-2}}.
				\end{align*}
				Since $\|\xk - x_*^{\mu} \|_{(X_*^{\mu})^{-2}} \leq \bar \delta$ ($\bar \delta$ satisfies \eqref{eq:bardelta}),  Lemma~\ref{lemma.DX_bnds} and $z_*^\mu=\mu(X_*^{\mu})^{-1}\mathbf e$ 
				yield
				\begin{align*}
					\|G_k^\mu X_*^{\mu} -  G_*^\mu X_*^{\mu} \| 
					\leq & \; \left(L_H \|x_*^{\mu}\|_{\infty}^2 +   \tfrac{\|z_*^\mu\|_{\infty}}{1 - \bar \delta} \left(1 + \tfrac{1}{1 - \bar \delta}\right)  \right) \left\|\xk - x_*^{\mu}\right\|_{(X_*^{\mu})^{-2}} \\
					= & \;\tilde L^\mu \|\xk - x_*^{\mu} \|_{(X_*^{\mu})^{-2}} 
					\leq  \tfrac{1}{2 \|(G_*^\mu X_*^{\mu})^{-1}\|} .
				\end{align*}
			\end{proof}

			\begin{lemma} 
				\label{lemma.H_inv_bound}
				Let $\mu_2$ be the threshold from Lemma~\ref{lemma.G_unif_posdef_iter} and suppose $\bar\delta\in(0,1)$ satisfies \eqref{eq:bardelta0}.
				Then, there exists $\mu_3 \in (0, \mu_2]$ so that for all $\xk\in\Bbardelta$ one has
				\[
				\left \|(G_k^\mu X_*^{\mu})^{-1}\right\|\leq 4 \left\| (\Gamma_*^0)^{-1}\right\|.
				\]
			\end{lemma}
			\begin{proof}
				By Assumptions~\ref{assumption.bnd_const_grad_lipschitz} and \ref{assumption.bnd_const_hess_lipschitz_1} and $\mu (X_*^{\mu})^{-1} = Z_*^{\mu}$ one has for all $\mu\in(0,\mu_2]$
				\begin{align}
					\label{eq.t11}
					\begin{split}
						\left\|G_*^{\mu}X_*^{\mu} -  \Gamma_*^0 \right\| &= \left\|H_*^\mu X_*^{\mu} + \mu (X_*^\mu)^{-1} - \left(H_*^0 X_*^0 + Z_*^0\right)\right\| \\
						& = \left \|H_*^\mu X_*^{\mu} - H_*^\mu X_*^0 + H_*^\mu X_*^0 - H_*^0 X_*^0 + Z_*^{\mu} - Z_*^0 \right\| \\
						& \leq \|H_*^{\mu} \| \|x_*^{\mu} - x_*^0 \| + L_H \|x_*^0\|_{\infty}  \|x_*^{\mu} - x_*^0 \| + \|z_*^{\mu} - z_*^0 \| \\
						& \leq \max \left \{ L_g + L_H \|x_*^0 \|_{\infty}, 1\right\} \left(  \|x_*^{\mu} - x_*^0 \| + \|z_*^{\mu} - z_*^0 \| \right),
					\end{split}
				\end{align}
				and by continuity of $(x_*^{\mu},z_*^{\mu})$ as a function of $\mu$
				\begin{equation}
					\label{eq.t22}
					\|x_*^{\mu} - x_*^0 \| + \|z_*^{\mu} - z_*^0 \| \leq \tfrac{1}{2\max \left \{ L_g + L_H \|x_*^0 \|_{\infty}, 1\right\} \|(\Gamma_*^0)^{-1} \|},
				\end{equation}
				for all $\mu\in(0,\mu_3]$ with $\mu_3\in(0,\mu_2]$ sufficiently small.
				Fixing $\mu\in(0,\mu_3]$ and combining \eqref{eq.t11} and \eqref{eq.t22} yields
				$\left\|G_*^{\mu}X_*^{\mu} -  \Gamma_*^0 \right\| \leq \tfrac{1}{2 \|(\Gamma_*^0)^{-1} \|}$.
				By Lemma~\ref{lemma:CT_Kelley} one has
				\begin{align}
					\label{eq.t33}
					\left\|(G_*^{\mu}X_*^{\mu})^{-1}\right\| \leq 2 \left\|(\Gamma_*^{0})^{-1}\right\|.
				\end{align}
				By Lemma~\ref{lem:technical1}, we have $ \| G_k^{\mu} X_*^{\mu} -  G_*^{\mu} X_*^{\mu} \| \leq \tfrac{1}{2 \|(G_*^{\mu} X_*^{\mu})^{-1}\|}$. This, along with Lemma~\ref{lemma:CT_Kelley} implies $\|(G_k^{\mu} X_*^{\mu})^{-1}\| \leq 2\| (G_*^{\mu} X_*^{\mu})^{-1} \|$ and by \eqref{eq.t33} the desired result follows.
			\end{proof}

			\begin{lemma}
				\label{lemma.H_inv_prelim_bounds}
				Let $\mu\in(0,\mu_3]$ where $\mu_3$ is the threshold from Lemma~\ref{lemma.H_inv_bound}, and suppose $\bar\delta\in(0,1)$ satisfies \eqref{eq:bardelta0}.
				Then for $\xk\in\Bbardelta$
				\begin{align*}
					\| (G_k^{\mu} X_*^{\mu})^{-1}  - ( \tilde G_k^{\mu} X_*^{\mu})^{-1}\| \leq&\; 2  \| (G^\mu_k X_*^{\mu})^{-1}\|^2 {\|x_*^{\mu}\|_{\infty}} \eps_H, \\
					\|(\tilde G_k^{\mu} X_*^{\mu})^{-1} \| 
					\leq &\; 2\|(G_k^{\mu} X_*^{\mu})^{-1} \|.
				\end{align*}
			\end{lemma}
			\begin{proof}
				By Lemmas~\ref{lemma:CT_Kelley} and \ref{lem:technical1}, one concludes that
				{$\|(G_k^{\mu} X_*^{\mu})^{-1} \| \leq  2 \|(G_*^{\mu} X_*^{\mu})^{-1}\|$}.
				Hence, by \eqref{eq:epsH} we obtain
				\begin{align*}
					\|G_k^{\mu} X_*^{\mu} -  \tilde G_k^{\mu} X_*^{\mu}\| =&\; \| H^\mu_k -  \tilde H^\mu_k \| \|X_*^{\mu}\|
					\leq\;  \|x_*^{\mu}\|_{\infty}\eps_H \\
					\leq& \;\tfrac{1}{4 \|(G_*^{\mu} X_*^{\mu})^{-1}\|} \leq
					\tfrac{1}{2 \|(G_k^{\mu} X_*^{\mu})^{-1}\|},
				\end{align*}
				and by Lemma~\ref{lemma:CT_Kelley},  we get the desired results.
			\end{proof}

			\begin{theorem}
				\label{thm_main_neigh_1b}
				Let $\mu_3$ be the threshold from Lemma~\ref{lemma.H_inv_bound} and suppose $\bar\delta\in(0,1)$ satisfies \eqref{eq:bardelta0}.
				Furthermore, let $\xi_M \in (1, \infty)$ and define
				\begin{align}
					\begin{split}\label{eq:Mconstants}
						M_{2} =&\; \xi_M \cdot 4\|(\Gamma_*^0)^{-1}\| \left( \tfrac{L_H \|x_*^{0}\|_{\infty}^2}{2} + \tfrac{\|z^0_*\|_{\infty}}{(1 - \bar \delta)^2} \right), \\
						M_{1} =&\; \xi_M \cdot 32\|(\Gamma_*^0)^{-1}\|^2\|x_*^{0}\|_{\infty} \left(L_g \|x_*^{0} \|_{\infty} + \tfrac{\|z^0_*\|_{\infty}}{1 -\bar \delta} \right),\\
						M_{0} = &\; 8\|(\Gamma_*^0)^{-1}\|. 
					\end{split}
				\end{align}
				Then there exists $\bar\mu\in(0,\mu_3]$ so that for all $\mu\in(0,\bar\mu]$
				\begin{equation}\label{eq:ekp}
					e_k^{\mu,+}
					\leq M_{2} (e^\mu_k)^2
					+  M_{1} \eps_H e^\mu_k
					+ M_{0} \epsilon_g,
				\end{equation}
				whenever $\xk\in\Bbardelta$.
			\end{theorem}
	
			\begin{proof}
				By Lemmas~\ref{lemma.bound_const_prelim_x_neigh} and \ref{lemma.H_inv_prelim_bounds} one has
				\begin{align*}
					e_{k}^+
					\leq M_k^{H, \mu} \left\|(G_k^{\mu} X_*^{\mu})^{-1}\right\| (e^\mu_k)^2
					+  2 M_k^{g, \mu} \left\|(G_k^{\mu} X_*^{\mu})^{-1}\right\|^2 \|x_*^{\mu}\|_{\infty} \eps_H e^\mu_k 
					+ 2 \|( G_k^{\mu} X_*^{\mu})^{-1} \| \epsilon_g.  
				\end{align*}
				Lemmas~\ref{lemma.DX_bnds} and \ref{lemma.H_inv_bound}, together with \eqref{eq:epsH} and \eqref{eq:bardelta}, imply that \eqref{eq:ekp} holds with
				\begin{align*}
					M^0_{2} =&\; 4\|(\Gamma_*^0)^{-1}\| \left( \tfrac{L_H \|x_*^{\mu}\|_{\infty}^2}{2} + \tfrac{\|z^\mu_*\|_{\infty}}{(1 - \bar \delta)^2} \right), \\
					M^0_{1} =&\; 32\|(\Gamma_*^0)^{-1}\|^2\|x_*^{\mu}\|_{\infty} \left(L_g \|x_*^{\mu} \|_{\infty} + \tfrac{\|z^\mu_*\|_{\infty}}{1 -\bar \delta} \right),\\
					M^0_{0} = &\; 8\|(\Gamma_*^0)^{-1}\|.
				\end{align*}
				Since $x_*^\mu$ and $z_*^\mu$ are continuous due to Theorem~\ref{thm:central_path} and since $\xi_M>1$, there exists $\bar\mu\in(0,\mu_3]$ so that \eqref{eq:ekp} holds for all $\mu\in(0,\bar\mu]$ with the constants defined in \eqref{eq:Mconstants}.
			\end{proof}
			
			
			\subsection{Convergence into neighborhoods around the optimal solution}
			\label{sec:local_neighborhood}
			
			We first prove a uniform lower bound on the step size, to ensure sufficient progress is made in every iteration.
			\begin{lemma}\label{lem:alpha_lb}
				Let $\mu>0$.
				Then there exists $\bar \alpha^\mu >0$, such that
				$\alpha_k\geq\bar\alpha^\mu/2$ for all $k \in \NN$, where $\alpha_k$ is step size sequence generated in Steps~\ref{s:armijo_begin}--\ref{s:armijo_end} when Algorithm~\ref{alg.bnd_const_fixed_mu} is executed for fixed $\mu$.
			\end{lemma}
			\begin{proof}
				Let $k\in\NN$.  As in the proof of Lemma~\ref{lemma.non-noisy-bnd-iterations} we can derive \eqref{eq.dummy_taylor} where $\breve G_k^{\mu} = \breve H^\mu_k + \tfrac{2\mu}{1 - \tau} (X_k^\mu)^{-2} $.  Then, using Assumption~\ref{assumption.bnd_err}, \eqref{eq.dummy_taylor}, and Lemma~\ref{lemma.lb_on_x_bnd_const}, we have 
				\begin{align}
					\label{eq.tmp3}
					\begin{split}
						&\ \tilde \varphi^{\mu}(\xk + \alpha_k d_k) -  \tilde \varphi^{\mu}(\xk)  -  \nu \alpha_k  (\nabla\tilde \varphi^{\mu}_k)^T d_k- \eps_R \\  \leq & \ \varphi^{\mu}(\xk + \alpha_k d_k) -  \varphi^{\mu}(\xk)  
						-  \nu \alpha_k (\nabla \varphi^{\mu}_k)^T  d_k
						+ \nu \alpha_k \|d_k\| \eps_g + 2 \eps_f - \eps_R \\
						\leq & \  (1 - \nu)\alpha_k (\nabla \varphi^{\mu}_k)^T  d_k+ \tfrac{\alpha_k^2}{2} d_k^T \breve G_k^{\mu} d_k +\nu  \alpha_k\|d_k\| \eps_g + 2 \eps_f - \eps_R \\
						\leq & \;  (1 - \nu)\alpha_k \|\nabla \varphi^{\mu}_k\| \|d_k\| + \tfrac{\alpha_k^2}{2} \left(L_g + \tfrac{2 \mu}{(1 - \tau) \delta_x^2} \right) \|d_k\|^2  +\nu \alpha_k \|d_k\| \eps_g + 2 \eps_f - \eps_R.     
					\end{split}
				\end{align}
				To examine the right-hand side, we define the convex quadratic function $\phi(\alpha) = a \alpha^2 + b \alpha - c$ with
				\begin{align*}
					a :=&\; \tfrac{1}{2}\left(L_g + \tfrac{2 \mu}{(1 - \tau) \delta_x^2} \right) \|d_k\|^2  > 0, \\
					b :=&\; \left((1 - \nu)  \|\nabla \varphi^{\mu}_k\|  +\nu \eps_g \right) \|d_k\| > 0,\\
					c:=&\; \eps_R - 2 \eps_f > 0,
				\end{align*}
				and let $\bar\alpha_k>0$ be the positive root of $\phi(\alpha)$.  Then by \eqref{eq.tmp3}, the Armijo condition \eqref{eq.relaxed_armijo_bnd} is satisfied when $\phi(\alpha_k)\leq0$, which is the case when $\alpha_k\in(0,\bar\alpha_k]$.  From Steps~\ref{s:armijo_begin}--\ref{s:armijo_end} we can then conclude that $\alpha_k\geq\bar\alpha_k/2$.  Because $a$ and $b$ are uniformly bounded by Lemmas~\ref{lemma.lb_on_x_bnd_const} and \ref{lemma.bnd_d} for all $k$, $\bar\alpha_k$ is uniformly bounded away from 0 (see Lemma~\ref{lemma.unif_lb_on_alpha} in \appB).  It follows that $\alpha_k\geq\inf_{k'\in\NN}\bar\alpha_{k'}=:\bar\alpha^{\mu}>0$.
			\end{proof}

		We are now ready to state the main theorem of this section.
		\begin{theorem}
			\label{thm.neigh2}
			Let $\bar\mu$ be the threshold from Lemma~\ref{lemma.H_inv_bound}, and suppose $\bar\delta\in(0,1)$ satisfies \eqref{eq:bardelta0}.
			Further suppose that $\eps_g$ and $\eps_H$, in addition to \eqref{eq:epsHstrict}, satisfy
			\begin{align}
				\label{eq.err_condition_thm}
				\begin{split}
					\eps_H < \tfrac{1}{M_1},  \ \ \text{and} \ \
					\eps_g \leq  \tfrac{ (1 - M_1 \eps_H)^2}{4 M_0 M_2},
				\end{split}
			\end{align}
			where $M_0$, $M_1$, and $M_2$ are defined in \eqref{eq:Mconstants} (and depend on $\bar\delta$). 
			Further, let
			\begin{align}
				\label{eq.chi1_chi2}
				\begin{split}
					\delta^+, \delta^- =  \tfrac{(1 - M_1 \eps_H) \pm \sqrt{\Delta}}{2 M_2}, \ \ \text{with} \ \   \Delta = (1 - M_1 \eps_H)^2 - 4 M_0 M_2 \eps_g,
				\end{split} 
			\end{align}
			and $\delta_1 = \delta^-$
			and $\delta_2=\min\{\delta^+,\bar\delta\}$ and assume $\delta_1\leq\bar\delta$.
			Then, for all $\mu\in(0,\bar\mu)$,
			\begin{enumerate}[label=(\roman*)]
				\item \label{c.1} $\xk \in B_{\delta_2}(x_*^{\mu}) \setminus \bar B_{\delta_1}(x_*^{\mu})$ implies $\|x_{k+1}^\mu - x_*^{\mu}\|_{(X_*^{\mu})^{-2}}  < \|\xk - x_*^{\mu}\|_{(X_*^{\mu})^{-2}} $;
				
				\item \label{c.2} $\xk\in \bar B_{\delta_1}(x_*^{\mu})$ implies $x_{k+1}^\mu \in \bar  B_{\delta_1}(x_*^{\mu})$; and
				
				\item \label{c.3} $\limsup_{k \to \infty} \|\xk -x_*^\mu\|_{(X_*^{\mu})^{-2}} \leq \delta_1$.
			\end{enumerate}
		\end{theorem}
		\begin{proof}
			Let $\mu\in(0,\mu_3]$ and let $\phi(e) = M_2\cdot e^2 + M_1\eps_H\cdot e +M_0\eps_g$. 
			The assumptions on $\eps_g$ and $\eps_H$ imply $\Delta,\delta^-,\delta^+ >0$.  If $\xk\in B_{\delta_2}(x_*^{\mu})$, then $e^{\mu}_k\leq \delta_2\leq\bar\delta$ and Theorem~\ref{thm_main_neigh_1b} implies $e_{k}^{\mu,+}\leq\phi(e^\mu_k)$.
			Hence
			\begin{equation}\label{eq:ekp1}
				\begin{split}
					e^\mu_{k+1} = &\; \|x_{k+1}^\mu-x_*^{\mu}\|_{(X_*^{\mu})^{-2}} =\|\xk + \alpha_kd_k-x_*^{\mu}\|_{(X_*^{\mu})^{-2}} \\ = &\;  \|(1-\alpha_k)(\xk - x^\mu_*)  + \alpha_k(\xk + d_k-x_*^{\mu})\|_{(X_*^{\mu})^{-2}} \\ 
					\leq &\; (1-\alpha_k)e^\mu_k + \alpha_k e_k^{\mu,+} \leq e^\mu_k + \alpha_k (\phi(e^\mu_k)-e^\mu_k ).
				\end{split}
			\end{equation}
			Since $\delta^-$ and $\delta^+$ are the roots of $\psi(\delta)=\phi(\delta)-\delta$ and $\psi$ is a strictly convex quadratic function, we have that $\psi(e^\mu_k)<0$ when $\delta^-<e^\mu_k<\min\{\delta^+,\bar\delta\}=\delta_2$, and consequently $\phi(e^\mu_k)=\psi(e^\mu_k)+e^\mu_k<e^\mu_k$.
			Since $\alpha_k>0$, \eqref{eq:ekp1} yields claim \ref{c.1}.
			%
			
			To show claim~\ref{c.2}, 
			we first prove that $M_0 \eps_g \leq \delta_1 = \delta^-$. By the definition of $\delta^-$ in \eqref{eq.chi1_chi2}, this is equivalent to showing $ (1 - M_1 \eps_H) - 2M_0 M_2 \eps_g\geq  \sqrt{\Delta}$.
			Observe that $(1 - M_1 \eps_H) - 2M_0 M_2 \eps_g$ is non-negative since $ \eps_g \leq \tfrac{(1 - M_1 \eps_H)^2}{4M_0 M_2} \leq \tfrac{1 - M_1 \eps_H}{2M_0 M_2} $ by \eqref{eq.err_condition_thm} and $(1-M_1 \eps_H)\in(0,1)$.   Thus, it is sufficient to show    $ \left((1 - M_1 \eps_H) - 2M_0 M_2 \eps_g \right)^2\geq  {\Delta}$. Expanding the quadratic on the left-hand side of this inequality and by the definition of $\Delta$ in \eqref{eq.chi1_chi2}, one has that this inequality trivially holds, as $M_0 M_2 \eps_g + M_1 \eps_H \geq 0$.
			
			Now consider $\xk\in \bar B_{\delta_1}(x_*^{\mu})$.
			Because $e_{k}\in[0,\delta_1]$, we have $\phi(e^\mu_k)\leq\phi(0)=M_0\eps_g\leq\delta_1$ since $\phi$ is monotonically decreasing over $[0,\delta^-]=[0,\delta_1]$ (recall $\delta_1=\delta^-$ by definition).  Together with \eqref{eq:ekp1}, this concludes the proof of claim \ref{c.2}.
		
				Finally, claim \ref{c.3} follows from part \ref{c.2} if $x_{k'}\in \bar B_{\delta_1}(x_*^{\mu})$ for some $k'$.
				For the other case, suppose $e_{k'}>\delta_1$ for all $k'$ and let $x_{k}\in B_{\delta_2}(x_*^{\mu})$.  Part \ref{c.1} then yields that $e_{k'+1}<e_{k'}$ for all $k'\geq k$.
				Choose an arbitrary $\eta\in(0,-\psi(e_{k}))$ and recall that $\delta_1=\delta^-$ and $\psi(e_{k})<0$ since $e_{k}\in(\delta^-,\delta^+)$.
				Let $\delta^-_{\eta}$and $\delta^+_{\eta}$ be the roots of $\psi(\delta)+\eta=0$. Then $\psi(\delta)+\eta<0$ for all $\delta\in(\delta^-_{\eta},\delta^+_{\eta})$.  Consequently, as long as $e^\mu_k>\delta_\eta^-$, we have $\phi(e^\mu_k)=\psi(e^\mu_k) + e^\mu_k < e^\mu_k-\eta$.  Substituting this into \eqref{eq:ekp1}, we have
				\[ e_{k+1} < (1-\alpha_k)e^\mu_k + \alpha_k(e^\mu_k-\eta)=e^\mu_k-\alpha_k\eta\leq e^\mu_k-\bar \alpha^\mu\eta,
				\]
				where $\bar \alpha^\mu>0$ is the constant defined in Lemma~\ref{lem:alpha_lb}.
				This shows that the error is decreased by at least $\bar\alpha^\mu\eta$ in each iteration as long as $e_{k}>\delta^-_\eta$.  This, combined with the fact that $e_{k'+1}<e_{k'}$ for all $k'\geq k$, yields that $e_{k'}\leq\delta^-_\eta$ for all $k'$ sufficiently large.
				To finish the proof note that the conclusion above holds for any $\eta\in(0,-\psi(e_{k}))$, and since $\delta^-_\eta\to\delta^-$ as $\eta\to0$, we obtain the desired result.    
			\end{proof}
			
			\begin{remark}\label{rem:interpretation}
				The interpretation of this result is that, if the noise levels are sufficiently small, there exists a neighborhood $B_{\delta_1}(x_*^\mu)$ around the solution from which the iterates cannot escape once they are inside.  Further, there is a larger neighborhood of attraction, $B_{\delta_2}(x_*^\mu)$, with the property that once iterates fall into this neighborhood, they converge monotonically towards $B_{\delta_1}(x_*^\mu)$ or enter $B_{\delta_1}(x_*^\mu)$.
				The condition on $\epsilon_g$ in \eqref{eq.err_condition_thm} shows a trade-off between the permissible noise in the gradient and the Hessian, where $\eps_g$ needs to be decreased if $\eps_H$ increases (see Figure~\ref{eq.ub_lb_chi_bar} below).
			\end{remark}

			\begin{remark}\label{rem:extreme_delta}
				Theorem~\ref{thm.neigh2} is valid for an entire range of $\bar\delta$.  
				Notice that $M_0$, $M_1$, and $M_2$ and consequently $\delta_1$ and $\delta_2$ depend on $\bar\delta$.
				Therefore, we can derive the largest ``attraction-neighborhood'' $\bar B_{\delta_2^{\max}}(x_*^{\mu})$ from Theorem~\ref{thm.neigh2} by looking for the largest value $\delta_2^{\max}$ of $\delta_2$ among those for which $\bar\delta$ satisfies the assumptions of the theorem.
				Similarly, defining $\delta_1^{\min}$ as the smallest possible value of $\delta_1$, we can find the smallest ``containment-neighborhood'' $\bar B_{\delta^{\min}_1}(x_*^{\mu})$.  We discuss an example of this in Section~\ref{sec:choicedeltabar}.
			\end{remark}
			
			The following is the key observation of this section.
			\begin{remark}[Active-set identification]\label{rem:activeset}
				In the non-noisy case ($\eps_f=\eps_g=\eps_H=0$), the interior-point framework obtains a solution of the original problem \eqref{eq.orig_bnd_const_prob} by approximately solving barrier problems \eqref{eq.barrier_bnd_const} for a decreasing sequence $\{\mu_\ell\}$ of barrier parameters with $\mu_\ell\to0$. Then, for $i\in\AAA$, the iterates $x^{\mu_\ell}_{k,i}$ converge to zero when $\mu_\ell\to0$ and Algorithm~\ref{alg.bnd_const_fixed_mu} is executed for sufficiently many iterations for each fixed value of $\mu_\ell$.  In other words, the interior-point method is able to identify which bounds are active at the optimal solution.  One may suspect that this is not the case when $\eps_f,\eps_g,\eps_H>0$ because the iterates $x^{\mu_\ell}_{k,i}$ might keep changing significantly even for small $\mu_\ell$ because of the noise.  However, if for every $\mu_\ell$, for which Algorithm~\ref{alg.bnd_const_fixed_mu} is executed, eventually $x_{{k}}^{\mu_\ell}\in B_{\delta_2}(x_*^{\mu_\ell})$ for a sufficiently large $k$, then
				$
				\tfrac{|x^{\mu_\ell}_{{k},i}-x^\mu_{*,i}|}{x^{\mu_\ell}_{*,i} } \leq \|(X_*^{\mu_\ell})^{-1}(x^{\mu_\ell}_{{k}}-x^{\mu_\ell}_{*})\| = \|x^{\mu_\ell}_{{k}}-x^{\mu_\ell}_{*}\|_{(X_*^{\mu_\ell})^{-2}} < \delta_2,
				$
				or equivalently $|x^{\mu_\ell}_{k,i}-x^{\mu_\ell}_{*,i}| < \delta_2 x^{\mu_\ell}_{*,i}$ for all $i \in \AAA $ and all large $k$.
				Since $\lim_{\mu\to0}x^{\mu}_{*,i}=0$ this implies that, as in the non-noisy case, $x^{\mu_\ell}_{k,i}$ converge to zero when $\mu_\ell\to0$. In other words, the algorithm is able to identify the set of active constraints even in the presence of noise; see also Figure~\ref{fig.neigh_x_y}.
			\end{remark}

			\subsection{Illustrative example}\label{sec:choicedeltabar}
	To illustrate the implications of Remark~\ref{rem:extreme_delta}, consider the optimization problem 
			\begin{equation*}
				\min_{x \in \RR^n}\ \tfrac{1}{2}\left(x_1 - 1\right)^2 + x_2\ \ \st \ \ x \geq 0.
			\end{equation*}
			The primal-dual path, in terms of $\mu \in [0, 1]$, is
			$
			x_*^{\mu } = \left(
			\tfrac{1 + \sqrt{1 + 4 \mu}}{2},  \, \mu\right)^T
			\text{ and }
			z_*^{\mu } = \left(
			\tfrac{2\mu}{1 + \sqrt{1 + 4 \mu}}, \, 1\right)^T.
			$
			Observe that $\|z_*^{\mu}\|_{\infty} = 1$. In addition,  $g(x) =
			(    x_1 -1, 
			1)^T
			$
			and 
			$H(x) = \begin{bmatrix}
				1 & 0 \\ 
				0 & 0
			\end{bmatrix}$ yield $L_g = 1$ and $L_H = 0$, respectively.
			Also, 
			\begin{align*}
				G_k^\mu X_*^\mu & = \begin{bmatrix}
					\left(1 + \frac{\mu}{(x_{k, 1}^\mu)^2}\right)x_{*, 1}^{\mu} & 0 \\
					0 & \frac{\mu}{(x_{k, 2}^\mu)^2} x_{*, 2}^\mu
				\end{bmatrix}, \text{ and }\\
				\|  (G_k^\mu X_*^\mu)^{-1}\| & = \max \left\{{\left(x_{*, 1}^{\mu} + \tfrac{\mu}{x_{*, 1}^{\mu}}\left(\tfrac{x_{*, 1}^{\mu}}{x_{k, 1}^\mu}\right)^2\right)^{-1}}, \tfrac{x_{*, 2}^{\mu}}{\mu}\left(\tfrac{x_{k, 2}^\mu}{x_{*, 2}^{\mu}}\right)^2
				\right\}.
			\end{align*}
			Since $x_{*, 1}^{\mu} > 1$, the first term in the $\max$-term is less than $1$.
			Also, by Lemma~\ref{lemma.DX_bnds}, one has that $\left(\nicefrac{x_{k, 2}^\mu}{x_{*, 2}^{\mu}}\right)^2 \leq (1 + e^\mu_k)^2$. This, along with $\frac{x_{*, 2}^{\mu}}{\mu} = 1$, implies that $  \|  (G_k^\mu X_*^\mu)^{-1}\| \leq (1 + e^\mu_k)^2 \leq (1 + \bar \delta)^2$. Replacing $4\|(\Gamma_*^0)^{-1}\|$ with $ \|  (G_k^\mu X_*^\mu)^{-1}\|$ in \eqref{eq:Mconstants} yields a tighter bound on $e_k^+$ and the constants $M_2$, $M_1$, and $M_0$ for this example are 
			\begin{align*}
				M_{2} = \left( \tfrac{1 + \bar\delta}{1 - \bar\delta} \right)^2, \ \ 
				M_{1} = 2 (1 + \bar\delta)^4 \left(\tfrac{1 + \sqrt{1 + 4 \mu}}{2}\right)  \left(\left(\tfrac{1 + \sqrt{1 + 4 \mu}}{2}\right) + \tfrac{1}{1 -\bar \delta} \right), \ \ 
				M_{0} =  2 (1 + \bar\delta)^2.
			\end{align*}
			Now consider $\eps_g = 0.02$ and $\eps_H = 0.01$ with $\mu = 10^{-6}$.
			Figure~\ref{fig.ub_lb_delta_bar} shows $\delta^-$ and $\delta^+$ as functions of $\bar \delta$. For visualization purposes, we also include a plot of the \emph{identity line} (with a slope of one), denoted by $\bar \delta$ in the figure legend.
			Note that $\delta^-$ and $\delta^+$ are monotonically increasing and decreasing, respectively, as $\bar\delta$ increases.
			Consequently, the radii in Remark~\ref{rem:extreme_delta} are $\delta_1^{\min} \approx 0.05$ and $\delta_2^{\max} \approx 0.24$, computed as the roots of $\delta^- = \bar \delta$ and $\delta^+ = \bar \delta$, respectively.

			Figure~\ref{fig.neigh_x_y} shows the neighborhood $\bar B_{\delta^{\max}_2}(x_*^{\mu})$ in the $(x_1, x_2)$-space and is a visualization for Remark~\ref{rem:activeset}.  Observe that the neighborhood is much narrower in the $x_2$- than in the $x_1$-direction. In addition, as $\mu$ decreases the neighborhood becomes even narrower in the $x_2$-direction. This implies that late iterates $x_{k, 2}^{\mu}$ remain closer to $x_{*, 2}^0=0$ as $\mu$ decreases despite the presence of noise in evaluating the objective function and its gradient. This observation suggests that decreasing the barrier parameter $\mu$ beyond some potential noise-induced threshold is beneficial in obtaining a more accurate solution and identifying the active constraints.
			In Table~\ref{tb.active_test_ndeg} below we see that this is observed also when the algorithm is applied to general problem instances.
			
			Finally, we want to show that Theorem~\ref{thm.neigh2} applies to a wide range of noise thresholds $\eps_H$ and $\eps_g$.
			Figure~\ref{eq.ub_lb_chi_bar} depicts a plot of $\delta_1^{\min}$ and $\delta_2^{\max}$ as a function of $\eps_g$ and $\eps_H$ for $\mu = 10^{-6}$. 
			The domain in Figure~\ref{subfiga} and \ref{subfigb} corresponds to the pairs $(\eps_g,\eps_H)$ for which the assumptions of Theorem~\ref{thm.neigh2} are satisfied, i.e., $\Delta\geq0$ in \eqref{eq.chi1_chi2}.
			As discussed in Remark~\ref{rem:interpretation}, \eqref{eq.err_condition_thm} indicates a trade-off in the permissible values of $\eps_H$ and $\eps_g$: When $\eps_g$ is large, $\eps_H$ needs to be small, and vice versa. 
			Also, observe that $\delta_2^{\max}$ increases and $\delta_1^{\min}$ decreases as $\eps_g$ and $\eps_H$ decrease. In particular, $\eps_g = \eps_H = 0$ yields $\delta_{1}^{\min} = 0$, i.e., $\bar B_{\delta^{\min}_1}(x_*^{\mu}) =\{x_*^{\mu}\}$, which implies that $\xk \to x_*^{\mu}$ as $k \to \infty$ by part \ref{c.3} in Theorem~\ref{thm_main_neigh_1b}, as one expects in the non-noisy case.\

			\begin{figure}
				\centering
				\begin{minipage}{.5\textwidth}
					\centering
					\includegraphics[trim={0cm 0cm 0cm 0.0cm},clip, width=\textwidth]{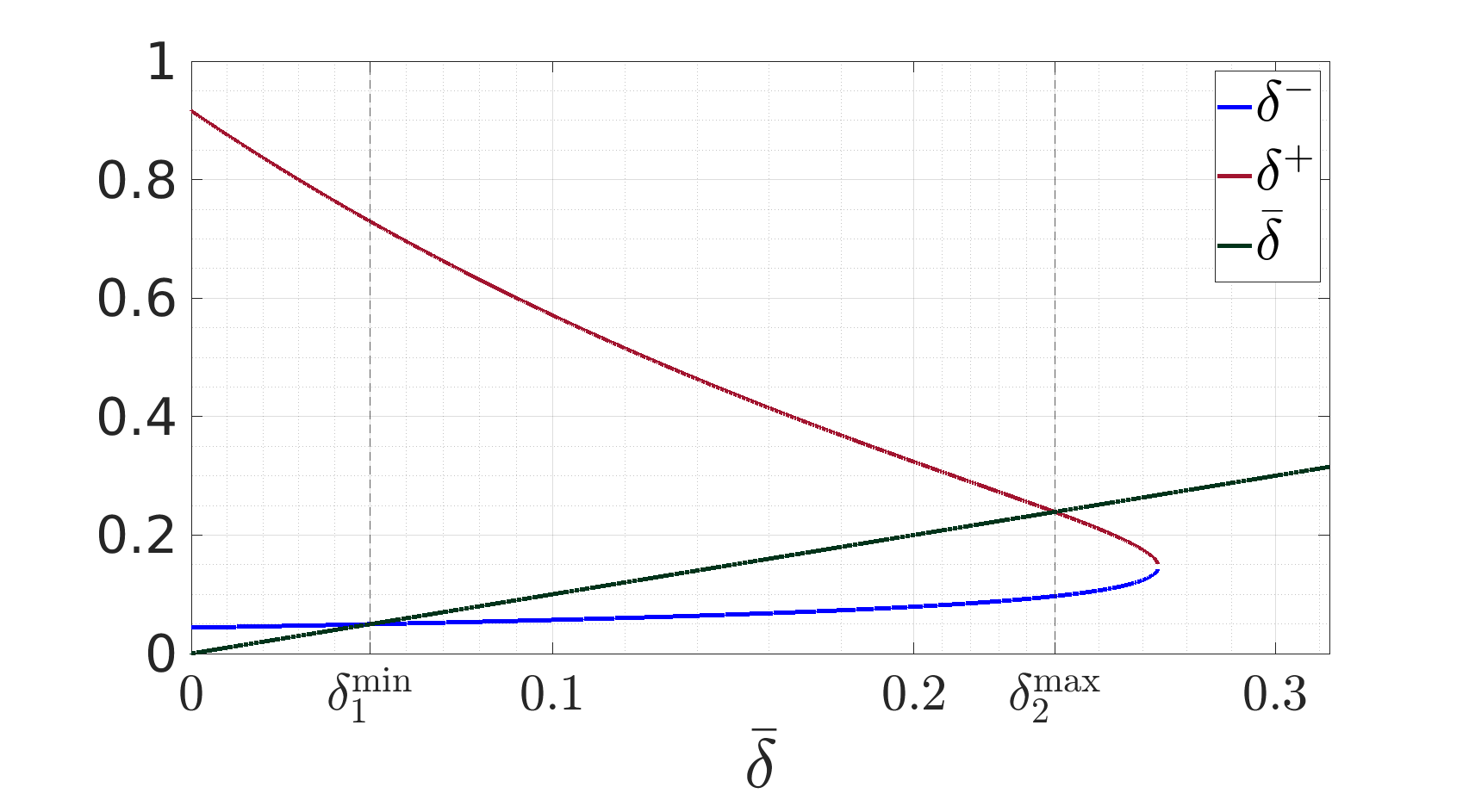}
					\caption{$\delta^-$ and $\delta^+$ as a function of $\bar \delta$.}
					\label{fig.ub_lb_delta_bar}
				\end{minipage}%
				\begin{minipage}{.5\textwidth}
					\centering
					\includegraphics[trim={0cm 0cm 2cm 0cm},clip, width=\textwidth]{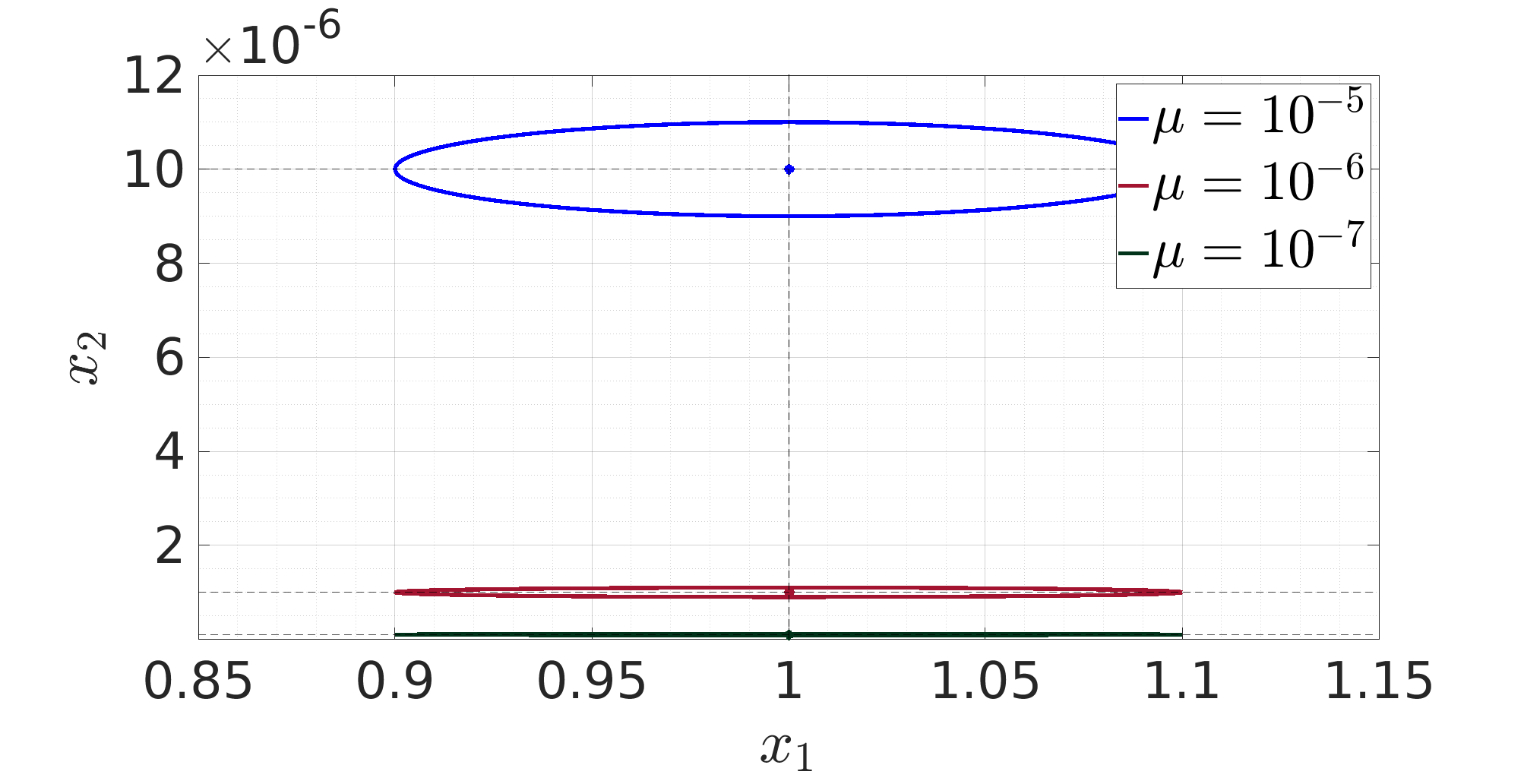}	
					\caption{Illustration of $\bar B_{\delta_2}(x_*^{\mu})$.}
					\label{fig.neigh_x_y}
				\end{minipage}
			\end{figure}

			\begin{figure}\centering
				\subfloat[Color map of $\delta_1^{\min}$]{\label{subfiga}\includegraphics[trim={.5cm 0cm .5cm 0cm},clip, width=.29\linewidth]{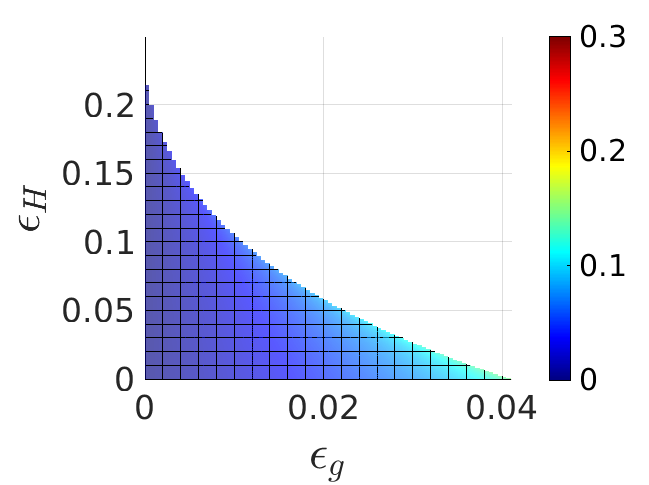}}\hfill
				\subfloat[Color map of $\delta_2^{\max}$]{\label{subfigb}\includegraphics[trim={.5cm 0cm .5cm 0cm},clip, width=.29\linewidth]{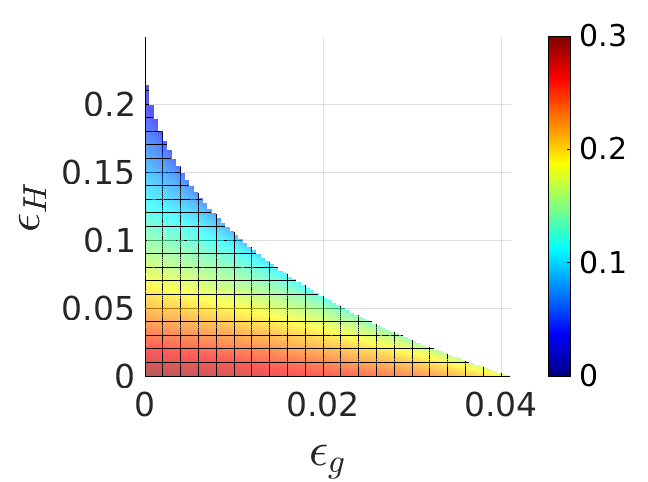}}\hfill
				\subfloat[3-dim view]{\label{c}\includegraphics[trim={0cm .2cm .5cm 0cm},clip,width=.29\linewidth]{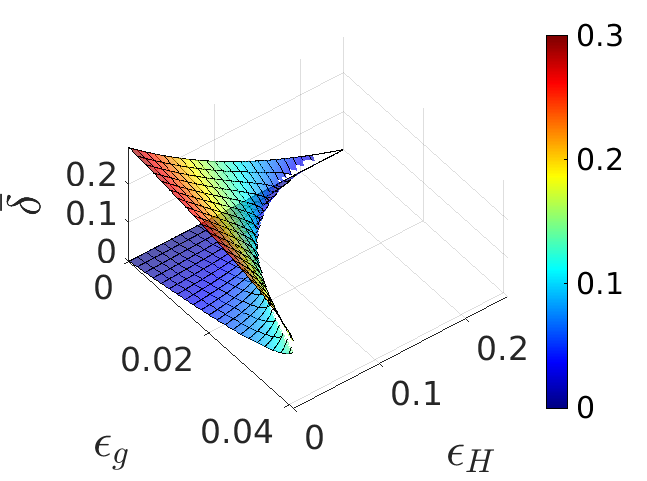}}
				\caption{Dependence of $\delta_1^{\min}$ and $\delta_2^{\max}$ as a function of $\eps_g$ and $\eps_H$.}
				\label{eq.ub_lb_chi_bar}
			\end{figure}

			\section{Numerical experiments}
			\label{sec.numerical_exp}
			In this section, we present numerical results that show that the theoretical findings in Sections~\ref{sec.global_conv} and \ref{sec.local_conv} can be observed in practice.
			
			We performed the experiments in Matlab 2020b with a primal-dual implementation of Algorithm~\ref{alg.bnd_const_fixed_mu} with $\nu=10^{-6}$ and $\eps_R=2.05\eps_f$. The primal-dual Hessian, based on \eqref{eq:G_primaldual}, is used in computing the search direction \eqref{eq.direction_bnd}. Algorithm~\ref{alg.bnd_const_update_barrier} in \appE \ is a formal description of the implemented method.
			We use the primal-dual variant because, as is well known, it significantly outperforms the primal method in terms of the number of iterations needed to converge in the non-noisy case.  In preliminary tests, we observed the same advantage in the presence of noise.  As discussed
			in Section~\ref{sec.global_conv}, the global convergence results still hold for this variant.

			To handle nonconvexity, we follow an approach similar to that in \cite{nocedal1999numerical} and choose $\hat G_k^\mu=\hat H_k^\mu + Z_k^\mu (X_k^\mu)^{-1}+\lam I$ for some $\lam\geq0$ to make sure that the positive-definite condition \eqref{eq:sigma_l_G} holds.
			We observed in our experiments that for small values of $\mu$, the regularization eventually chooses $\lam=0$ also for nonconvex instances, providing evidence for our theoretical result in Lemma~\ref{lemma.G_unif_posdef_iter}.

			The problems considered in this section are selected from the CUTE (Constrained and Unconstrained Testing Environment) problem collection \cite{bongartz1995cute}. 
			We chose the bound-constrained instances for which the AMPL models are available \cite{vanderbei_ampl}\footnote{We divided the objective functions in problems \texttt{ncvxbqp*} by $n$.}. 
			We omitted problems in which no constraint was active at the optimal solution, as they are not relevant to verifying the theoretical results of identifying active constraints in Section~\ref{sec.local_conv}.  The sizes of the problems vary between $n=2$ and $n=1000$ variables.

			In all experiments, we added Bernoulli-distributed noise with parameter $0.5$ and magnitude $\eps_f$ to the deterministic objective function, and the added noise in the gradient is uniformly distributed on the surface of an n-dimensional ball with radius $\eps_g$ around the deterministic gradient. We further perturb the diagonal elements of the Hessian with a Bernoulli-distributed noise with parameter $0.5$ and magnitude $\eps_H$.

			\subsection{Termination test}
			We begin by studying the effectiveness of the stopping test in Theorem~\ref{thm.practical_termination_v2}.  We ran the algorithm for 1000 iterations and recorded the first iteration $k=\texttt{ter}$ in which condition \ref{it.stop1_v2} or \ref{it.stop2_v2} in Theorem~\ref{thm.practical_termination_v2} holds.
			The results are given in Table~\ref{tb.stop_test_1}, for $\mu = 10^{-1}$ and noise levels $\eps_f = 10^{-2}$ and $\eps_g = \eps_H = 10^{-1}$.  (We choose the derivative noise $\eps_g$ and $\eps_H$ as $\sqrt{\eps_f}$ to imitate the errors one might obtain if the derivatives are approximated by finite differences using noisy 
			function values.)
			
			First, we note that the algorithm is able to solve all instances.  The termination test is triggered often within a number of iterations that is very close to the iteration count for a noise-less version of the instance (solved to a tight tolerance of $10^{-6}$).
			This indicates that the termination test is capable of detecting when the algorithm enters a regime in which steps start to be dominated by noise.
			%
			Furthermore, by comparing the number of iterations with the number of function evaluations, we see that in almost all instances the first trial step size ($\alpha_k=\alpha_k^{\max}$) is accepted in the first \texttt{ter} iterations.
			
			Comparing the unscaled optimality measures $\|\na\tilde\varphi_0^\mu\|$ and $\|\na\tilde\varphi_{\texttt{ter}}^\mu\|$, we see that the method made significant progress from the starting points until the termination tests were satisfied.  The column $\|\na\tilde\varphi^\mu\|^\texttt{av}$ gives an idea of the best possible outcome if we executed the method beyond the termination test.  Since $\eps_g=10^{-1}$, we expect these values to be around $10^{-1}$ as well.
			
			Considering $\|\na\tilde\varphi_\texttt{ter}^\mu\|_{(\hat G^{\mu}_\texttt{ter})^{^{-1}}}$ and $\|\na\tilde\varphi^\mu\|^{\texttt{av}}_{(\hat G^{\mu})^{^{-1}}}$, which report the scaled optimality measures used in the termination tests, we see that in most instances the termination test was triggered when the optimality measure $\|\na\tilde\varphi_k^\mu\|_{(\hat G^{\mu}_k)^{^{-1}}}$ was within a factor of at most 10 of the best possible outcome.
			The last two columns give the two quantities in the termination test in Theorem~\ref{thm.practical_termination_v2}.  We can see that in all but one instance the strategy described in Remark~\ref{rem:balanceT} was able to balance the two terms, thereby tightening the tolerances and delaying a potentially too early termination.
			Since the derivation of the termination test relies on some conservative inequalities, we cannot expect that it can predict the smallest achievable $\|\na\tilde\varphi^\mu\|^{\texttt{av}}_{(\hat G^{\mu})^{^{-1}}}$ accurately, so in a practical setting, one may want to run the algorithm a few additional iterations after the test has been triggered.
			An experiment with smaller noise levels, reported in Table~\ref{tb.stop_test_3} in \appC,
			shows similar observations.

			Overall, this experiment indicates that the termination test seems to be able to find a compromise between wasting unproductive iterations and taking sufficiently many iterations to reduce the optimality measure close to the best possible value.

			\begin{table}
				\caption[]{
					Performance of the stopping test with $\mu = 10^{-1}$, $\eps_f = 10^{-2}$, $\eps_g = \eps_H = 10^{-1}$. \texttt{det}: number of iterations needed for noiseless instance with tolerance $10^{-6}$; \texttt{\#f}: total number of function evaluations until iteration $k=\texttt{ter}$; 
					${\|\nabla \tilde \varphi_{\texttt{ter}}^\mu\|}^{}_{\hat G^{-1}}={\|\nabla \tilde \varphi_{\texttt{ter}}^\mu\|}^{}_{(\hat G_{\texttt{ter}}^{\mu})^{^{-1}}}$;
					$\|\nabla \tilde \varphi^\mu \|^{\texttt{av}}$: geometric mean of $\|\nabla \tilde \varphi_k^\mu \|$ over the last 10 iterations, (i.e., $k = 991, \dots, 1000$)  of the run;
					${\|\nabla \tilde \varphi\|}^{\texttt{av}}_{\hat G^{-1}} $ is defined analogously for ${\|\nabla \tilde \varphi_{\texttt{ter}}^\mu\|}^{}_{\hat G^{-1}}$.
				}
				\label{tb.stop_test_1}
				\centering
				\texttt{
					\resizebox{1\textwidth}{!}{%
						\begin{tabular}{|p{1.5cm}||r|r||r|r||p{1.5cm}|p{1.5cm}|p{1.6cm}||c|c||p{1.5cm}|p{1.5cm}| }
							\hline
							\rule{0pt}{15pt}
							Problem  &  \multicolumn{1}{c|}{$n$} &  \multicolumn{1}{c||}{det} &  \multicolumn{1}{c|}{ter} & \multicolumn{1}{c||}{\#f } & \multicolumn{1}{c|}{$ \|\nabla \tilde \varphi_0^\mu\|$} & \multicolumn{1}{c|}{$ \|\nabla \tilde  \varphi_{\texttt{ter}}^\mu\|$} &\multicolumn{1}{c||}{
								${\|\nabla \tilde \varphi^\mu\|^{\texttt{av}}}$
							} &   \multicolumn{1}{c|}{ ${\|\nabla \tilde \varphi_{\texttt{ter}}^\mu\|}^{}_{\hat G^{-1}}$
							} &\multicolumn{1}{c||}{ ${{ \|\nabla \tilde \varphi^\mu\|}^{\texttt{av}}_{\hat G^{-1}}} $
							} & \multicolumn{1}{c|}{$T_{1,\texttt{ter}}$} & \multicolumn{1}{c|}{$T_{2,\texttt{ter}}$} 
							\\[5pt]
							\hline
							biggsb1 & 1000 & 7 & 5 & 6 & 3.14e+01 & 1.50e-01 & 1.42e-01 & 8.83e-02 & 8.22e-02 & 4.05e-01 & 4.05e-01 \\
							chenhark & 1000 & 16 & 19 & 20 & 1.57e+01 & 6.40e-01 & 1.59e-01 & 1.68e+00 & 2.42e-01 & 1.81e+00 & 1.81e+00 \\
							cvxbqp1 & 100 & 8 & 6 & 7 & 2.57e+03 & 2.43e+00 & 1.45e-01 & 2.12e-02 & 3.14e-03 & 3.04e-01 & 3.04e-01 \\
							eg1 & 3 & 7 & 5 & 7 & 2.26e+00 & 5.55e-01 & 1.16e-01 & 2.32e-01 & 4.68e-02 & 3.19e-01 & 3.77e-01 \\
							eigena & 110 & 21 & 12 & 14 & 7.28e+01 & 7.31e+00 & 2.64e-01 & 7.47e-01 & 9.83e-02 & 1.12e+00 & 1.12e+00 \\
							explin & 120 & 18 & 16 & 17 & 7.65e+03 & 4.05e-01 & 1.40e-01 & 2.76e-02 & 4.57e-03 & 3.08e-01 & 3.08e-01 \\
							explin2 & 120 & 18 & 18 & 19 & 7.65e+03 & 1.44e-01 & 1.40e-01 & 3.80e-03 & 3.13e-03 & 3.02e-01 & 3.02e-01 \\
							expquad & 120 & 17 & 16 & 17 & 7.64e+03 & 3.50e+00 & 1.40e-01 & 3.05e-01 & 4.72e-02 & 3.23e-01 & 3.23e-01 \\
							harkerp2 & 100 & 10 & 23 & 24 & 9.36e+06 & 4.32e+01 & 1.45e-01 & 2.57e-01 & 1.14e-02 & 3.77e-01 & 3.77e-01 \\
							mccormck & 1000 & 14 & 7 & 8 & 9.38e+01 & 2.88e-01 & 1.42e-01 & 1.24e-01 & 5.87e-02 & 3.54e-01 & 3.54e-01 \\
							mdhole & 2 & 33 & 15 & 23 & 2.75e+03 & 8.59e-01 & 9.53e-02 & 5.42e-01 & 2.13e-02 & 6.43e-01 & 6.43e-01 \\
							ncvxbqp1 & 1000 & 52 & 55 & 57 & 8.15e+01 & 3.35e+00 & 1.43e-01 & 2.49e-01 & 4.91e-02 & 2.13e+00 & 2.13e+00 \\
							ncvxbqp2 & 1000 & 108 & 103 & 104 & 6.01e+01 & 3.34e-01 & 1.48e-01 & 3.26e-01 & 7.68e-02 & 1.09e+00 & 1.09e+00 \\
							ncvxbqp3 & 1000 & 151 & 146 & 148 & 4.83e+01 & 1.31e+01 & 1.40e-01 & 8.83e-01 & 6.84e-02 & 2.49e+00 & 2.49e+00 \\
							nonscomp & 1000 & 25 & 7 & 8 & 7.59e+03 & 5.65e-01 & 1.85e-01 & 8.77e-02 & 3.07e-02 & 7.49e-01 & 7.49e-01 \\
							obstclal & 64 & 7 & 4 & 5 & 7.53e+00 & 1.99e-01 & 1.40e-01 & 1.58e-01 & 8.76e-02 & 4.67e-01 & 4.67e-01 \\
							obstclbl & 64 & 6 & 4 & 5 & 1.95e+02 & 5.14e-01 & 1.44e-01 & 5.90e-02 & 3.60e-02 & 3.55e-01 & 3.55e-01 \\
							obstclbu & 64 & 6 & 4 & 5 & 1.91e+02 & 3.79e-01 & 1.44e-01 & 7.10e-02 & 3.60e-02 & 3.55e-01 & 3.55e-01 \\
							pentdi & 1000 & 6 & 3 & 4 & 1.72e+01 & 1.14e+00 & 1.42e-01 & 1.58e-01 & 3.06e-02 & 3.16e-01 & 3.16e-01 \\
							qrtquad & 120 & 21 & 20 & 23 & 7.54e+03 & 1.35e+00 & 1.39e-01 & 7.40e-02 & 4.70e-02 & 3.23e-01 & 3.23e-01 \\
							qudlin & 12 & 26 & 14 & 15 & 2.58e+02 & 3.31e+00 & 2.11e-01 & 2.64e-01 & 1.15e-01 & 7.40e-01 & 7.40e-01 \\
							sim2bqp & 2 & 7 & 4 & 5 & 4.02e+01 & 3.37e-01 & 6.08e-02 & 8.65e-02 & 2.00e-02 & 3.40e-01 & 3.40e-01 \\
							\hline
				\end{tabular}}}
			\end{table}

			\subsection{Identification of active set}
			
			We explore whether the algorithm can identify the active constraints, as postulated in Remark~\ref{rem:activeset}.  
			The analysis in Section~\ref{sec.local_conv} was done for the primal barrier method. The experiments reported here used the same primal-dual interior-point method as in the previous section.  We will see that the active-set identification is also possible with that variant.

			We start with a small instance of \texttt{harkerp2},
			\begin{equation}
				\label{eq.harkerp2}
				\begin{split}
					\min \ \   -\sum_{i=1}^n \Big(\tfrac{1}{2}x_i^2 + x_i \Big) + \Big( \sum_{i=1}^n x_i\Big)^2 + 2\sum_{j=2}^n \Big(\sum_{i = j}^n  x_i\Big)^2 \quad 
					\st \ \   x_i \geq 0, \ \ \forall i \in [n],    
				\end{split}
			\end{equation}
			with $n = 4$.
			The primal-dual optimal solution of this problem is $x_*^0 = (1, 0, 0, 0)^T$ and $z_*^0=(0, 1, 1, 1)^T$, i.e., all but the first bound constraints are active at the solution and strict complementarity holds.

			Figure~\ref{fig.slack_harkerp2} plots $\log_{10}|x_{k, i} - x_{*, i}^0|$ as a function of $k$ when the algorithm is started from $x_0 = (1, 2, 3, 4)^T$ with $\eps_f = 10^{-2}$ and $\eps_g = \eps_H = 10^{-1}$. The barrier parameter is decreased by a factor of 10 every 30 iterations. 
			Observe that, as $\mu$ decreases, $x_{k,2},x_{k,3},x_{k,4}$ approach their optimal value zero without much erratic behavior, whereas $x_{k,1}$, for which its optimal value is not at the boundary, keeps moving around $x_{*,1}=1$ significantly, no matter how small $\mu$ is.
			This experiment provides evidence that the theoretical predictions of Remark~\ref{rem:activeset} are observed in practice.
			\begin{figure}
				\centering
				\includegraphics[width=\textwidth,clip=true,trim=115 10 100 10]{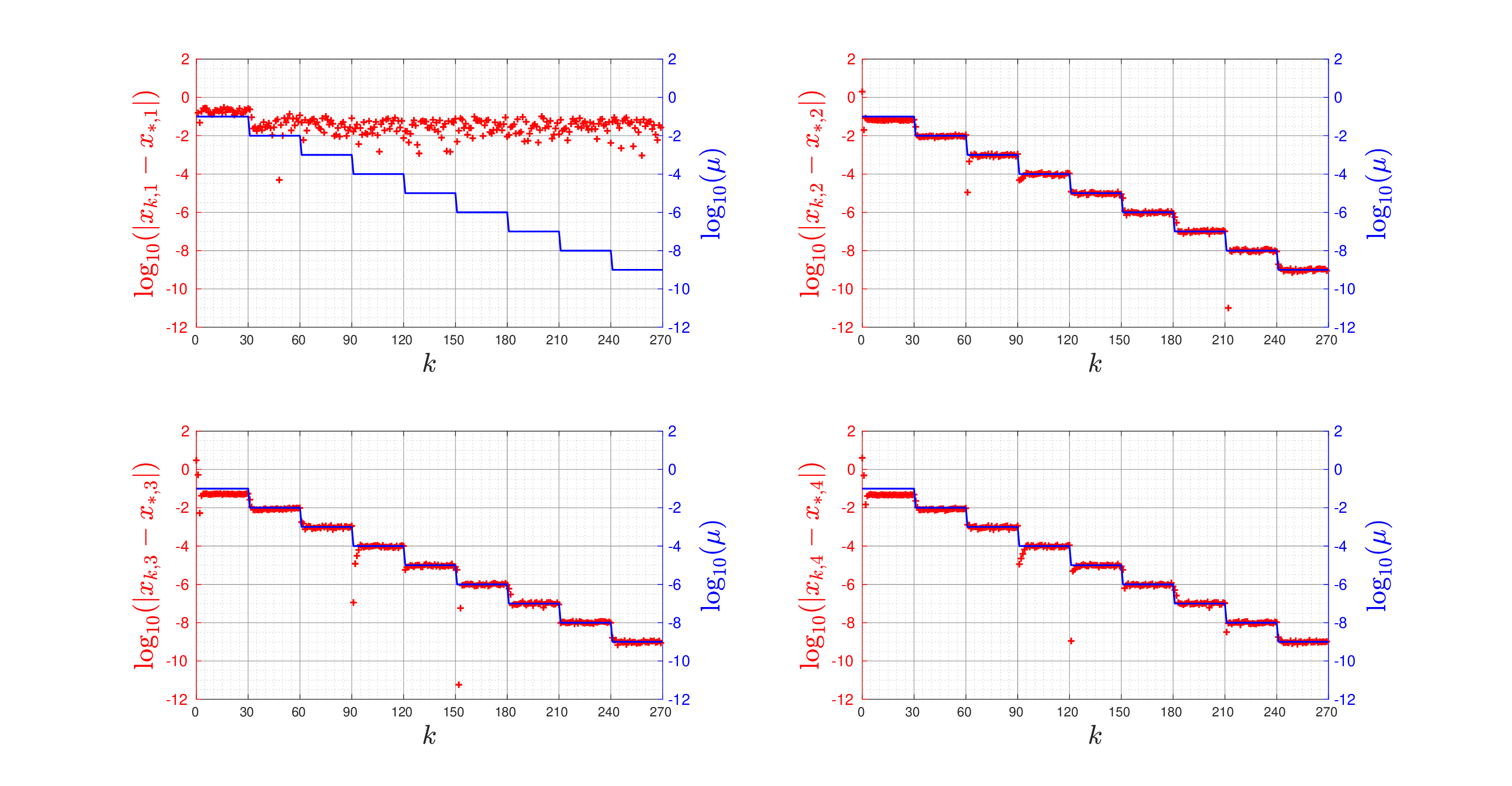} 
				\caption{\texttt{harkerp2}. $\log_{10}|x_i - x_{*, i}^0|$ for $i \in [4]$.  }
				\label{fig.slack_harkerp2}
			\end{figure}

			The final experiment examines this phenomenon for all bound-constrained examples with strictly complementary solutions.
			For each instance, we identified the set of active constraints for which strict complementarity holds, i.e., $\AAA_s=\{i\in\AAA:z^0_{*,i}>0\}$ (for details see Table~\ref{tb.indices} in \appA).  The data for Table~\ref{tb.active_test_ndeg} was generated from the last 10 iterates after running the algorithm for 5000 iterations, which is long enough to reach asymptotic behavior even for very small values of $\mu$.  
			As we proved in Section~\ref{sec.local_conv}, $\|x^\mu_{k,\AAA_s}\|$ converges to zero as $\mu\to0$, and it does so at the same rate as $\mu$, which is expected due to $z^\mu_{*,i}x^\mu_{*,i}=\mu$.

			\subsection{Update of the barrier parameter}
			\label{sec:barrier_update}
			In order to solve the original problem \eqref{eq.orig_bnd_const_prob}, we are interested in obtaining a good solution of the barrier problem (in the sense of Theorem~\ref{thm.practical_termination_v2}) for a target value $\bar\mu>0$ that is small enough so that the barrier term does not push the iterates too far from the boundary of the feasible region.  
			If one is interested in identifying which variables are at their bound, a very small value of $\bar\mu$ should be chosen, as discussed in the previous section. 
			On the other hand, if the active constraints are not of interest, it is not clear that decreasing $\mu_\ell$ below a moderate value related to the noise level is advantageous since the non-active variables remain affected by the noise and dominate the effect on the objective function.
			
			In practice, starting the optimization with the final target value, e.g., $\bar\mu=10^{-7}$, may lead to an extremely large number of iterations before convergence, a phenomenon that we also observed in our experiments in the noisy setting.
			In the non-noisy setting, a practical interior-point method starts with a large initial value of $\mu_0$ and solves a sequence of barrier problems approximately, for decreasing values of $\mu_\ell$.  The approximate solution of the barrier problem for $\mu_\ell$ then provides a good starting point for solving the next barrier problem with $\mu_{\ell+1}$. 
			
			In the non-noisy setting, $\mu_\ell$ is decreased whenever an optimality measure is below a fraction of $\mu_\ell$.  
			However, this approach cannot be applied here because the best-achievable tolerance, the right-hand-side of \eqref{eq:defCCC}, does not decrease to zero as $\mu_\ell\to0$. 
			
			In experiments that are reported in Appendix~\ref{sec.barrier_rule},
			we explored a heuristic that decreases the barrier parameter
			whenever the stopping test in Remark~\ref{rem:balanceT} holds and $\|X_k^{\mu_\ell} z_k^{\mu_\ell} - \mu_\ell e\|_\infty\leq\kappa_\mu\mu_{\ell}$ for some $\kappa_\mu\in(0,1)$, or a fixed number $N_\mu$ of iterations for $\mu_\ell$ are exceeded.  With this procedure, solutions of \eqref{eq.barrier_bnd_const} with $\bar\mu=10^{-7}$ could be computed in a reasonable number of iterations.

			\begin{table}
				\caption{Behavior of the non-degenerate active variable iterates
					$x^\mu_{k,\AAA_s}$
					over the last 10 iterations. (* Converged to a different local solution.)}
				\label{tb.active_test_ndeg}
				\centering
				\begin{threeparttable}
					\texttt{
						\resizebox{0.7\textwidth}{!}
						{%
							\begin{tabular}{|l||c|c|c|c|}
								\hline
								\rule{0pt}{10pt}
								\multirow{3}{*}{\rm Problem} & \multicolumn{4}{{p{9cm}|}}
								{\rm max.\ of $\|x^\mu_{k,\AAA_s}\|_{\infty}$ in last 10 iterations.}
								\\ [3pt] \cline{2-5} 
								\rule{0pt}{10pt}
								&\multicolumn{2}{{p{4.5cm}|}}{$\eps_f = 10^{-2}, \eps_g = \eps_H = 10^{-1}$} & \multicolumn{2}{{p{4.5cm}|}}{$\eps_f = 10^{-6}, \eps_g = \eps_H = 10^{-3}$}                                                    \\ [3pt]  \cline{2-5} 
								\rule{0pt}{10pt}
								& \multicolumn{1}{{p{1.9cm}|}}{ $\mu= 10^{-4}$}  & \multicolumn{1}{{p{1.9cm}|}}{$\mu= 10^{-8}$}   & \multicolumn{1}{{p{1.9cm}|}}{$\mu= 10^{-4}$}                         & \multicolumn{1}{{p{1.9cm}|}}{$\mu= 10^{-8}$}   \\ [2pt]  \hline
								chenhark & 1.05e-04 & 1.01e-08 & 1.05e-04 & 1.00e-08 \\
								cvxbqp1 & 8.45e-05 & 8.45e-09 & 8.33e-05 & 8.33e-09 \\
								eg1 & 2.76e-04 & 1.46e-08 & 1.99e-04 & 1.28e-08 \\
								explin & 1.30e-05 & 1.30e-09 & 1.30e-05 & 1.30e-09 \\
								explin2 & 1.03e-05 & 1.03e-09 & 1.03e-05 & 1.03e-09 \\
								expquad & 1.00e-05 & 1.00e-09 & 1.00e-05 & 1.00e-09 \\
								harkerp2 & 1.03e-04 & 1.07e-08 & 9.88e-05 & 1.00e-08 \\
								mccormck & 1.07e-04 & 1.07e-08 & 1.07e-04 & 1.07e-08 \\
								mdhole & 1.16e-04 & 1.16e-08 & 1.00e-04 & 1.00e-08 \\
								ncvxbqp1 & 2.22e-03 & 2.20e-07 & 1.97e-03 & 1.97e-07 \\
								ncvxbqp2 & 2.10e-03 & 2.10e-07 & 1.97e-03 & 1.97e-07 \\
								ncvxbqp3 & 5.44e-02 & 3.31e-06 & 2.42e-02 & 2.82e-06 \\
								obstclal & 5.76e-03 & 7.58e-07 & 3.13e-03 & 3.55e-07 \\
								obstclbl & 3.29e-03 & 1.09e-06 & 2.63e-03 & 4.37e-07 \\
								obstclbu & 3.29e-03 & 1.09e-06 & 2.63e-03 & 4.37e-07 \\
								pentdi & 4.36e-04 & 4.00e-08 & 4.26e-04 & 4.00e-08 \\
								qrtquad & 1.01e-05 & 1.01e-09 & 1.01e-05 & \, 1.36e-01$^*$\\
								qudlin & 6.63e-06 & 1.00e-09 & 1.00e-05 & 1.00e-09 \\
								sim2bqp & 1.11e-04 & 1.11e-08 & 1.00e-04 & 1.00e-08 \\
								\hline
							\end{tabular}
						}
					}
				\end{threeparttable}
			\end{table}

			\section{Concluding remarks}
			\label{sec.conclusion}
			The goal of this paper is to provide a first step towards understanding the impact of the barrier term on the convergence properties of an interior-point algorithm for problem instances with bounded noise.  By relaxing only the Armijo line-search condition, we keep the method very similar to the standard interior-point framework used in the non-noisy setting so that existing implementations can be easily adjusted.
			To keep the initial analysis manageable, we restricted our focus to problems that have only bound constraints.  As a next step, we will explore interior-point methods for noisy instances with general inequality constraints.
			We conjecture that active-set identification can also be observed in this context.
			
			We point out that the practical termination test introduced in Section~\ref{sec.stopping_test} is not limited to the interior-point setting.  The methodology of estimating the unknown problem-dependent parameters in a theoretical convergence theorem like Theorem~\ref{thm.neigh} by quantities observed during the run of the algorithm can also be applied to methods with similar formulations of a global convergence result, such as in \cite{oztoprak2023constrained,sun2023trust}.

			\section*{Acknowledgement}
			We thank Jorge Nocedal and Frank~E. Curtis for inspiring conversations that helped us to improve the work presented here.
			
			\bibliographystyle{siamplain}
			\bibliography{references}

			\pagebreak

			\appendix
			\section{Indices of  non-degenerate and degenerate active constraints, inactive constraints, and free variables}
			\label{app.active_idx}%
			In Table~\ref{tb.indices}, we provide the number of non-degenerate and degenerate active constraints, inactive constraints, and free variables, along with their corresponding indices, for the test problems discussed in Section~\ref{sec.numerical_exp}. This is particularly relevant for replicating the results in Table~\ref{tb.active_test_ndeg}. In order to determine which bound constraints are active, we run IPOPT to solve the non-noisy problems and examine the value of the primal and dual variables. If the dual variables are nonzero for the variables that their optimal value is at the boundary of the feasible region, we consider those bound constraints active. 
			
			\begin{landscape}
				\begin{table}
					\caption{Indices and number of non-degenerate, degenerate, inactive, and free variables}
					\label{tb.indices}
					\centering
					\texttt{
						\resizebox{1.2\textwidth}{!}{%
							\begin{tabular}{|p{1.5cm}|l||l|l|l|l|l|l|l|l|}
								\hline
								\rule{0pt}{15pt}
								\multirow{2}{*}{Problem} & \multirow{2}{*}{$n$} & \multicolumn{2}{l|}{non-deg. act. vars}    & \multicolumn{2}{l|}{deg. act. vars}    & \multicolumn{2}{l|}{bnd. \& inact. vars}    & \multicolumn{2}{l|}{free vars}    \\ \cline{3-10} 
								&                   & \multicolumn{1}{l|}{ct} & \multicolumn{1}{c|}{${\AAA_s}$} & \multicolumn{1}{l|}{ct} & \multicolumn{1}{c|}{${\AAA_d}$} & \multicolumn{1}{l|}{ct} &  \multicolumn{1}{c|}{${\III_b}$} & \multicolumn{1}{l|}{ct} & \multicolumn{1}{c|}{${\III_f}$} \\ \hline
								biggsb1 & $1000$ & $0$ &  $\emptyset$ & $999$ & $[999]$  & $0$ & $\emptyset$ & $1$ & $\{1000\}$ \\
								\hline
								chenhark & $1000$ & $300$ &  $\{ 701, \dots, 1000\}$& $200$ &  $ \{501, \dots, 700\}$ &  $500$ & $  [500]$ & $0$ & $\emptyset$ \\
								\hline
								cvxbqp1 & $100$ & $100$ & $\{1, \dots, 100\}$ & $0$ & $\emptyset$ & $0$ & $\emptyset$ &  $0$ & $\emptyset$  \\
								\hline
								eg1 & $3$ & $1$ & $\{3\}$   & $0$ & $\emptyset$ & $1$ & $\{2\}$ & $1$ & $\{1\}$ \\
								\hline
								eigena & $110$ & $0$ & $\emptyset$ & $90$ & $[110]\setminus \III_b $ & $20$ & $\begin{array}{lcl}\{1, 2, 12, 14, 23, 26, 34, 38,  45, 50, 56, \\ 62, 67, 74, 78, 86, 89, 98, 100, 110 \}\end{array}$  & $0$ & $\emptyset$ \\
								\hline
								explin & $120$ & $115$ & $[120]\setminus \III_b$ & $0$ & $\emptyset$  & 5 & $  \{2, 4, 6, 8, 10\}$ & $0$ & $\emptyset$ \\
								\hline
								explin2 & $120$ & $117$ & $ [120]\setminus \III_b$ & $0$ & $\emptyset$ & 3  & $  \{6, 8, 10\}$  & $0$ & $\emptyset$ \\
								\hline
								expquad & $120$ & $10$  & $[10]$ & $0$ & $\emptyset$ &  $0$ & $\emptyset$ & 110 & $ [120]\setminus \AAA_s$ \\
								\hline
								harkerp2 & $100$ & $99$ & $\{2, \dots, 100\}$  & $0$ & $\emptyset$ & 1 & $\{1 \}$ & $0$ & $\emptyset$ \\
								\hline
								mccormck & $1000$ &  $1$  & $\{1000 \}$   & $0$ & $\emptyset$ & 999 & $[999]$ & $0$ & $\emptyset$ \\
								\hline
								mdhole & $2$ & $1$  & $\{1 \}$ & $0$ & $\emptyset$ & $0$ & $\emptyset$ & 1 & $\{2 \}$  \\
								\hline
								ncvxbqp1 & $1000$ & $1000$  & $[1000]$ & $0$  & $\emptyset$ & $0$ & $\emptyset$ & $0$ & $\emptyset$ \\
								\hline
								ncvxbqp2 & $1000$ & 992 & $[1000] \setminus \III_b$ & $0$  & $\emptyset$ & $8$ & $\{146, 248, 266, 302, 338, 428, 446, 464\}$ & $0$  & $\emptyset$\\
								\hline
								ncvxbqp3 & $1000$ & $987$ & $[1000] \setminus \III_b$  & $0$ & $\emptyset$ &  13 & $\begin{array}{lcl}\{{3}, 352, 388, 406, 424, 442, 460,\\ 478, 496, 536, 548, 722, 740, {746}\}\end{array}$   & $0$ & $\emptyset$ \\
								\hline
								nonscomp & $1000$ &  $0$ & $\emptyset$  & $316$ & $\{3, 6, 9, \dots, 948 \}$  & 684 & $ [1000] \setminus \AAA_d$ & $0$ & $\emptyset$\\
								\hline
								obstclal & $64$ & $27$ & $\begin{array}{lcl}\{ 11, \dots, 14, 18, \dots, 23, 26 \dots 31,\\ 34, \dots, 39, 42, \dots, 46 \} \end{array} $& $0$ & $\emptyset$  & $37$ & $[64] \setminus \AAA_s$& $0$ & $\emptyset$\\
								\hline
								obstclbl & $64$ & $48$ & $[64] \setminus \III_b$ & $0$ & $\emptyset$& 16&$\begin{array}{lcl}\{ 4,5,12,13,25,26,31,32,33,\\ 34,39,40,52,53,60,61 \} \end{array} $ & $0$ & $\emptyset$ \\
								\hline
								obstclbu & $64$ & $48$ & $[64] \setminus \III_b$ & $0$ & $\emptyset$ & 16&$\begin{array}{lcl}\{ 4,5,12,13,25,26,31,32,33,\\ 34,39,40,52,53,60,61 \} \end{array} $ & $0$ & $\emptyset$ \\
								\hline
								pentdi & $1000$ & $502$  & $\{3, 498, 501, 502, \dots, 1000\}$ & 496  &$[1000] \setminus (\AAA_s \cup \III_b)$ &2 & $\{1, 500 \}$  & $0$ & $\emptyset$\\
								\hline
								qrtquad & $120$ & $5$ &$\{1, 3, 5, 7, 9\} $& $0$ & $\emptyset$ & 115 &$[120] \setminus \AAA_s$ & $0$ & $\emptyset$\\
								\hline
								{qudlin} & $12$ & $10$  & $\{3, \dots, 12\} $ & $0$ & $\emptyset$ & 2 &  $\{1, {2}\} $  & $0$ & $\emptyset$\\
								\hline
								sim2bqp & $2$ & $1$ & $\{2\} $ & $0$ & $\emptyset$ & $0$ & $\emptyset$ &  1 & $\{1\} $ \\
								\hline
					\end{tabular}}}
				\end{table}
			\end{landscape}

			\section{Uniform Lower Bound on the step size in Lemma~\ref{lem:alpha_lb}}
			\label{app.lb_on_alpha}
			\begin{lemma}
				\label{lemma.unif_lb_on_alpha}
				Consider $\phi(\alpha) = a \alpha^2 + b \alpha - c$ in Lemma~\ref{lem:alpha_lb}. 
				The positive root of $\phi(\alpha)$, i.e., $\alpha_k = \tfrac{-b + \sqrt{b^2 + 4 a c}}{2 a}$ is a decreasing function of $a$ and $b$.
			\end{lemma}
			\begin{proof}
				The partial derivative of $\alpha_k$ with respect to $a$ is given by
				\begin{align*}
					\frac{\partial \alpha_k}{\partial a} = \frac{b - \tfrac{2 ac + b^2}{\sqrt{4 a c + b^2}}}{2 a^2 },
				\end{align*} 
				observe that $b - \tfrac{2 ac + b^2}{\sqrt{4 a c + b^2}} \leq 0 $ because $4 a^2 c^2 \geq 0$. Thus, $\alpha_k$ is a decreasing function of $a$. The partial derivative of $\alpha_k$ with respect to $b$ is given by
				\begin{align*}
					\frac{\partial \alpha_k}{\partial b} = \frac{\tfrac{b}{\sqrt{4 a c + b^2}} - 1}{2 a},
				\end{align*}
				observe that $\tfrac{b}{\sqrt{4 a c + b^2}} - 1 \leq 0$ because  $ 4  a c \geq 0$. Thus, $a_k$ is a decreasing function of $b$.
			\end{proof}

			\section{Further numerical results on the barrier subproblem}
			\label{sec.more_num_res}
			
			In this appendix, we provide additional numerical results concerning the effectiveness of the stopping test outlined in Theorem~\ref{thm.practical_termination_v2} with a smaller noise level. We observe a similar algorithmic behavior to that provided in Table~\ref{tb.stop_test_1}.

			\begin{table}[ht]
				\caption[]{Performance of the stopping test with $\mu = 10^{-1}$, $\eps_f = 10^{-6}$, $\eps_g = \eps_H = 10^{-3}$, $\nu = 10^{-6}$.}
				\label{tb.stop_test_3}
				\centering
				\texttt{
					\resizebox{1\textwidth}{!}{%
						\begin{tabular}{|p{1.6cm}||r|r|r|p{1.5cm}|p{1.5cm}|p{1.6cm}|c|c|p{1.5cm}|p{1.5cm}|  }
							\hline
							\rule{0pt}{15pt}
							Problem  &  \multicolumn{1}{c|}{$n$} &  \multicolumn{1}{c|}{ter} & \multicolumn{1}{c|}{\#f } & \multicolumn{1}{c|}{$ \|\nabla \tilde \varphi_0^\mu\|$} & \multicolumn{1}{c|}{$ \|\nabla \tilde  \varphi_{\texttt{ter}}^\mu\|$} &\multicolumn{1}{c|}{
								${\|\nabla \tilde \varphi^\mu\|^{\texttt{av}}}$
							} &   \multicolumn{1}{c|}{ ${\|\nabla \tilde \varphi_{\texttt{ter}}^\mu\|}_{\hat G{^{-1}}}$
							} &\multicolumn{1}{c|}{ ${{ \|\nabla \tilde \varphi^\mu\|}^{\texttt{av}}_{\hat G^{-1}}} $
							} & \multicolumn{1}{c|}{$T_{1,\texttt{ter}}$} & \multicolumn{1}{c|}{$T_{2,\texttt{ter}}$} 
							\\[5pt]
							\hline
							biggsb1 & 1000 & 5 & 6 & 3.14e+01 & 4.27e-03 & 1.42e-03 & 2.46e-03 & 8.20e-04 & 4.02e-03 & 4.02e-03 \\
							chenhark & 1000 & 271 & 287 & 1.57e+01 & 1.90e-02 & 1.05e-02 & 8.67e-01 & 1.85e-01 & 1.26e+00 & 1.26e+00 \\
							cvxbqp1 & 100 & 7 & 8 & 2.57e+03 & 1.50e-03 & 1.45e-03 & 2.25e-05 & 3.13e-05 & 3.04e-03 & 3.04e-03 \\
							eg1 & 3 & 6 & 7 & 2.36e+00 & 1.13e-03 & 1.15e-03 & 6.91e-04 & 4.84e-04 & 3.67e-03 & 3.67e-03 \\
							eigena & 110 & 19 & 21 & 7.28e+01 & 1.21e-01 & 1.40e-03 & 8.77e-03 & 1.18e-03 & 1.07e-02 & 1.07e-02 \\
							explin & 120 & 17 & 18 & 7.65e+03 & 1.41e-03 & 1.40e-03 & 5.82e-05 & 4.56e-05 & 3.08e-03 & 3.08e-03 \\
							explin2 & 120 & 17 & 18 & 7.65e+03 & 1.41e-03 & 1.40e-03 & 2.74e-05 & 3.13e-05 & 3.02e-03 & 3.02e-03 \\
							expquad & 120 & 15 & 16 & 7.64e+03 & 5.30e-02 & 1.40e-03 & 1.17e-03 & 4.71e-04 & 3.23e-03 & 3.23e-03 \\
							harkerp2 & 100 & 10 & 11 & 9.36e+06 & 9.41e-03 & 1.45e-03 & 1.26e-04 & 1.16e-04 & 3.79e-03 & 3.79e-03 \\
							mccormck & 1000 & 7 & 8 & 9.38e+01 & 1.37e-03 & 1.42e-03 & 5.68e-04 & 5.87e-04 & 3.53e-03 & 3.53e-03 \\
							mdhole & 2 & 29 & 49 & 2.75e+03 & 7.39e-04 & 1.48e-03 & 2.78e-04 & 3.65e-04 & 3.33e-03 & 3.33e-03 \\
							ncvxbqp1 & 1000 & 51 & 52 & 8.15e+01 & 3.08e-03 & 1.42e-03 & 1.80e-04 & 1.75e-04 & 9.35e-03 & 9.35e-03 \\
							ncvxbqp2 & 1000 & 107 & 108 & 6.01e+01 & 1.42e-03 & 1.42e-03 & 6.00e-04 & 4.43e-04 & 1.11e-02 & 1.11e-02 \\
							ncvxbqp3 & 1000 & 146 & 147 & 4.83e+01 & 1.75e-02 & 1.39e-03 & 2.16e-03 & 6.59e-04 & 1.54e-02 & 1.54e-02 \\
							nonscomp & 1000 & 7 & 14 & 7.59e+03 & 9.17e+00 & 1.61e-03 & 1.14e+00 & 7.66e-04 & 1.24e+00 & 1.24e+00 \\
							obstclal & 64 & 5 & 6 & 7.54e+00 & 4.87e-03 & 1.42e-03 & 4.04e-03 & 8.43e-04 & 4.65e-03 & 4.65e-03 \\
							obstclbl & 64 & 5 & 6 & 1.95e+02 & 2.09e-03 & 1.43e-03 & 8.25e-04 & 3.60e-04 & 3.55e-03 & 3.55e-03 \\
							obstclbu & 64 & 5 & 6 & 1.91e+02 & 2.19e-03 & 1.43e-03 & 7.31e-04 & 3.60e-04 & 3.55e-03 & 3.55e-03 \\
							pentdi & 1000 & 5 & 6 & 1.72e+01 & 1.45e-03 & 1.42e-03 & 3.15e-04 & 3.06e-04 & 3.17e-03 & 3.17e-03 \\
							qrtquad & 120 & 19 & 22 & 7.54e+03 & 1.84e-03 & 1.39e-03 & 4.99e-04 & 4.70e-04 & 3.23e-03 & 3.23e-03 \\
							qudlin & 12 & 24 & 30 & 2.58e+02 & 1.04e-02 & 1.25e-03 & 7.57e-03 & 8.64e-04 & 8.19e-03 & 8.19e-03 \\
							sim2bqp & 2 & 5 & 6 & 4.03e+01 & 8.57e-03 & 1.30e-03 & 1.73e-03 & 5.11e-04 & 3.39e-03 & 3.39e-03 \\
							\hline
				\end{tabular}}}
			\end{table}
			
			\section{Numerical observation: Impact of gradient noise distribution on non-noisy gradient} 
			\label{app.obs}
			In this appendix, we share an observation related to the non-noisy gradient at the final iterations of the algorithm.  Depending on the specific distribution of the randomly added noise to generate the noisy gradient, we see an interesting pattern in the non-noisy gradient.
			
			We consider the following optimization problem
			\begin{equation}
				\label{eq.obs}
				\min_{x \in \RR^n}\ \tfrac{c_1}{2}\left(x_1 - 1\right)^2 + c_2 \cdot x_2\ \ \st \ \ x \geq 0,
			\end{equation}
			where values of the coefficients $c_1$ and $c_2$ are specified in each example. 
			For all the following experiments, noise in the objective function follows a Bernoulli distribution with parameter of $0.5$ and magnitude of $\eps_f = 10^{-2}$  and the diagonal elements of the Hessian are perturbed with a Bernoulli distributed noise with parameter $0.5$ and magnitude $\eps_H = 10^{-1}$. The distribution of the noise in the gradient is specified in each example respectively, where $\eps_g = 10^{-1}$.  In each example, we solve the barrier subproblem with $\mu = 10^{-8}$ and run the algorithm for $1000$ iterations. The scatter plots of the iterates $x_2$ versus $x_1$, the noisy gradient $\nabla_{x_2} \tilde \varphi$ versus $\nabla_{x_1} \tilde \varphi$, and the deterministic gradient $\nabla_{x_2}  \varphi$ versus $\nabla_{x_1}  \varphi$ are provided for the last $200$ iterations of the run of the algorithm (See Figures~\ref{fig.obs_ex1}, \ref{fig.obs_ex2}, \ref{fig.obs_ex3}, and \ref{fig.obs_ex4}).
			\begin{example} \label{ex.obs1}
				Let  $c_1 = c_2 = 1$ in problem~\eqref{eq.obs}. Noise in the gradient is uniformly distributed on the surface of a 2-dimensional circle with radius $\eps_g$ around the deterministic gradient.
				
				In Figure~\ref{fig.obs_ex1}, we observe that the points $(x_1, x_2)$ are scattered around an ellipse. Since $x_{2,*} = 0$, the ellipse is narrower in the $x_2$ direction. In addition, $(x_1, x_2)$'s have denser concentrations around the circumference of the ellipse. We further observe that $(\nabla_{x_1} \tilde \varphi, \nabla_{x_2} \tilde \varphi)$'s are scattered uniformly in a circle with radius $\approx 10^{-1}$. 
				In addition, the deterministic gradient $(\nabla_{x_1}  \varphi, \nabla_{x_2}  \varphi)$'s are densely concentrated around the circumference of a circle with radius $\approx 10^{-1}$.
				\begin{figure}[H]
					\centering
					\includegraphics[width=\textwidth,clip=true,trim=10 170 100 100]{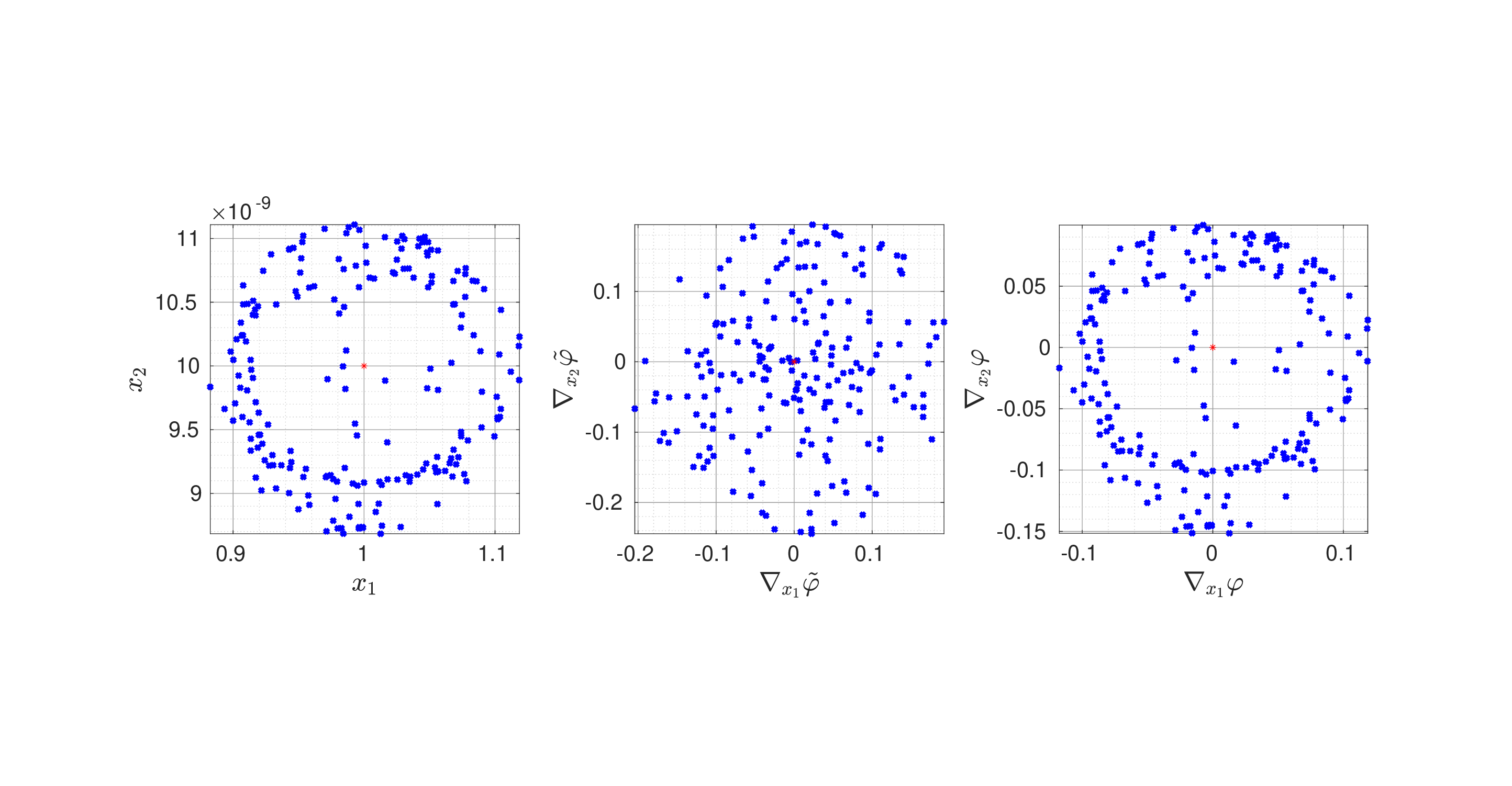}
					\caption{Example~\ref{ex.obs1} observation}
					\label{fig.obs_ex1}
				\end{figure}
			\end{example}
			\begin{example}  \label{ex.obs2}
				Let  $c_1 = c_2 = 1000$ in problem~\eqref{eq.obs}. Noise in the gradient is uniformly distributed on the surface of a 2-dimensional circle with radius $\eps_g$ around the deterministic gradient.

				In Figure~\ref{fig.obs_ex2}, we observe similar phenomena as in Example~\ref{ex.obs1} and Figure~\ref{fig.obs_ex1}. The concentrations of points $(x_1, x_2)$ and $(\nabla_{x_1}  \varphi, \nabla_{x_2}  \varphi)$ around the circumferences of the ellipse and circle are even more pronounced. This is due to the scaling of the problem with $c_1 = c_2 = 1000$. Comparing the range of the axes of the scatter plots of $x_2$ versus $x_1$ in Figures~\ref{fig.obs_ex1} and \ref{fig.obs_ex2} suggests that this scaling results in a more accurate solution.
				\begin{figure}[H]
					\centering
					\includegraphics[width=\textwidth,clip=true,trim=10 170 100 100]{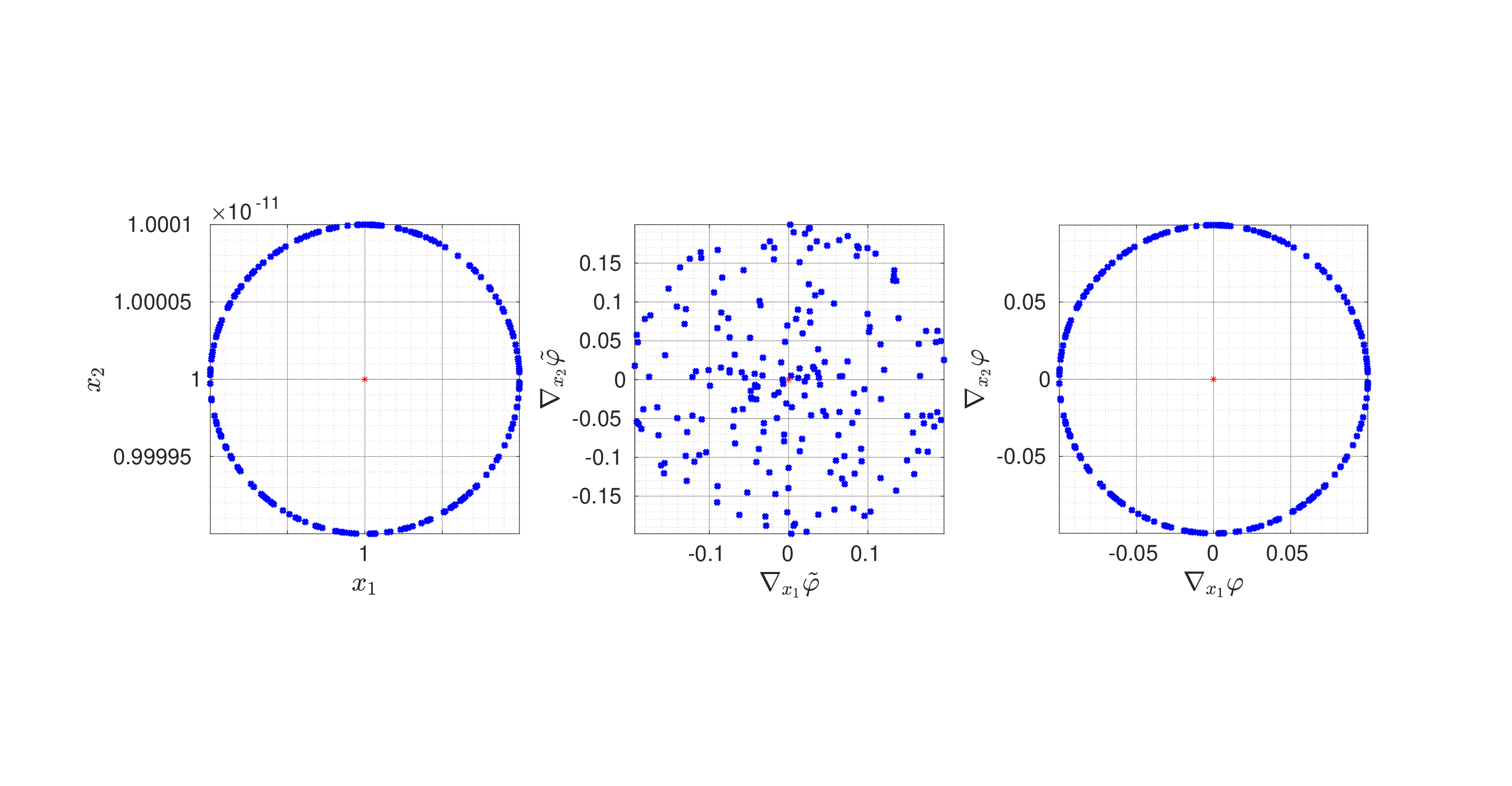}
					\caption{Example~\ref{ex.obs2} observation}
					\label{fig.obs_ex2}
				\end{figure}
			\end{example}
			\begin{example} \label{ex.obs3}
				Let  $c_1 = c_2 = 1$ in problem~\eqref{eq.obs}. Each element of the gradient is perturbed with a uniformly distributed noise with support $[-\eps_g, \eps_g]$.

				In Figure~\ref{fig.obs_ex3}, we observe that the points $(x_1, x_2)$ are scattered in a rectangle and since $x_{2,*} = 0$, the rectangle is narrower in the $x_2$ direction. We further observe that $(\nabla_{x_1} \tilde \varphi, \nabla_{x_2} \tilde \varphi)$ and $(\nabla_{x_1}  \varphi, \nabla_{x_2}  \varphi)$'s are scattered uniformly in a square/rectangle where the length of the edges is close to $0.1$.
				
				\begin{figure}[H]
					\centering
					\includegraphics[width=\textwidth,clip=true,trim=10 170 100 100]{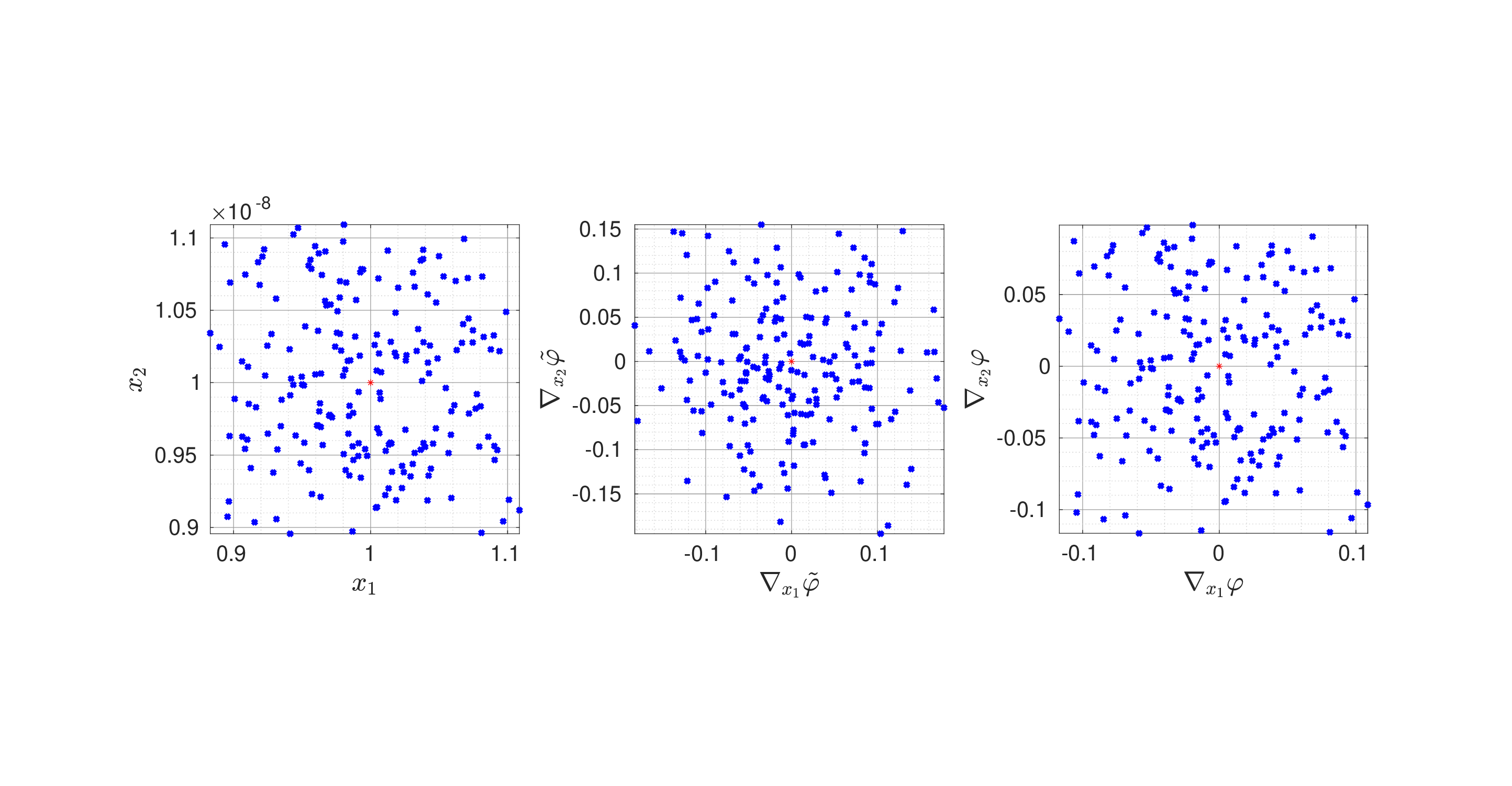}
					\caption{Example~\ref{ex.obs3} observation}
					\label{fig.obs_ex3}
				\end{figure}
			\end{example}
			\begin{example}  \label{ex.obs4}
				Let  $c_1 = c_2 = 1000$ in problem~\eqref{eq.obs}. Each element of the gradient is perturbed with a uniformly distributed noise with support $[-\eps_g, \eps_g]$.
				
				In Figure~\ref{fig.obs_ex4}, we have a similar observation as in Example~\ref{ex.obs3} and Figure~\ref{fig.obs_ex3}. Again, comparing the range of the axes in the scatter plots of $x_2$ versus $x_1$ in Figures~\ref{fig.obs_ex3} and \ref{fig.obs_ex4} suggests that this scaling results in a more accurate solution.
				
				\begin{figure}[H]
					\centering
					\includegraphics[width=\textwidth,clip=true,trim=10 170 100 100]{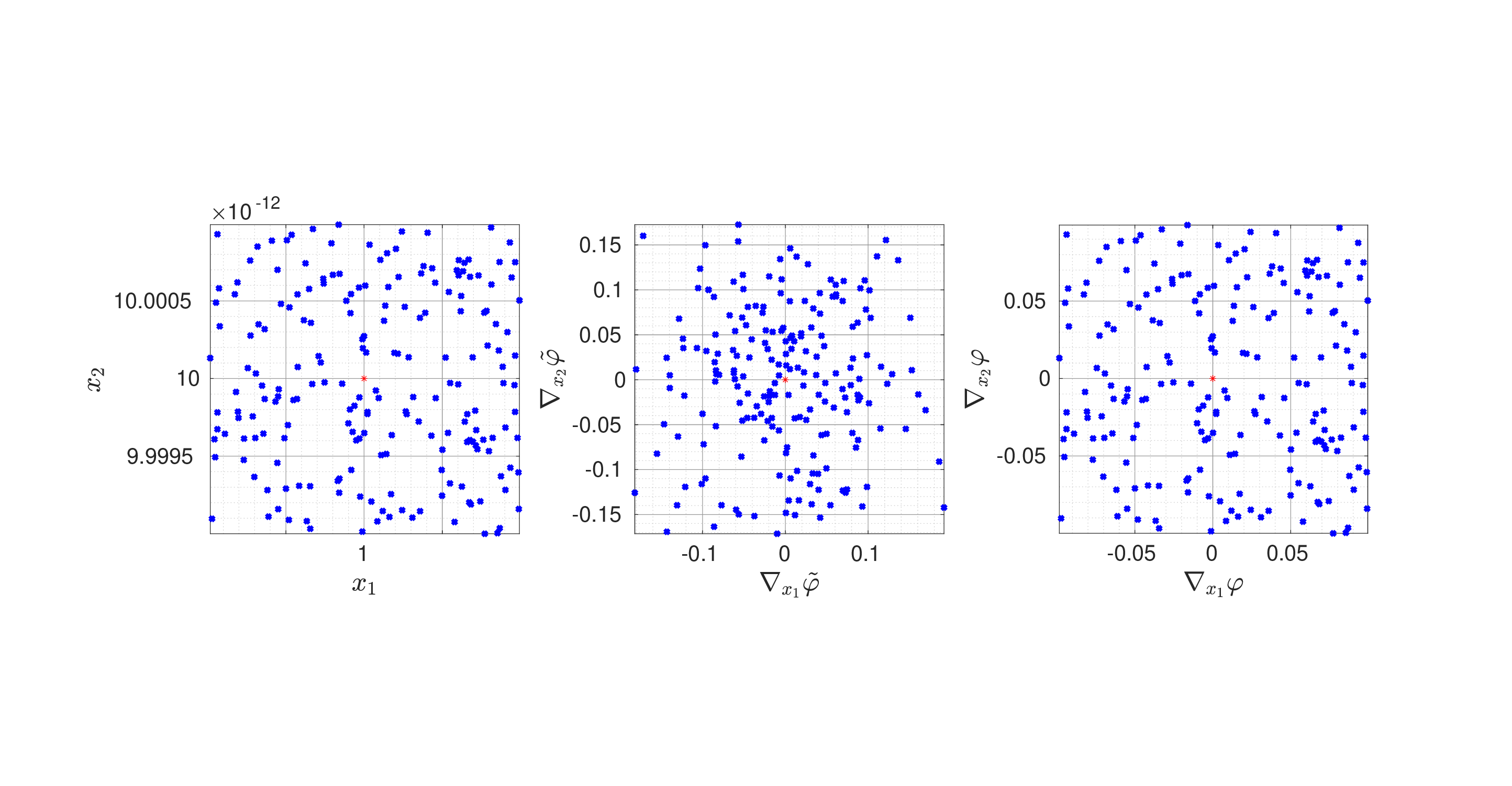}
					\caption{Example~\ref{ex.obs4} observation}
					\label{fig.obs_ex4}
				\end{figure}
			\end{example}
			
			\section{Numerical results on the barrier parameter update rule}
			
			\label{sec.barrier_rule}
			
			In this appendix,  we propose a heuristic for updating the barrier parameter $\mu$ in order to solve problem~\eqref{eq.orig_bnd_const_prob} in the interior-point method framework and illustrate its effectiveness by providing numerical experiments.
			
			As mentioned in Section~\ref{sec.numerical_exp}, we observed that the primal-dual implementation improves the convergence behavior of the algorithm significantly compared to the primal approach. Thus, the implemented algorithm follows the primal-dual approach, as presented in Algorithm~\ref{alg.bnd_const_update_barrier}. This algorithm consists of two loops: the outer \emph{for loop} and the inner \emph{while loop}. The overall algorithm stops when the barrier parameter becomes smaller than a $\mu_{\min}$. Ignoring the stopping condition of the inner while loop, this loop can be viewed as the primal-dual implementation of Algorithm~\ref{alg.bnd_const_fixed_mu} for a fixed $\mu = \mu_\ell$. 
			Hence, all experiments conducted with a fixed $\mu$, as presented in Section~\ref{sec.numerical_exp}, correspond to the inner while loop of Algorithm~\ref{alg.bnd_const_update_barrier}.

			\begin{algorithm}[ht]
				\caption{: Interior-point with relaxed line search for solving problem \eqref{eq.orig_bnd_const_prob} }
				\label{alg.bnd_const_update_barrier}
				\begin{algorithmic}[1]
					\Require $\nu \in (0, \tfrac12)$; $\eps_R > 2 \eps_f$; $\tau_{\min} \in (0, 1)$;  $N_{\mu} = 10$; $\kappa_{\mu} = 10$; $\kappa_{\text{dec}} = 0.1$;  $\kappa_{\Sigma} = 10^4$; $\mu_{\min} \in (0, 1)$; Initial iterate $x_0 \in \RR^n$; Initial dual variables $z_0 \in \RR^{n_u}$, $\mu_0 \in \RR_{>0}$.
					\For{$\ell = 0, 1, \dots$}
					\State \label{st.frac_to_bnd}Set $\tau_\ell \gets \max\{\tau_{\min}, 1 - \mu_\ell \} $.
					\While{\label{st.inner_while}a prescribed stopping condition is not satisfied}
				\State Compute $d_k$  and $d_k^z$  by \eqref{eq.d_x_0} and \eqref{eq.d_z}, respectively.
				\State Compute $\alpha_k^{\max}$ by \eqref{eq.frac_to_the_bnd} and $\alpha_k^{\text{dual}}$ by \eqref{eq.frac_to_the_bnd_z} with $\mu = \mu_\ell$ and $\tau = \tau_\ell$ .
				\State \label{st.alpha_k}Set {$\alpha_k \gets \left(\tfrac{1}{2}\right)^i \alpha_k^{\max} $ ($i$: smallest element of $\NN $ yielding  \eqref{eq.relaxed_armijo_bnd} with $\mu = \mu_\ell$).  
				}
				\State \label{st.x_k_z_k}Set $x_{k + 1}^{\mu_\ell} \gets x_k^{\mu_\ell} +\alpha_k d_k$, and  $z_{k+1}^{\mu_\ell} \gets  z_k^{\mu_\ell} + \alpha_k^{\text{dual}} d^z_k$.
				\State Reset $z_{k+1}^{\mu_\ell}$ according to \eqref{eq.dual_var_proj}.
				\State Set $k \gets k+1$.
				\EndWhile
				\IIf{$\mu_\ell \leq \mu_{\min}$} Stop.
				\EndIIf
				\State \label{st.mu_update}Set $\mu_{l+1} \gets \kappa_{\text{dec}} \mu_\ell$.
				\State \label{st.x_z_reset}Set $x_0^{\mu_{l+1}} \gets x_k^{\mu_\ell}$, and $z_0^{\mu_{l+1}} \gets z_k^{\mu_\ell}$. 
				\State \label{st.k_reset}Set $k \gets 0$.
				\EndFor
			\end{algorithmic}
		\end{algorithm}
		
		Before presenting the results of the numerical experiments, let us briefly review Algorithm~\ref{alg.bnd_const_update_barrier}. At the beginning of each iteration $j$ of the outer for loop, the fraction-to-the-boundary rule parameter $\tau_\ell$  is updated \cite{wachter2006implementation} in line \ref{st.frac_to_bnd}. Then, the while loop condition checks whether the barrier subproblem is solved to a prescribed desired accuracy.   If the condition is met, then the barrier parameter is updated in line \ref{st.mu_update}. In addition, the iterates and iteration counter (for the inner loop) are reset in lines \ref{st.x_z_reset} and \ref{st.k_reset}, respectively. Conversely, if the condition is not met, in iteration $k$ of the inner while loop, the steps $d_k$ and $d_k^z$ are defined as the solutions of the linear system 
		\begin{align}
			\label{eq.d_x_z}
			\begin{bmatrix}
				\tilde H_k^\mu & - I \\
				Z_k^\mu  &  X_k^\mu\\
			\end{bmatrix}\begin{bmatrix}
				d_k \\
				d_k^z 
			\end{bmatrix} = -\begin{bmatrix}
				\tilde g^\mu_k - z_k^\mu \\
				X_k^\mu z_k^\mu - \mu  e  \\
			\end{bmatrix}.
		\end{align}
		To compute $d_k$ and $d_k^z$, we multiply the second row of \eqref{eq.d_x_z} by $(X_k^\mu)^{-1}$ and add up both rows to obtain 
		\begin{align}
			\label{eq.d_x_0}
			\left(\tilde H_k^\mu + (X_k^\mu)^{-1} Z_k^\mu\right) d_k = -(\tilde g^\mu_k - \mu (X_k^\mu)^{-1}\mathbf{e}).
		\end{align}
		As we already discussed in Section~\ref{sec.numerical_exp}, we ensure that the $\tilde H_k^{\mu} + (X_k^\mu)^{-1} Z_k^\mu$ is positive definite by adding a multiple of the identity matrix to it before calculating the step $d_k$. 
		After finding $d_k$ by solving \eqref{eq.d_x_0}, we proceed to compute $d_k^z$ by 
		\begin{align}
			\label{eq.d_z}
			d_k^z = - (X_k^\mu)^{-1} (Z_k^\mu d_k - \mu e) - z_k^\mu.
		\end{align}
		Next, the fraction-to-the-boundary rule determines the upper bound on the step sizes in $d_k$ and $d_k^z$ directions. The upper bound on the primal step size, $\alpha_k^{\max}$, is defined similar to \eqref{eq.frac_to_the_bnd} with $\tau = \tau_\ell$. As for the upper bound on the dual step size, we define $ \alpha_k^{\text{dual}}$ as follows
		\begin{align}
			\alpha_k^{\text{dual}} := \max \left\{ \alpha \in (0, 1]:  z_k^\mu + \alpha d_k^z \geq (1 - \tau_\ell) z_k^\mu \right\}  \label{eq.frac_to_the_bnd_z}
		\end{align}
		with $\tau_l \in (0, 1)$. Subsequently, $\alpha_k$ is determined in line~\ref{st.alpha_k} and the primal and dual iterates are updated in line~\ref{st.x_k_z_k}. 
		Following \cite{wachter2006implementation}, to prevent the primal-dual Hessian $ \tilde H_k^{\mu} + (X_k^\mu)^{-1}Z_k^\mu$ from deviating significantly from the primal Hessian $\tilde G_k^\mu = \tilde H_k^\mu + \mu (X_k^\mu)^{-2}$,  the dual variables are projected as 
		\begin{align}
			\label{eq.dual_var_proj}
			z_{k+1, i}^\mu \gets \max \left\{\min \left \{ z_{k+1, i}^\mu, \frac{\kappa_{\Sigma} \mu_\ell}{x_{k+1, i}^\mu} \right\}, \frac{\mu_\ell}{\kappa_{\Sigma} x_{k+1, i}^\mu}\right\}, \ \ \text{for} \ \ i \in [n],
		\end{align}
		where $\kappa_{\Sigma} \geq 1$.

		In the remainder of this appendix, we present numerical results regarding the performance of the overall Algorithm~\ref{alg.bnd_const_update_barrier} for solving problem~\eqref{eq.orig_bnd_const_prob} and the importance of the heuristic for decreasing the barrier parameter.
		
		As discussed in Section~\ref{rem:balanceT}, for good practical performance it is important to start the algorithm with a relatively large value of $\mu_0$ (such as $0.1$) and only decrease the barrier parameter when the corresponding barrier problem has been solved sufficiently accurately.
		The stopping test in Theorem~\ref{thm.practical_termination_v2} provides an upper bound on the smallest values of the scaled norm of the gradient of the barrier function upon that can be achieved. This condition naturally suggests to reduce the barrier parameter as soon as condition~\ref{it.stop1_v2} or \ref{it.stop2_v2} is satisfied, i.e., whenever $\left\|\nabla \tilde \varphi^{\mu}(x_k^\mu) \right\|_{(\hat G_k^{\mu})^{-1}} \leq \max\{T_{1, k}, T_{2, k} \}$ with $T_{1, k}$ and $T_{2, k} $ defined in \eqref{eq:Tk}. However, the upper bound $\max\{T_{1, k}, T_{2, k} \}$ is based on pessimistic assumptions on the error in the function.
		More importantly, the tolerances do not become tighter as $\mu$ decreases and the termination test might get triggered very quickly after an update of $\mu_\ell$ even though not enough iterations have been taken to solve the new barrier problem well.
		We demonstrate this issue in Figure~\ref{fig.non_deg_stop_test_harkerp2}. We use $j$ to indicate the accumulation of the total number of iterations that Algorithm~\ref{alg.bnd_const_update_barrier} has executed, i.e., the total number of outer and inner loops iterations.
		
		Figure~\ref{fig.non_deg_stop_test_harkerp2} presents the plot of $\log_{10}(\|x_{j, \AAA_s} - x_{*, \AAA_s}\|_{\infty})$ against the iteration number $j$ 
		on the left vertical axis, and the plot of $\log_{10}(\mu)$ against $j$ on the right vertical axis. Using the stopping test described in Theorem~\ref{thm.practical_termination_v2} the barrier parameter is decreased for the first time in iteration~$23$. Subsequently, it is decreased in every iteration until reaching its minimum value in iteration~$30$. 
		Thus, if the termination test in Theorem~\ref{thm.practical_termination_v2} is utilized, the algorithm stops in iteration~$31$, shown by the dashed vertical line, where $\|x_{j, \AAA_s} - x_{*, \AAA_s}\|_{\infty}$ is approximately $10^{-6}$. 
		However, if we continue running the algorithm,  $\|x_{j, \AAA_s} - x_{*, \AAA_s}\|_{\infty}$ decreases to $10^{-8}$, yielding a more accurate solution. Thus, the stopping condition in  Theorem~\ref{thm.practical_termination_v2} is satisfied in every iteration after being satisfied once, leading to early termination of the algorithm.

		\begin{figure}[ht]
			\centering
			\includegraphics[width=.75\textwidth,clip=true,trim=100 5 30 30]{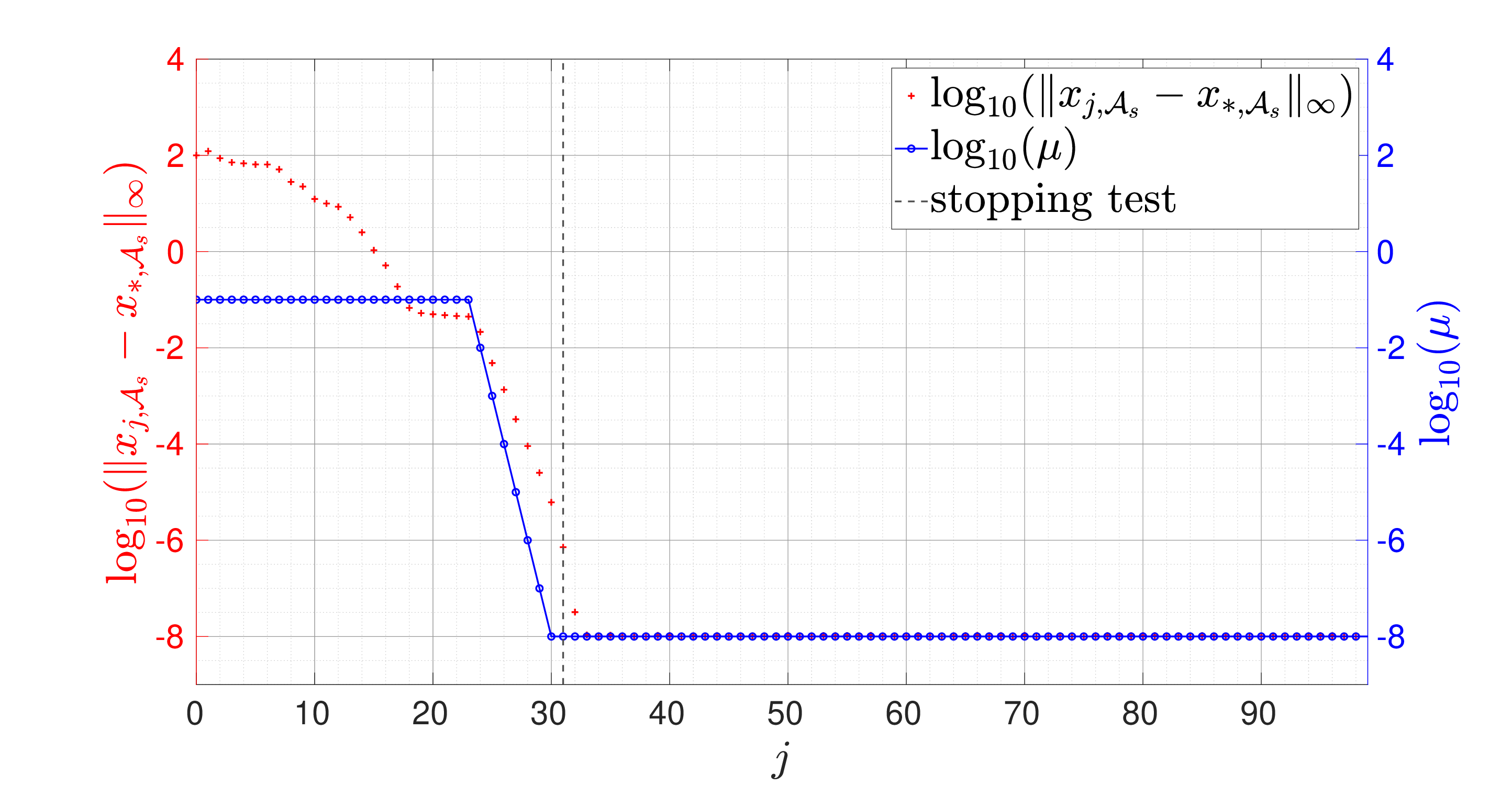}
			\caption{$\log_{10}(\|x_{j, \AAA_s} - x_{*, \AAA_s}\|_{\infty})$ and $\eps_f = 10^{-2}$, $\eps_g = \eps_H = 10^{-1}$ for \texttt{harkerp2} (stopping test)}
			\label{fig.non_deg_stop_test_harkerp2}
		\end{figure}

		To address this issue, we turn to the deterministic setting where the norm of the KKT residuals is used as the criterion for updating the barrier parameter, i.e., $\mu$ is decreased if $$\max\left\{\|\nabla f_k^\mu - z_k^\mu \|_{\infty}, \|X_k^\mu z_k^\mu - \mu e\|_{\infty} \right\}\leq \kappa_{\mu} \mu, $$ for a prescribed constant $\kappa_{\mu} \in \RR_{>0}$ \cite{wachter2006implementation}. Following the deterministic approach and by adopting the primal-dual point of view of interior-point methods,  we propose a heuristic that combines the stopping test described in Theorem~\eqref{thm.practical_termination_v2}
		with the norm of the perturbed complementarity condition, i.e., $\|Xz - \mu e \|_{\infty}$.
		We define the following two conditions, i.e.,  C1 and C2 
		
		\begin{itemize}
			\item [] C1.  $\|\nabla \tilde \varphi_k^{\mu} \|_{(\hat G_k^{\mu})^{-1}} \leq \max \left \{ T_{1, k}, T_{2, k} \right \} + \kappa_{\mu} \mu$,
			\item [] C2. $\|X_k^\mu z_k^\mu - \mu e \|_{\infty} \leq \kappa_{\mu} \mu$,
		\end{itemize}
		and the barrier parameter is decreased if these two conditions are satisfied with $\mu = \mu_\ell$. In other words, these two conditions serve as the stopping condition for the while loop in line~\ref{st.inner_while} of Algorithm~\ref{alg.bnd_const_update_barrier}.
		Observe that since $\nabla f_k^\mu - z_k^\mu \approx \nabla \varphi_k^\mu$, Condition C1 mimics $\|\nabla f(\xk) - z_k^\mu \| \leq \kappa_{\mu} \mu$ in the non-noisy setting, while Condition C2 is the perturbed complementarity condition, similar to the non-noisy setting.
		However, it is not clear that Condition C2 can always be satisfied since the noise might lead to erratic changes in $X_k^\mu z_k^\mu$.
		Hence, after satisfying  Condition C1, we impose a limit $N_{\mu}$ on the number of iterations to attempt to satisfy Condition C2 before proceeding to decrease the $\mu$.
		Using this heuristic, the issue that we observed in Figure~\ref{fig.non_deg_stop_test_harkerp2} is resolved as it is demonstrated in Figure~\ref{fig.non_deg_heuristic_harkerp2}. In comparison to Figure~\ref{fig.non_deg_stop_test_harkerp2}, Figure~\ref{fig.non_deg_heuristic_harkerp2} demonstrates a more gradual decrease in the barrier parameter. Using this heuristic the algorithm stops in iteration~$36$ where $\|x_{j, \AAA_s} - x_{*, \AAA_s}\|_{\infty}$ equals $10^{-8}$. Thus running the algorithm further does not improve the accuracy of the solution.
		\begin{figure}[ht]
			\centering
			\includegraphics[width=.75\textwidth,clip=true,trim=100 5 30 30]{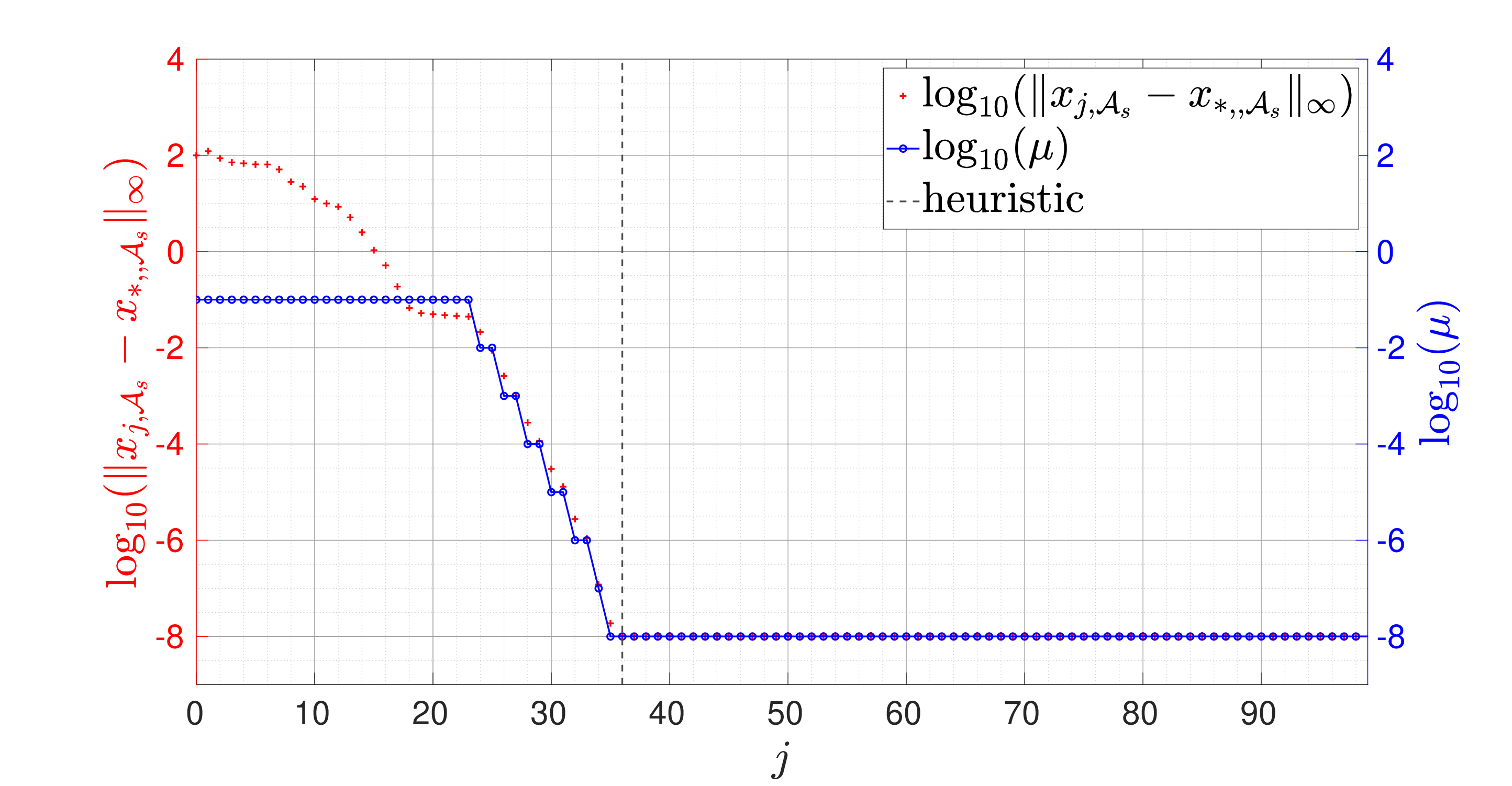}
			\caption{$\log_{10}(\|x_{j, \AAA_s} - x_{*, \AAA_s}\|_{\infty})$ and $\eps_f = 10^{-2}$, $\eps_g = \eps_H = 10^{-1}$ for \texttt{harkerp2} (heuristic)}
			\label{fig.non_deg_heuristic_harkerp2}
		\end{figure}

		Tables~\ref{tb.heuristic} and \ref{tb.heuristic_2} display the performance of different update heuristics for all problems, for two different noise levels.
		The barrier parameter is initialized with $\mu_0=0.1$ and decreased by a factor of 10 until it reaches the final target value $\bar\mu=10^{-7}$.
		\texttt{Heuristic} refers to the method that we outlined, i.e., satisfying conditions C1 and C2 and using $N_{\mu}= 10$
		as a safeguard to avoid getting stuck. (The maximum number of $N_{\mu}$ per barrier problem was never reached in these experiments.) 
		\texttt{Periodic} refers to a method that decreases the barrier parameter every $40$ iterations. The header \texttt{Stopping Test} refers to the stopping condition provided in Theorem~\ref{thm.practical_termination_v2} as the sole criterion for decreasing the barrier parameter (as in Figure~\ref{fig.non_deg_stop_test_harkerp2}).

		We observe that strongly active constraints are identified in all instances. In addition, for most problems, $\|\nabla \varphi_{\texttt{ter}}\|$'s are within a factor of at most 10 of the noise level. 
		For the \texttt{Periodic} method, which takes many more iterations than the other strategies, the values for $\|\nabla \varphi_{\texttt{ter}}\|$ have slightly smaller values compared to those of the other two methods, as one would expect. 
		The strategies \texttt{Heuristic} and \texttt{Stopping Test} perform very similarly in terms of the number of iterations taken and the final accuracy.
		The fact that \texttt{Stopping Test} does as well as \texttt{Heuristic} is an indication that for these instances, the barrier parameter can be reduced very quickly towards the end of the optimization, and the extra safeguard from Condition C2 might not be necessary in these cases.

		For some problems, we observed in Table~\ref{tb.heuristic} large values of the final (unscaled)
		norm of the barrier function gradient, reaching values exceeding $10^{3}$.  We attribute this observation to numerical issues but it deserves further investigation. 
		We also note that $\|\nabla \varphi_{\texttt{ter}}\|$ for \texttt{Periodic} is often equal to the gradient noise level $\epsilon_g$.  This can be explained by the specific way we generated the random perturbations that are added to obtain the noisy gradient, see Figure~\ref{fig.obs_ex2}.

			\begin{table}[ht]
				\caption{Barrier parameter update rule with $\eps_f = 10^{-4}$, and $\eps_g = \eps_H = 10^{-2}$. See Table~\ref{tb.stop_test_1} for the explanation of some of the headers.}
				\label{tb.heuristic}
				\centering
				\texttt{
					\resizebox{.9\textwidth}{!}{%
						\begin{tabular}
							{|c||c| c|c|c| c| c| c|c|c|    }
							\hline
							\rule{0pt}{15pt}
							\multirow{3}{*}{Problem} & \multicolumn{3}{c|}{Heuristic} & %
							\multicolumn{3}{c|}{Periodic} & \multicolumn{3}{c|}{Stopping Test}\\
							[5pt] \cline{2-10}
							\rule{0pt}{15pt}
							& ter & $\|x_{\text{ter}, \AAA_s}\|_{\infty}$ & $\|\nabla \varphi_{\text{ter}}\|$ &  ter & $\|x_{\text{ter}, \AAA_s}\|_{\infty}$ & $\|\nabla \varphi_{\text{ter}}\|$ & ter & $\|x_{\text{ter}, \AAA_s}\|_{\infty}$ & $\|\nabla \varphi_{\text{ter}}\|$  \\
							[5pt] \cline{1-10}
							\hline
							biggsb1 & 19 & NA & 1.26e-02 & 280 & NA & 1.05e-02 & 14 & NA & 1.64e-01 \\
							chenhark & 57 & 1.60e-07 & 8.35e-02 & 280 & 1.12e-07 & 1.16e-02 & 57 & 1.60e-07 & 8.35e-02 \\
							cvxbqp1 & 17 & 8.45e-08 & 1.32e-01 & 280 & 8.33e-08 & 1.00e-02 & 17 & 8.45e-08 & 1.32e-01 \\
							eg1 & 14 & 1.06e-07 & 1.60e-01 & 280 & 1.27e-07 & 1.00e-02 & 14 & 1.06e-07 & 1.60e-01 \\
							eigena & 35 & NA & 4.23e+03 & 280 & NA & 8.12e+03 & 31 & NA & 1.84e-01 \\
							explin & 27 & 1.30e-08 & 1.38e-01 & 280 & 1.30e-08 & 4.03e-02 & 27 & 1.30e-08 & 1.38e-01 \\
							explin2 & 27 & 1.03e-08 & 1.37e-01 & 280 & 1.03e-08 & 4.03e-02 & 27 & 1.03e-08 & 1.37e-01 \\
							expquad & 25 & 1.00e-08 & 4.30e-02 & 280 & 1.00e-08 & 1.00e-02 & 25 & 1.00e-08 & 4.30e-02 \\
							harkerp2 & 28 & 1.15e-07 & 8.72e-01 & 280 & 1.00e-07 & 9.99e-03 & 27 & 1.81e-07 & 3.20e+00 \\
							mccormck & 16 & 1.07e-07 & 1.00e-02 & 280 & 1.07e-07 & 1.00e-02 & 16 & 1.07e-07 & 1.00e-02 \\
							mdhole & 41 & 1.19e-07 & 1.53e-01 & 280 & 1.01e-07 & 1.00e-02 & 41 & 1.19e-07 & 1.53e-01 \\
							ncvxbqp1 & 65 & 1.92e-06 & 1.34e-01 & 280 & 1.98e-06 & 1.00e-02 & 65 & 1.92e-06 & 1.34e-01 \\
							ncvxbqp2 & 114 & 2.08e-06 & 2.42e-01 & 280 & 1.98e-06 & 1.00e-02 & 114 & 2.08e-06 & 2.42e-01 \\
							ncvxbqp3 & 165 & 3.23e-05 & 1.43e-01 & 280 & 2.96e-05 & 1.00e-02 & 163 & 5.75e-05 & 3.95e+03 \\
							nonscomp & 23 & NA & 1.62e-02 & 280 & NA & 1.03e-02 & 20 & NA & 1.22e-01 \\
							obstclal & 20 & 3.49e-06 & 1.07e-02 & 280 & 3.78e-06 & 1.00e-02 & 19 & 4.60e-06 & 1.23e-01 \\
							obstclbl & 19 & 3.90e-06 & 1.27e-01 & 280 & 4.61e-06 & 1.00e-02 & 19 & 3.90e-06 & 1.27e-01 \\
							obstclbu & 19 & 3.90e-06 & 1.27e-01 & 280 & 4.61e-06 & 1.00e-02 & 19 & 3.90e-06 & 1.27e-01 \\
							pentdi & 24 & 4.02e-07 & 1.17e-02 & 280 & 4.02e-07 & 1.02e-02 & 23 & 3.92e-07 & 9.94e-02 \\
							qrtquad & 28 & 1.01e-08 & 2.70e-02 & 280 & 1.01e-08 & 1.00e-02 & 28 & 1.01e-08 & 2.70e-02 \\
							qudlin & 36 & 1.00e-08 & 1.46e+00 & 280 & 1.00e-08 & 5.38e-02 & 29 & 1.00e-08 & 1.26e-01 \\
							sim2bqp & 14 & 7.92e-08 & 2.62e-01 & 280 & 1.01e-07 & 9.98e-03 & 14 & 7.92e-08 & 2.62e-01 \\
							\hline
				\end{tabular}}}
			\end{table}

			\begin{table}[ht]
				\caption{Barrier parameter update rule with $\eps_f = 10^{-6}$, and $\eps_g = \eps_H = 10^{-3}$. See Table~\ref{tb.stop_test_1} for the explanation of some of the headers.}
				\label{tb.heuristic_2}
				\centering
				\texttt{
					\resizebox{.9\textwidth}{!}{%
						\begin{tabular}
							{|c||c| c|c|c| c| c| c|c|c|    }
							\hline
							\rule{0pt}{15pt}
							\multirow{3}{*}{Problem} & \multicolumn{3}{c|}{Heuristic} & %
							\multicolumn{3}{c|}{Periodic} & \multicolumn{3}{c|}{Stopping Test}\\
							[5pt] \cline{2-10}
							\rule{0pt}{15pt}
							& ter & $\|x_{\text{ter}, \AAA_s}\|_{\infty}$ & $\|\nabla \varphi_{\text{ter}}\|$ &  ter & $\|x_{\text{ter}, \AAA_s}\|_{\infty}$ & $\|\nabla \varphi_{\text{ter}}\|$ & ter & $\|x_{\text{ter}, \AAA_s}\|_{\infty}$ & $\|\nabla \varphi_{\text{ter}}\|$  \\
							[5pt] \cline{1-10}
							\hline
							biggsb1 & 22 & NA & 1.23e-03 & 280 & NA & 5.62e-04 & 18 & NA & 2.49e+00 \\
							chenhark & 323 & 1.70e-07 & 1.93e-02 & 280 & 1.15e-07 & 1.67e-03 & 323 & 1.70e-07 & 1.93e-02 \\
							cvxbqp1 & 22 & 8.33e-08 & 1.00e-03 & 280 & 8.33e-08 & 1.00e-03 & 22 & 8.33e-08 & 1.00e-03 \\
							eg1 & 17 & 1.28e-07 & 2.26e-03 & 280 & 1.28e-07 & 1.00e-03 & 17 & 1.28e-07 & 2.26e-03 \\
							eigena & 39 & NA & 1.63e-02 & 280 & NA & 1.00e-03 & 38 & NA & 5.49e-02 \\
							explin & 29 & 1.30e-08 & 3.88e-02 & 280 & 1.30e-08 & 3.88e-02 & 29 & 1.30e-08 & 3.88e-02 \\
							explin2 & 30 & 1.03e-08 & 3.88e-02 & 280 & 1.03e-08 & 3.88e-02 & 30 & 1.03e-08 & 3.88e-02 \\
							expquad & 29 & 1.00e-08 & 4.50e-03 & 280 & 1.00e-08 & 1.00e-03 & 29 & 1.00e-08 & 4.50e-03 \\
							harkerp2 & 33 & 1.00e-07 & 1.09e-03 & 280 & 1.00e-07 & 1.00e-03 & 33 & 1.00e-07 & 1.09e-03 \\
							mccormck & 19 & 1.07e-07 & 1.04e-03 & 280 & 1.07e-07 & 1.00e-03 & 19 & 1.07e-07 & 1.04e-03 \\
							mdhole & 40 & 9.83e-08 & 1.65e-02 & 280 & 1.00e-07 & 1.00e-03 & 40 & 9.83e-08 & 1.65e-02 \\
							ncvxbqp1 & 65 & 1.97e-06 & 1.12e-03 & 280 & 1.97e-06 & 1.16e-03 & 65 & 1.97e-06 & 1.12e-03 \\
							ncvxbqp2 & 122 & 1.97e-06 & 1.07e-03 & 280 & 1.97e-06 & 1.06e-03 & 122 & 1.97e-06 & 1.07e-03 \\
							ncvxbqp3 & 163 & 2.78e-05 & 1.03e-03 & 280 & 2.81e-05 & 1.04e-03 & 163 & 2.78e-05 & 1.03e-03 \\
							nonscomp & 26 & NA & 1.26e-02 & 280 & NA & 9.99e-04 & 22 & NA & 7.15e-02 \\
							obstclal & 24 & 3.56e-06 & 1.00e-03 & 280 & 3.56e-06 & 1.00e-03 & 24 & 3.56e-06 & 1.00e-03 \\
							obstclbl & 23 & 4.34e-06 & 9.98e-04 & 280 & 4.36e-06 & 1.00e-03 & 23 & 4.34e-06 & 9.98e-04 \\
							obstclbu & 23 & 4.34e-06 & 9.98e-04 & 280 & 4.36e-06 & 1.00e-03 & 23 & 4.34e-06 & 9.98e-04 \\
							pentdi & 31 & 4.01e-07 & 1.07e-03 & 280 & 4.01e-07 & 9.99e-04 & 31 & 4.01e-07 & 1.07e-03 \\
							qrtquad & 30 & 1.01e-08 & 3.19e-03 & 280 & 1.01e-08 & 1.00e-03 & 30 & 1.01e-08 & 3.19e-03 \\
							qudlin & 44 & 9.99e-09 & 4.84e-02 & 280 & 1.00e-08 & 1.68e-01 & 37 & 9.67e-07 & 1.83e+03 \\
							sim2bqp & 18 & 9.91e-08 & 9.24e-03 & 280 & 1.00e-07 & 1.00e-03 & 18 & 9.91e-08 & 9.24e-03 \\
							\hline
				\end{tabular}}}
			\end{table}

\end{document}